\documentclass[12pt]{amsart}
\usepackage{amssymb}
\usepackage{amscd}
\usepackage{amsfonts}
\usepackage{latexsym}
\usepackage[all]{xy}
\usepackage{verbatim}

\newtheorem{theorem}{Theorem}[section]
\newtheorem{lemma}[theorem]{Lemma}
\newtheorem{proposition}[theorem]{Proposition}
\newtheorem{corollary}[theorem]{Corollary} 
\theoremstyle{definition}  
\newtheorem{definition}[theorem]{Definition}
\newtheorem{example}[theorem]{Example}

\newtheorem{remark}[theorem]{Remark}

\newcommand{\id}{\text{id}} 
 
\newcommand{\Ker}{\text{Ker\,}}

\newcommand{\Fun}{\text{Fun}}
\newcommand{\EqBr}{\text{EqBr}}
\newcommand{\uEqBr}{\underline{\EqBr}}
\newcommand{\End}{\text{End}} 
\renewcommand{\Vec}{\text{Vec}}
\newcommand{\Pic}{\text{BrPic}}
\newcommand{\uPic}{\underline{\text{BrPic}}}
\newcommand{\uuPic}{\underline{\underline{\text{BrPic}}}}
\newcommand{\pic}{\text{Pic}}
\newcommand{\upic}{\underline{\text{Pic}}}

\newcommand{\Hom}{\text{Hom}} 
 
\newcommand{\FPdim}{\text{FPdim}}

\newcommand{\Aut}{\text{Aut}}

\newcommand{\Out}{\text{Out}}

\newcommand{\Rep}{\text{Rep}}

\newcommand{\rev}{\text{rev}}

\newcommand{\BN}{\mathbb{N}}

\newcommand{\op}{\text{op}}

\newcommand{\B}{\mathcal{B}}
\newcommand{\C}{\mathcal{C}}
\newcommand{\D}{\mathcal{D}}
\newcommand{\E}{\mathcal{E}}

\newcommand{\Z}{\mathcal{Z}}

\newcommand{\M}{\mathcal{M}}
\newcommand{\A}{\mathcal{A}}
\newcommand{\N}{\mathcal{N}}
\newcommand{\K}{\mathcal{K}}

\renewcommand{\O}{\mathcal{O}}

\newcommand{\be}{\mathbf{1}}

\renewcommand{\be}{\mathbf{1}}

\newcommand{\bt}{\boxtimes}
\newcommand{\ot}{\otimes}
\newcommand{\q}{\mathbb{Q}^{\times}_{>0}/(\mathbb{Q}^{\times}_{>0})^2}
\newcommand{\BZ}{\mathbb{Z}}

\newcommand{\Ga}{\Gamma}
\newcommand{\tensor}{\otimes}
\newcommand{\del}{\partial}
\newcommand{\Aaut}{{\rm Aut}_{\tensor}(\D)}
\newcommand{\Oout}{{\rm Out}_{\tensor}(\D)}
\newcommand{\mapsonto}{\twoheadrightarrow}
\newcommand{\rtensor}{\boxtimes}
\newtheorem{fact}[theorem]{Fact}
\newcommand{\italic}{\textit}
\newcommand{\converges}{\Rightarrow}

\begin{document}
\title{Fusion categories and homotopy theory}

\author{Pavel Etingof}
\address{P.E.: Department of Mathematics, Massachusetts Institute of Technology,
Cambridge, MA 02139, USA}
\email{etingof@math.mit.edu}

\author{Dmitri Nikshych}
\address{D.N.: Department of Mathematics and Statistics,
University of New Hampshire,  Durham, NH 03824, USA}
\email{nikshych@math.unh.edu}

\author{Victor Ostrik}
\address{V.O.: Department of Mathematics,
University of Oregon, Eugene, OR 97403, USA}
\email{vostrik@math.uoregon.edu}

\date{\today}
\begin{abstract}
We apply the yoga of classical homotopy theory to classification problems 
of $G$-extensions of fusion and braided fusion categories, 
where $G$ is a finite group. Namely, we reduce such problems to classification 
(up to homotopy) of maps from $BG$ to classifiying spaces of certain higher groupoids. 
In particular, to every fusion category $\C$ we attach the 3-groupoid $\uuPic(C)$ 
of invertible $\C$-bimodule categories, called the Brauer-Picard groupoid of $\C$, 
such that equivalence classes of $G$-extensions of $\C$ are in bijection with homotopy
classes of maps from $BG$ to the classifying space of
$\uuPic(\C)$. This gives rise to an explicit
description of both the obstructions to existence of extensions 
and the data parametrizing them; we work these out both topologically and algebraically.  

One of the central results of the paper is that the 2-truncation of $\uuPic(\C)$  
is canonically the 2-groupoid of braided autoequivalences of the Drinfeld center $\Z(\C)$ of $\C$. 
In particular, this implies that the Brauer-Picard group ${\rm BrPic}(\C)$ (i.e., the group of equivalence classes of 
invertible $\C$-bimodule categories) is naturally isomorphic to the group of braided autoequivalences of $\Z(\C)$.
Thus, if $\C=\Vec_A$, where $A$ is a finite abelian group, then ${\rm BrPic}(\C)$ is the orthogonal group $O(A\oplus A^*)$. 
This allows one to obtain a rather explicit classification 
of extensions in this case; in particular, in the case $G=\Bbb Z_2$, we rederive 
(without computations) the classical result of Tambara and Yamagami. Moreover, we explicitly describe
the category of all $(\Vec_{A_1},\Vec_{A_2})$-bimodule categories (not necessarily invertible ones) by showing that 
it is equivalent to the hyperbolic part of the category of Lagrangian correspondences. 
\end{abstract} 
\maketitle  

\setcounter{tocdepth}{1}

\newpage
\tableofcontents
\newpage


\section{Introduction}

Fusion categories (introduced in \cite{ENO1}) form a class of relatively simple
tensor categories. It would be very interesting to give a classification of fusion
categories but this seems to be out of reach at the moment. A more feasible task
is to come up with some new examples and constructions of such categories. 
In this paper we are making a step in this direction. Namely, for a finite group $G$ 
there is a natural notion of {\em $G-$graded} fusion category, see \S \ref{graded TC} 
below\footnote{We note that one can find in the literature a different (but related) notion
of {\em graded monoidal category}, see \cite{FW,CGO}.}.  
The {\em trivial component} of a $G-$graded fusion category is itself a smaller fusion category
and we say that a $G-$graded fusion category is {\em $G-$extension} of its trivial component.
The goal of this paper is to apply classical homotopy theory to classify 
$G-$extensions of a given fusion category.

To do so, we introduce the {\it Brauer-Picard groupoid of fusion
categories} $\uuPic$. 
By definition, this is a 3-groupoid, whose objects are
fusion categories, 1-morphisms from $\C$ to $\D$ 
are invertible $(\C,\D)$-bimodule categories, 2-morphisms are equivalences
of such bimodule categories, and 3-morphisms are isomorphisms of such
equivalences. This \linebreak 3-groupoid can be truncated in the usual way 
to a 2-groupoid $\uPic$ and further to a 1-groupoid 
(i.e., an ordinary groupoid) $\Pic$; 
the group of automorphisms of $\C$ in this groupoid 
is the {\it Brauer-Picard group} $\Pic(\C)$ of $\C$, which 
is the group of equivalence classes of invertible $\C$-bimodule
categories. 

We also define the 2-groupoid $\uEqBr$, whose objects are 
braided fusion categories, 1-morphisms are braided equivalences,
and 2-morphisms are isomorphisms of such equivalences. 
It can be truncated in the usual way to an ordinary groupoid 
$\EqBr$; the group of automorphisms of a braided fusion category
$\B$ in this groupoid is the group $\EqBr(\B)$ of isomorphism classes of 
braided autoequivalences
of $\B$. 

Let $\C$ and $\D$ be fusion categories. 
Any invertible $(\C,\D)$-bimodule 
category $\M$ naturally gives rise to a Morita equivalence
between $\C$ and $\D$. Hence, by the result of M\"uger 
\cite{Mu} it defines a braided 
equivalence of the Drinfeld centers 
$\Z(\M): \Z(\C)\to \Z(\D)$. 
This implies that the operation $\Z$
of taking the Drinfeld center is a 
2-functor $\uPic\to \uEqBr$.   

Our first main result, which is a strengthening of  \cite[Theorem 3.1]{ENO3}, is  

\begin{theorem}\label{breq}
The 2-functor $\Z$ is a fully faithful embedding 
$\uPic\to \uEqBr$.
In particular, for every fusion category 
$\C$ we have a natural group isomorphism 
$\Pic(\C)\cong \EqBr(\Z(\C))$. 
\end{theorem}

This result allows one to calculate
the group $\Pic(\C)$ in the case 
$\C={\rm Vec}_A$, the category of vector
spaces graded by a group $A$. 
In particular, we immediately get the following
corollary of Theorem \ref{breq}: 

\begin{corollary}\label{ortho}
If $A$ is an abelian group and $\C=\Vec_A$ then
$\Pic(\C)=O(A\oplus A^*)$, the split orthogonal group of $A\oplus
A^*$ (i.e. the group of automorphisms of $A\oplus A^*$ preserving 
the hyperbolic quadratic form $q(a,f)=f(a))$. 
\end{corollary}

To apply the above to classifying extensions, we recall that
to the 3-groupoid $\uuPic$ one can attach its classifying
space $B\uuPic$, defined up to homotopy equivalence. This space
falls into connected components, labeled by Morita equivalence
classes of fusion categories. Each connected component
$B\uuPic(\C)$ corresponding to a fusion category $\C$  
is a 3-type, i.e. it has three nontrivial homotopy groups: its fundamental 
group $\pi_1$ is $\Pic(\C)$, 
$\pi_2$ is the group of isomorphism classes of
invertible objects of $\Z(\C)$, 
and $\pi_3={\bold k}^\times$ (the multiplicative group of the
ground field). 

It then follows from general abstract
nonsense that extensions of $\C$ by a group $G$ are parametrized 
by maps of classifying spaces $BG\to B\uuPic(\C)$. Thus, to classify
extensions, one needs to classify the homotopy classes of such
maps, which we proceed to do using the classical obstruction
theory. This leads us to our second main result,
which is the following explicit description 
of extensions of fusion categories, which is similar to the classical description  
of group extensions \cite{EM} (and is made more explicit in the body of the
paper).

\begin{theorem}\label{extcl}
Extensions of a fusion category $\C$ by a finite group $G$ are parametrized by triples 
$(c,M,\alpha)$, where $c: G\to \Pic(\C)$ is 
a group homomorphism, $M$ belongs to a certain torsor $T_c^2$ over 
$H^2(G,\pi_2)$ (where $G$ acts on $\pi_2$ via $c$), 
and $\alpha$ belongs to a certain torsor $T_{c,M}^3$ over $H^3(G,{\bold k}^\times)$. 
Here the data $c,M$ must satisfy the conditions 
that certain obstructions $O_3(c)\in H^3(G,\pi_2)$ and 
$O_4(c,M)\in H^4(G,{\bold k}^\times)$ vanish.   
\end{theorem}

We also give a purely algebraic proof of 
Theorem \ref{extcl}, which does not rely on homotopy theory. 
(This proof spells out the computations 
that on the topological side are hidden in the machinery
of homotopy theory). After this, we proceed to examples and applications. 
In particular, we give a conceptual proof of the
classification of categorifications of Tambara-Yamagami fusion
rings \cite{TY} (the original proof is by a direct computation)\footnote{We 
note that the quest for a computation-free 
derivation of the remarkable result of Tambara and Yamagami 
was one of motivations for this work.}.

At the end of the paper we discuss a number of related topics. 
In particular, we describe explicitly the monoidal $2$-category of all
bimodule categories over $\C=\Vec_A$ (not necessarily invertible
ones). It turns out to be equivalent to a full subcategory of the category of
Lagrangian correspondences for metric groups (abelian groups 
with a nondegenerate quadratic form). 

\subsection{Organization}
Section~\ref{Sect2} contains background material from the theory of fusion categories
and their module categories.  There is new material in Section~\ref{SubSectSO} where 
we give  a definition (due to V.~Drinfeld)  of the special  orthogonal group of a metric group.

The notion of a tensor product of module categories over a fusion category $\C$ plays a central role
in this work. It extends categorically the notion of tensor product of modules over a ring.
In Section~\ref{Sect3} we define, following \cite{Ta},  the tensor product $\M \bt_\C \N$ of a right $\C$-module category $\M$
and a left $\C$-module category $\N$ by a certain universal property. We prove its existence
and give several equivalent characterizations of it useful for practical purposes.  We also
introduce a monoidal $2$-category $\mathbf{Bimodc}(\C)$  of  $\C$-bimodule categories 
and  explicitly describe the product of bimodule categories over the categories of vector spaces
graded  by abelian groups.

In Section~\ref{Sect4} we study bimodule categories invertible under the above tensor product.
We introduce for a fusion category $\C$ its categorical {\em Brauer-Picard} $2$-group 
$\uuPic(\C)$ consisting of invertible $\C$-bimodule catgeories
and for a braided fusion category  $\B$ its categorical {\em Picard} $2$-group
$\underline{\underline{\rm Pic}}(\B)$ consisting of invertible $\B$-module categories. 

Section~\ref{Sect 5} contains the proof of Theorem~\ref{breq} and its generalization
Theorem~\ref{genbreq}. 

In Section~\ref{Sect6} we prove that homogeneous components of a fusion category $\C =\oplus_{g\in G}\,\C_g$
graded by a finite group $G$ (i.e., a $G$-extension) are invertible bimodule categories over the trivial component $\C_e$.

In Section~\ref{classification topological} we show that morphisms from a group $G$ to various
categorical groups attached to a (braided) fusion category $\C$ 
(or, equivalently, maps between the corresponding  classifying spaces)
are in bijection with fundamental tensor category  
constructions involving $G$ and $\C$: extensions, actions, braided $G$-crossed extensions, etc.
Here we also give a topological version of the proof of Theorem~\ref{extcl}.

In Section~\ref{classification proper} we give a detailed algebraic version of the proof of Theorem~\ref{extcl}. 
We give a formula for the associativity constraint obstruction $O_4(c,\, M)$ in terms of the Pontryagin-Whitehead
quadratic function and prove a divisibility Theorem~\ref{divi} for the order of $O_4(c,\, M)$ 
in $H^4(G,\, \mathbf{k}^\times)$. 

In Section~\ref{ex section} we apply our classification of extensions to recover Tambara-Yamagami
categories \cite{TY} as $\mathbb{Z}/2\mathbb{Z}$-extensions of the category $\Vec_A$ of $A$-graded
vector spaces, where $A$ is an abelian group. 

In Section~\ref{Sect10} we explicitly describe tensor products of $(\Vec_A-\Vec_B)$-bimodule
categories, where $A,\, B$ are finite abelian groups. This description
is given in terms of elementary linear algebra and uses the language of Lagrangian correspondences.

Finally, in the appendix, written by Ehud Meir, it is explicitly shown using
the Lyndon-Hochschild-Serre spectral sequence that it the case of
pointed extensions, our classification of extensions reproduces
the usual theory of extensions of groups with 3-cocycles. 

\begin{remark}
1. We emphasize that the homotopy-theoretic approach to monoidal categories in the style of this paper
is not new, and by now is largely a part of folklore. The principal goal of this paper is to use this approach 
to obtain concrete results about classification of fusion categories. 

2. We expect that the results of this paper extend, with appropriate changes, 
to the case of not necessarily semisimple finite tensor categories, using the methods of \cite{EO}.
One of the new features will be that in the non-semisimple case, the groups  
${\rm Pic}(\C),{\rm EqBr}(\B)$ need not be finite groups - they may be affine algebraic groups of positive 
dimension.  
\end{remark}

\subsection{Acknowledgments.} 
We are deeply grateful to V.~Drinfeld for many inspiring
conversations. Without his influence,  
this paper would not have been written. 
In particular, he suggested the main idea - to use homotopy 
theory of classifying spaces to describe extensions of fusion categories.  
We also thank Jacob Lurie for useful discussions
(in particular, for explanations regarding Proposition
\ref{whit}), and Fernando Muro for explanations and references. 
The work of P.E.\ was  partially supported
by the NSF grant DMS-0504847.  The work of D.N.\ was partially supported 
by the NSA grant H98230-07-1-0081 and the NSF grant DMS-0800545.
The work of V.O.\ was  partially supported
by the NSF grant DMS-0602263.


\section{Preliminaries} 
\label{Sect2}

\subsection{General conventions} 

In this paper, we will freely use the basic theory of 
fusion categories and module categories over them. For basics
on these topics, we refer the reader to  
\cite{BK,O1,ENO1, DGNO}. All fusion categories in this paper
will be over an algebraically
closed field ${\bold k}$ of characteristic zero, 
and all module categories will be semisimple left module
categories (unless noted otherwise). 
We will also use the theory of higher categories and especially
higher groupoids, for which we refer the
reader to \cite{Lu}. However, for reader's convenience, we recall
some of the most important definitions and facts 
that are used below. 


\subsection{Categorical $n$-groups}

For an integer $n\ge 1$, a {\em categorical $n$-group} 
is a monoidal $n$-groupoid whose objects are invertible.
In particular, a categorical 0-group is an ordinary group, and 
a categorical 1-group (or simply a categorical group) is also called 
a {\em gr-category} (if the corresponding group of objects is
abelian, such a structure is often called a Picard groupoid). 
Any categorical $n$-group can be viewed as an $(n+1)$-groupoid
with one object, and vice versa. 

Note that any categorical $n$-group can be truncated 
to a categorical $(n-1)$-group by forgetting the $n$-morphisms
and identifying isomorphic $(n-1)$-morphisms. Conversely, any 
categorical $(n-1)$-group can be regarded as a categorical 
$n$-group by adding the identity $n$-morphism 
from every $(n-1)$-morphism to itself.  

\subsection{Graded tensor categories and extensions}
\label{graded TC}

Let $G$ be a finite group. Recall that a $G$-grading on a tensor category $\C$
is a decomposition 
\begin{equation}
\label{def grading}
\C = \bigoplus_{g\in G}\, \C_g
\end{equation}
into a direct sum of full abelian subcategories
such that the tensor product $\otimes$ maps  
$\C_g \times \C_h$ to $\C_{gh}$ for all $g,h\in G$.
In this case, the trivial component $\C_e$ is a full tensor subcategory of $\C$,
and each $\C_g$ is a $\C_e$-bimodule category. We will always assume that the grading is faithful, i.e., 
$\C_g\neq 0$ for all $g\in G$.

\begin{definition}
A {\em $G$-extension} of a fusion category $\D$ is a
$G$-graded fusion category $\C$ whose trivial component is equivalent to $\D$.
\end{definition}

\subsection{Quadratic forms, bicharacters, metric groups,
Lagrangian subgroups}

Let $E$ be a finite abelian group. 
A {\it bicharacter} on $E$ with values in ${\bold k}^\times$ 
is a biadditive map $b: E\times E\to {\bold k}^\times$.
A {\it symmetric bicharacter} on $E$ (also called an inner product 
or a symmetric bilinear form) is a bicharacter $b$ such that $b(x,y)=b(y,x)$.
A {\it skew-symmetric bicharacter} on $E$ (also called a
 skew-symmetric bilinear form) is a bicharacter $b$ such that $b(x,x)=1$.

Let $E^*={\rm Hom}(E,{\bold k}^\times)$ be the character group of $E$.
By acting on its first argument, any bicharacter $b$ on $E$ defines
a group homomorphism $\widehat{b}: E\to E^*$. We say that $b$  
is {\it nondegenerate} if $\widehat{b}$ is an isomorphism. 
Note that if $E$ admits a nondegenerate skew-symmetric bicharacter, then 
$|E|$ is a square.  

A {\it quadratic form} on $E$ is a function $q: E\to {\bold k}^\times$ 
such that $q(x)=q(x^{-1})$, and $b_q(x,y):=q(x+y)/q(x)q(y)$ is a symmetric
bilinear form. If the order of the group $E$ is odd, the assignment $q\to b_q$ defines a bijection 
between symmetric bilinear forms and quadratic forms, but in
general, it is not a bijection. 

We will say that a quadratic form $q$ is {\it nondegenerate} 
if the bilinear form $b_q$ is nondegenerate. In this case we say that $(E,q)$ 
is a {\it metric group}. To every metric group $(E,q)$, one can attach its {\it orthogonal group}
$O(E,q)$, which is the group of automorphisms of $E$ preserving $q$. 
For example, if $A$ is any finite abelian group then $A\oplus A^*$ is a metric group, 
with hyperbolic quadratic form $q(a,f):=f(a)$. To simplify notation, we will denote the corresponding orthogonal group 
by $O(A\oplus A^*)$.

If $E$ is a finite abelian group with a bicharacter $b$, and $N\subset E$ 
is a subgroup, then the {\it orthogonal complement} $N^\perp$ is the set of $a\in E$ such that 
$b(x,a)=1$ for any $x\in N$. If $b$ is nondegenerate, then $N^\perp$ is identified with $E/N$, so $|N|\cdot 
|N^\perp|=|E|$. 

Let $(E,q)$ be a metric group. We say that a subgroup $L$ of $E$ is {\em isotropic} if 
$q(a)=1$ for any $a\in L$. This implies that $b_q(L)\subset
(E/L)^*$, which implies that $|L|^2\le |E|$. We say that an 
isotropic subgroup $L$ of $E$ is {\em Lagrangian} if $|L|^2=|E|$. 

\subsection{Frobenius-Perron dimensions in module categories}
\label{FPmod}
             
Let $\C$ be a fusion category and let $\M$ be a $\C$-module category.
Recall that for a pair of objects $M,\, N \in \M$ their {\em internal Hom} is the object
of $\C$ denoted $\underline\Hom(M,\,N)$ determined by the natural isomorphism
\[
\Hom_\C(X,\, \underline\Hom(M,\,N)) \cong \Hom_\M(X\ot M,\, N),\qquad X\in \C.
\]
We use this notion to define {\em canonical} Frobenius-Perron dimensions of objects of $\M$.
Let $K_0(\C),\, K_0(\M)$ be the Grothendieck ring of $\C$ and the Grothendieck  group of $\M$.
It follows from \cite{ENO1} that there is a unique $K_0(\C)$-module map 
\[
\FPdim: K_0(\M)\to \mathbb{R}
\]
determined by  
\begin{equation}
\label{can FP}
\FPdim(\underline{\Hom}(M,\,N)) =\FPdim(M)\FPdim(N) 
\end{equation}
for all objects  $M,\, N \in \M$.

Let $\M$ be an indecomposable left $\C$-module category. Let $\O(\C)$
and $\O(\M)$
denote the sets of isomorphism classes of simple objects in $\C$ and $\M$.

\begin{proposition}
$\sum_{M\in \O(\M)}\, \FPdim(M)^2 =\dim(\C)$.
\end{proposition}
\begin{proof}
Let  $R_\C :=\sum_{X\in \O(\C)}\, \FPdim(X)X \in K_0(\C)$ be the virtual regu\-lar object of $\C$. 
We choose a Frobenius-Perron dimension function $d:K_0(\M)\to \mathbb{R}$ 
as in \cite[Proposition 8.7]{ENO1} normalized by
\[
\sum_{M\in \O(\M)}\,  d(M)^2 =\FPdim(\C)
\]
and let $R_\M := \sum_{M\in \O(\M)}\, d(M)M$.
We compute
\begin{eqnarray*}
\sum_{M\in \O(\M)}\, \FPdim(M)^2
&=& \FPdim(  \oplus_{M\in \O(M)}\, \underline{\Hom}(M,\,M) ) \\
&=& \sum_{M\in \O(\M)}\,  [R_\C\ot M : M] \\
&=& \sum_{M\in \O(\M)}\,  d(M) [R_\M : M] \\
&=&  \sum_{M\in \O(\M)}\,  d(M)^2 = \FPdim(\C),
\end{eqnarray*}
as required.
\end{proof}

\begin{remark}
The Frobenius-Perron dimensions in $\M$ defined in \eqref{can FP}
are completely determined by the following properties:
\begin{enumerate}
\item[(i)] $\FPdim(M) >0$ for all $M\in \O(M)$,
\item[(ii)]  $\FPdim(X\ot M) =\FPdim(X)\FPdim(M)$ for all $X\in \C,\,M\in \M$.,
\item[(iii)]   $\sum_{M\in \O(\M)}\, \FPdim(\M)^2 =\dim(\C)$.
\end{enumerate}
\end{remark}

\subsection{The special orthogonal group}
\label{SubSectSO}

Let $(M,q)$ be a metric group. If $L_1,L_2\subset M$ are Lagrangian subgroups, define
$d(L_1,L_2)\in\q$ to be the image of the number
$|L_1|/|L_1\cap L_2|=|L_2|/|L_1\cap L_2|=|M|^{1/2}/|L_1\cap L_2|\in\BN$.
Clearly $d(L_2,L_1)=d(L_1,L_2)=d(L_1,L_2)^{-1}$ and $d(L,L)=1$.

The following proposition and its proof were provided to us by V. Drinfeld.

\begin{proposition} \label{3Lagr}
$d(L_1,L_2)d(L_2,L_3)=d(L_1,L_3)$ for any Lagrangian subgroups $L_1,L_2,L_3\subset M$.
\end{proposition} 

The proposition follows from Lemmas \ref{LagrLem1} -- \ref{LagrLem2} below.

\begin{lemma} \label{LagrLem1}
$d(L_1,L_2)d(L_2,L_3)/d(L_1,L_3)\in\q$ is the image of $|A/B|\in\BN$, where
$A:=(L_1+L_2)\cap L_3$, $B:=(L_1\cap L_3)+(L_2\cap L_3)$.
\end{lemma}

\begin{proof} 
By definition, $d(L_1,L_2)d(L_2,L_3)/d(L_1,L_3)\in\q$ is the image of
\[
|L_1|\cdot |L_2|\cdot |L_3|\cdot |L_1\cap L_2|^{-1} \cdot  |L_1\cap L_3|^{-1} \cdot
|L_2\cap L_3|^{-1} \in\BN
\]
On the other hand, 
\begin{eqnarray*}
|B| &=& |L_1\cap L_3|\cdot |L_2\cap L_3|/|L_1\cap L_2\cap L_3|, \\
|A| &=& |L_1+L_2|\cdot |L_3|/|L_1+L_2+L_3|\\
&=& |L_1| \cdot |L_2|\cdot |L_3|\cdot |L_1\cap L_2|^{-1}\cdot  |L_1+L_2+L_3|^{-1}.
\end{eqnarray*}
Finally, $L_1\cap L_2\cap L_3=(L_1+L_2+L_3)^{\perp}$, so 
$|L_1\cap L_2\cap L_3|\cdot |L_1+L_2+L_3|=|M|=|L_i|^2$ is a square.
\end{proof} 

By Lemma \ref{LagrLem1}, proving Proposition \ref{3Lagr} amounts to showing that
$|A/B|$ is a square. To this end, it suffices to construct a non-degenerate skew-symmetric bicharacter
$c: (A/B)\times (A/B)\to {\bold k}^{\times}$.

Here is the construction. Let $x,y\in A:=(L_1+L_2)\cap L_3$. Represent $x$ and $y$ as
\[
x=x_1+x_2 , \; y=y_1+y_2, \quad x_i,y_i\in L_i
\]
and set $c(x,y):=b(x_1,y_2)=b(x,y_2)=b(x_1,y_2)$, where $b: M\times M\to {\bold k}^{\times}$ is
the symmetric bicharacter associated to $q$. It is easy to see that $c:A\times A\to {\bold k}^{\times}$ 
is a well-defined bicharacter.

\begin{lemma} \label{LagrLem2}
\begin{enumerate}
\item[(i)] $c(x,x)=1$.
\item[(ii)] The kernel of $c:A\times A\to {\bold k}^{\times}$  equals $B$.
\end{enumerate}
\end{lemma}

\begin{proof}
(i) $c(x,x)=b(x_1,x_2)=q(x)q(x_1)^{-1}q(x_2)^{-1}=1$ because   $x\in L_3$,  $x_1\in L_1$,
and $x_2\in L_2$.

(ii) An element $x\in A$ belongs to the kernel of $c$ if and only if $b(x,y)=1$ for all
$y\in L_2\cap (L_1+L_3)$. The orthogonal complement of $ L_2\cap (L_1+L_3)$ with respect to
$b: M\times M\to {\bold k}^{\times}$ equals $L_2+(L_1\cap L_3)$, so $\Ker c=A\cap (L_2+(L_1\cap L_3))$.
Since $A\subset L_3$ we see that $\Ker c\subset (L_2\cap L_3 )+(L_1\cap L_3 )=B$.
On the other hand, $B\subset A$ and $B\subset L_2+(L_1\cap L_3)$, so $B\subset\Ker c$.
\end{proof}

Now for a metric group $E$ and $g\in O(E,q)$ define $\det(g)\in\q$ to be the image of 
$|(g-1)E|\in\mathbb{N}$. 

\begin{proposition}\label{dete}
The map  
$\det : O(E,q)\to\q$ is a  homomorphism. 
\end{proposition}
\begin{proof}
Let $g,h\in O(E,q)$, and let $M=E\oplus E$ with quadratic form 
$Q(x,y)=q(y)/q(x)$, $x,y\in E$. Let $L_1,L_2,L_3\subset M$ be 
the graphs of ${\rm Id}$, $g^{-1}$, and $h$.
They are Lagrangian, and $L_1\cap L_2={\rm Ker}(g-1)$, so $d(L_1,L_2)=\det(g)$.
Similarly, $d(L_1,L_3)=\det(h)$, and $d(L_2,L_3)=\det(gh)$. 
Thus, by Proposition \ref{3Lagr}, $\det(gh)=\det(g)\det(h)$.     
\end{proof}
 
\begin{proposition}\label{dete1} 
If $L$ is a Lagrangian subgroup of $E$ then $\det(g)=d(L,g(L))$. 
\end{proposition}

\begin{proof} 
First, note that by Proposition \ref{3Lagr}, $d(L,g(L))$ is
independent on the choice of $L$. So let us call this function
$\delta(g)$. Next, note that $\delta(g)=\delta(g,1)$, where 
$(g,1)\in O(E\oplus E,q^{-1}\oplus q)$. Finally, note that 
$$
\delta(g,1)=d(E_{\rm diag},(g,1)(E_{\rm diag}))=\det(g),
$$
where $E_{\rm diag}$ is the diagonal copy of $E$. 
\end{proof}

\begin{definition}   \label{SOdef}
The kernel of  the homomorphism 
\[
\det : O(E,q)\to\q
\]
is called the  {\em special orthogonal group} 
and denoted by $SO(E,q)$.
\end{definition}

\begin{remark}
If $E$ is a vector space over $\Bbb F_p$ with $p>2$, then it is easy to see that $\det$ is the usual determinant
(so Definition~\ref{SOdef} agrees with the familiar one from linear algebra). 
Indeed, in this case, any orthogonal transformation is the composition of reflections, 
and it is clear that on reflections, the two definitions of the determinant coincide. 
On the other hand, if $E$ is a vector space over $\Bbb F_2$ and $q$ takes values $\pm 1$
(i.e., in $\Bbb F_2$), then $\det(g)$ coincides with the Dickson invariant of $g$ \cite{Di}, 
(which is also known as Dickson's pseudodeterminant \cite{G}), 
while the usual determinant is trivial. 
\end{remark}

\subsection{Module categories over $\Vec_G$} 
\label{VecGmod}

Let $G$ be a finite group and let
$\C:=\Vec_G$ be the fusion category of $G$-graded vector spaces.
We will denote simple 
objects of $\Vec_G$ simply by $g\in G$.

Recall that equivalence classes of  indecomposable left
$\Vec_G$-module categories correspond to pairs 
$(H,\psi)$ where $H \subset G $
is a subgroup and $\psi \in Z^2(H, {\bold k}^\times)$ 
is a 2-cocycle (modulo cohomological equivalence).  
Namely, let $\M$ be an indecomposable left 
$\Vec_G$-module category, and let $Y$ be a simple
object of $\M$.  Set $H:= \{  g\in G \mid   g\ot Y \cong Y   \}$ 
and let $\psi(x_1, x_2), x_1,x_2\in H$ 
be the scalar such that the associativity constraint
\[
(x_1\ot x_2)\ot Y \xrightarrow{\sim}  x_1\ot (x_2\ot Y)
\]
is given by $\psi(x_1, x_2) \id_{x_1x_2 Y}$.
Let $\M(H, \psi)$ denote the $\Vec_G$-module 
category corresponding to $(H,\psi)$.
Note that the set of isomorphism classes of simple objects of $\M(H, \psi)$  
is in bijection with the set $G/H$ of right cosets of $H$ in $G$.

For any $x \in G $ 
set $H^x:= xHx^{-1}$ and define $\psi^x \in Z^2(H^x, {\bold k}^\times)$ by
\[
\psi^x(xy_1x^{-1}, xy_2x^{-1}) := \psi(y_1, y_2), \quad y_1,y_2\in H.
\]
Two $\Vec_G$-module
categories $\M(H,\psi)$  and 
$\M(H',\psi')$ are equivalent if and only if there is $x\in G$ such that
$H' =xHx^{-1} $ and  $\psi'$ is cohomologous to $\psi^x$.

If $H$ is abelian, then $H^2(H,{\bold k}^\times)$ is 
the group of skew-symmetric bicharacters of $H$.
Thus, if $A$ is a finite abelian group, then 
the indecomposable left module categories over $\Vec_A$
are $\M(H,\psi)$, where $H\subset A$ is a subgroup, and 
$\psi$ is a skew-symmetric bicharacter of $H$. 

Let $A,B$ be abelian groups, $\phi: B\to A$ be a group 
homomorphism (not necessarily injective), 
and $\xi$ be a skew-symmetric bicharacter of 
$B$ with coefficients in ${\bold k}^\times$. 
Let $K={\rm Ker}\phi$, $K^\perp$ be the orthogonal complement
of $K$ in $B$ under $\xi$, and $H=\phi(K^\perp)$. It is easy to
show that $\xi$ descends to a skew-symmetric bicharacter of 
$H$, which we will denote by $\psi$.   

\begin{proposition}\label{noninj}
Let $\N$ be the category of $A$-graded vector spaces
which are right-equivariant under the action of $B$ (via $\phi$)
with 2-cocycle $\xi$. Then as a left $\Vec_A$ module, 
$\N\cong m\cdot \M({H,\psi})$, where
$$
m=\frac{|K|\cdot |K^{\perp}|}{|B|}=|K\cap {\rm Rad}(\xi)|.
$$
\end{proposition}

\begin{proof}
The simple objects of $\N$ are obviously parameterized 
by pairs $(z,\rho)$, where $z\in A/\phi(B)$, and $\rho$ is an irreducible
projective representation of $K$ with cohomology class $\xi|_K$,
which implies that the number of simple objects of 
$\N$ and $m\cdot \M({H,\psi})$ is the same. 

Now consider the stabilizer $S$ of a pair $(z,\rho)$ in $A$. Obviously, 
$S$ is contained in $\phi(B)$, since an element of $S$ must
preserve $z$. Further, if $g\in B$, the action of $\phi(g)$ on 
$\rho$ is by tensoring with the character
$\xi(g,\cdot)$. So the condition that $\phi(g)$ fixes $\rho$
is that $\xi(g,k)=1$ for any $k\in {\rm Rad}(\xi|_K)=K^\perp\cap
K$, i.e. $g\in K+K^\perp$ (indeed, we have the equality 
$(K\cap K^\perp)^\perp=K^{\perp\perp}+K^\perp=K+K^\perp$, since 
$K^{\perp \perp}=K+{\rm Rad}(\xi)$, and 
$K^\perp$ contains ${\rm Rad}(\xi)$). 
Thus $S=H$. It is straightforward to check that the corresponding
second cohomology class on $S$ is exactly $\psi$.
The proposition is proved. 
\end{proof} 

We also have the following proposition, whose proof is easy and omitted.  

\begin{proposition}\label{subg}
Let $A$ be a finite abelian group, $H,B\subset A$ subgroups, and $\psi\in H^2(H,{\bold k}^\times)$
be a skew-symmetric bicharacter. Then one has an equivalence of left $\Vec_B$-module categories
$$
\M(H,\psi)|_{\Vec_B}\cong m\cdot \M(H\cap B,\psi|_{H\cap B}),
$$
where $m$ is the index of $B+H$ in $A$. 
\end{proposition}

\subsection{The center of a bimodule category} 
\label{bimodule center}

Let $\C$ be a fusion category with unit object $\be$ and
associativity constraint  $\alpha_{X,Y,Z}:(X\ot Y)\ot Z \xrightarrow{\sim} X\ot (Y \ot Z)$,
and let $\M$ be a $\C$-bimodule category. The following definition
was given in \cite{GNN}.


\begin{definition}
\label{center of a module}
The {\em center} of $\M$ is the category $\Z_\C(\M)$ of $\C$-bimodule functors
from $\C$ to $\M$.
\end{definition}

Explicitly, the objects of $\Z_\C(\M)$ are pairs $(M,\, \gamma)$, where $M$ is an object of $\M$ and
\begin{equation}
\label{gamma}
\gamma = \{ \gamma_{X} : X\ot M \xrightarrow{\sim}  M \ot X \}_{ X \in \C}
\end{equation}
is a natural family of isomorphisms making the following diagram
commutative:
\begin{equation}
\label{central object}
\xymatrix{
 X \ot (M \ot Y) \ar[rr]^{\alpha_{X,M,Y}^{-1}}  & &
(X \ot M) \ot Y  \ar[d]^{\gamma_{X}\otimes \id_Y} & \\
X\ot (Y\ot M) \ar[u]^{\id_X\otimes \gamma_{Y}}
\ar[d]_{\alpha_{X,Y,M}^{-1}}   &  & (M \ot X) \ot Y \\
 (X\ot Y)\ot M \ar[rr]_{\gamma_{X\ot Y}} &  & M\ot (X\ot Y)
\ar[u]_{\alpha_{M,X,Y,}^{-1}} & }
\end{equation}
where $\alpha$'s denote the associativity constraints in $\M$.

Indeed, a $\C$-bimodule functor $F: \C \to \M$ is completely determined
by the pair $(F(\be),\, \{\gamma_X\}_{X\in \C})$, where $\gamma = \{\gamma_X\}_{X\in \C}$
is the collection of isomorphisms
\[
\gamma_X : X \ot F(\be) \xrightarrow{\sim}  F(X) \xrightarrow{\sim}   F(\be) \ot X
\]
coming from the $\C$-bimodule structure on $F$.


\begin{remark}
 $\Z_\C(\M)$ is a semisimple abelian category. It has a natural  structure
of a $\Z(\C)$-module category. Also, it is clear that
$\Z_\C(\C)=\Z(\C)$. 
\end{remark}

\subsection{The opposite module category}
\label{op modcat}

Let $\C$ be a fusion category, and $\M$ a right $\C$-module category.
Let $\M^\op$ be the category opposite to $\M$.
Then $\M^\op$ is a left $\C$-module category with
the $\C$-action $\odot$ given by $X\odot M := M\ot {}^*X$.
Similarly, if $\N$ is a left $\C$-module category, then 
$\N^\op$ is a right $\C$-module category, 
with the $\C$-action $\odot$ given by $N\odot X := X^*\otimes N$.
Note that $(\M^\op)^\op$ is canonically equivalent to $\M$ as a 
$\C$-module category.

More generally, given a $(\C,\D)$-bimodule category $\M$, 
the above definitions make   $\M^\op$ a $(\D,\C)$-bimodule category.

\section{Tensor product of module categories}
\label{Sect3}

\subsection{Definition of the  tensor product of module categories over a fusion category}

Let $\C$, $\D$ be fusion categories. By definition, a $(\C,\D)$-bimodule category is
a module category over $\C \boxtimes \D^{\rev}$, where
$\D^{\rev}$ is the category $\D$ with reversed tensor product. 

Let $\M =(\M,\, m)$ be a right $\C$-module category and let $\N =(\N, \, n)$ 
be a left $\C$-module category. Here $m$ and $n$ are the associativity constraints:
\begin{eqnarray*}
m_{M,X,Y} &:& M \ot (X \ot  Y) \to (M\ot X)\ot Y,\\
n_{X,Y, N} &:& (X\ot Y)\ot N\to X \ot (Y \ot N),
\end{eqnarray*}
where $X,Y\in \C, M\in \M,\, N\in \N$.

Let $\A$ be a semisimple  abelian category.

\begin{definition}
\label{C-balanced functor}
Let $F : \M \times \N \to \A$ be a bifunctor additive in every argument. 
We say that $F$ is  {\em $\C$-balanced} if there is 
a natural family of isomorphisms 
$$
b_{M,X,N} : F(M\otimes X,\, N)  \cong F(M,\, X \otimes N), 
$$
satisfying the following commutative diagram
\begin{equation}
\label{C-balanced diagram}
\xymatrix{
F(M \otimes (X \otimes Y),\, N)   \ar[d]_{b_{M, X\otimes Y, N}}   
\ar[rr]^{m_{M,X,Y}} & & F((M \otimes X )\otimes Y,\, N) 
\ar[d]^{b_{M \otimes X, Y, N}}    \\
F(M ,\, (X \otimes Y) \otimes N)   & &
F(M\otimes X ,\, Y\otimes N) \ar[dl]^{b_{M,X, Y\otimes N}}  \\
&  F(M,\, X \otimes(Y \otimes N) )  \ar[ul]^{n^{-1}_{X,Y,N}}  , & }
\end{equation}
for all $M\in \M,\, N\in \N,\,X,Y\in \C$.
\end{definition}

\begin{remark}
\label{balanced ext}
A bifunctor $\M \times \N \to \A$ as above  canonically extends 
to a functor $\M \boxtimes \N \to \A$, where $M\boxtimes \N$ 
is the Deligne product of abelian categories \cite{D}. Clearly, one can formulate
the balancing property in terms of  functors $\M \boxtimes \N \to \A$.
\end{remark}

We define tensor product of $\C$-module categories 
by  ``categorifying''  the definition of a tensor product of modules over a ring.
This extends the notion of Deligne's tensor product of abelian categories 
(i.e., module categories over $\Vec$) to the context of module categories over tensor categories. 
In the setting of additive ${\bold k}$-linear (not necessarily abelian) categories
the notion of tensor product of module categories was given by D.~Tambara in \cite{Ta}.

\begin{definition}
\label{tensor product in BC}
A {\em tensor product} of a right $\C$-module category $\M$ and a left $\C$-module
category $\N$ is an abelian category $\M \boxtimes_\C \N$ together with a
$\C$-balanced functor
\begin{equation}
\label{B-map}
B_{\M,\N}: \M \times \N \to \M \boxtimes_\C \N
\end{equation}
inducing, for every abelian category $\A$,
an equivalence
between the category of $\C$-balanced functors from $\M \times \N$ to $\A$ and 
the category of functors from $\M \boxtimes_\C \N$ to $\A$:
\begin{equation}
\Fun_{bal} (\M \times \N,\, \A) \cong \Fun (\M \boxtimes_\C \N,\, \A).
\end{equation}
\end{definition}

\begin{remark}
Equivalently, bifunctor \eqref{B-map}
is universal for all $\C$-balanced
bifunctors from $\M \times \N$ to abelian categories. In other words, 
for any  $\C$-balanced functor 
$F : \M \times \N \to \A$ there exists a unique additive functor
$F' : \M \boxtimes_\C \N  \to \A$ making the following diagram commutative
\begin{equation}
\label{universality diagram}
\xymatrix{
\M \times \N \ar[dd]_{B_{\M,\N}} \ar[ddrr]^F & & \\
& &  \\
\M\boxtimes_\C \N \ar@{-->}[rr]_{F'} & & \A.
}
\end{equation}
If $\M$ and $\N$ are $\C$-bimodule categories then so is $\M \boxtimes_\C \N$.
\end{remark}

\subsection{Tensor  product as a category of module functors}

Let us show that the tensor product of bimodule categories 
introduced in Definition~\ref{tensor product in BC} does exist. 

Let $\C$ be a fusion category, let $\M$ be a right $\C$-module category
and $\N$ be a left $\C$-module category. There is an obvious equivalence
\begin{equation}
\label{bt = fun}
\M \bt \N \xrightarrow{\sim} \Fun(\M^\op,\, \N) : M\bt N \mapsto \Hom_\M( -,\, M)\ot N,
\end{equation}

Under the isomorphism \eqref{bt = fun} $\C$-balanced functors $\M\bt \N\to \A$
correspond to functors $F: \Fun(\M^\op,\, \N) \to \A$  with an isomorphism
\begin{equation}
\label{functor coherence}
F(T(X\ot ?)) \cong F(X\ot T(?)),\qquad \mbox{where } T: \M^\op \to \N,
\end{equation}
satisfying a coherence condition similar to diagram \eqref{C-balanced diagram}.

\begin{proposition}
\label{tp = functors}
There is an equivalence of abelian categories
\begin{equation}
\label{check}
\M \bt_\C \N \cong \Fun_\C(\M^\op,\, \N).
\end{equation}
\end{proposition}
\begin{proof}
Let $F: \M \boxtimes \N \to \mathcal{A}$ be 
the extension of some   $\C$-balanced bifunctor as in Remark~\ref{balanced ext}
and let $G: \mathcal{A} \to \M \boxtimes \N$ be its right adjoint.
Using the equivalence \eqref{bt = fun} and coherence \eqref{functor coherence}  
one can check that for every 
$A\in \mathcal{A}$ the functor $G(A)$ in $\Fun(\M^\op,\, \N)$  has a canonical
structure of a  $\C$-module functor. 
Thus, $G$ factors through the obvious forgetful functor 
$U: \Fun_\C(\M^\op,\, \N) \to \Fun(\M^\op,\, \N)$:
\begin{equation}
\label{adjoint universality diagram}
\xymatrix{
\Fun(\M^\op,\, \N)  & & \\
& &  \\
\Fun_\C(\M^\op,\, \N) \ar[uu]^{U} 
& & \A \ar@{-->}^(.4){G'}[ll] \ar[uull]_{G}.
}
\end{equation}
Taking left adjoints we recover diagram~\eqref{universality diagram}.
\end{proof}

\begin{remark}
\label{properties of btD}
\begin{enumerate}
\item[(i)] It is easy to see that if $\M$ is a $(\D,\C)$-bimodule category, 
and $\N$ is a $(\C,\E)$-bimodule category then \eqref{check} is an equivalence 
of $(\D,\E)$-bimodule categories.
\item[(ii)]  Let $\M$ be a right $\C$-module category, 
$\N$ a $(\C,\D)$-bimodule category, and $\K$ a left $\D$-module
category. Then there is a canonical 
equivalence $(\M\bt_\C \N) \bt_\D \K \cong \M\bt_\C (\N \bt_\D \K )$ of categories.
Hence the notation $\M\bt_\C \N \bt_\D \K$ will yield no ambiguity.
\end{enumerate}
\end{remark}

We refer the reader to the work of J.~Greenough \cite{Gr} for an
alternative proof of Proposition~\ref{tp = functors}. It is shown in \cite{Gr}
that for any fusion category  
$\C$ its bimodule categories equipped with the tensor product  
$\boxtimes_\C$  form a (non semi-strict)
monoidal $2$-category in the sense of Kapranov and Voevodsky \cite{KV}.
We denote this monoidal 2-category 
by ${\bf Bimodc}(\C)$.

More generally, one can define the 
tricategory ${\bf Bimodc}$ of
bimodule categories over fusion categories, 
in which 1-morphisms from $\C$ to $\D$ are 
$(\C,\D)$-bimodule categories (with composition being tensor
product of bimodule categories as defined above), 2-morphisms are
bimodule functors 
between such bimodule categories, and 3-morphisms are morphisms of such
bimodule functors. Then ${\bf Bimodc}(\C)$ consists of
1-morphisms from $\C$ to $\C$ 
in ${\bf Bimodc}$, and the corresponding 2-morphisms and
3-morphisms. 

\begin{remark} The tricategory ${\bf Bimodc}$ is a categorification
of the 2-category ${\bf Bimod}$, whose objects are rings,
1-morphisms are bimodules, and 2-morphisms are homomorphisms of
bimodules. 
\end{remark}

\subsection{Tensor product as the center of a bimodule category}

Let $\C$ be a fusion category.
Below we describe the tensor product of $\C$-module categories 
in a way convenient for computations. Recall that the center
of a $\C$-bimodule category was defined in Section~\ref{bimodule center}.

As before, let $\M$ be a right $\C$-module category and let $\N$
be a left $\C$-module category. The category $\M\bt \N$ has 
a natural structure of a $\C$-bimodule category.
It turns out that its center $\Z_\C({\M \boxtimes \N})$ can be identified
with $\M \boxtimes_\C \N$.

Let $F: \M \times \N \to \A$ be a $\C$-balanced functor.
Let $\bar{F} :\M \boxtimes \N \to \A$ be the extension of $F$
and let  $G : \A \to \M \boxtimes \N$ be the functor right adjoint to $\bar{F}$.
Let
\begin{equation}
i: \Hom_\A(\bar{F}(V),\, W) \cong \Hom_{\M \boxtimes \N}(V,\, G(W) ) 
\end{equation}
be the adjunction isomorphism.
Let  
\[
c_{X, G(A)} : G(A) \otimes (X \boxtimes {\bold 1}) 
\cong ({\bold 1}\boxtimes X) \otimes G(A),\quad  A\in\A,
\]
be the image under $i$ of the isomorphism
\[
b_{V, {}^*X} : F(V \otimes ({}^*X \boxtimes {\bold 1})) \cong F(({\bold 1}\boxtimes {}^*X) \otimes V),\quad V\in \M.
\]
Then $G(A)$ is an object of $\Z_\C({\M \boxtimes \N})$
and 
$G' : \A \to  \Z_\C({\M \boxtimes \N}): A \mapsto  G(A)$
satisfies $UG' = G$, where
\begin{equation}
\label{U}
U: \Z_\C({\M \boxtimes \N}) \to \M \boxtimes \N
\end{equation}
is the obvious forgetful functor. Let 
\begin{equation}
\label{IMN}
I_{\M,\N}: \M \boxtimes \N \to
\Z_\C({\M \boxtimes \N})
\end{equation}
be the right adjoint of $U$. 

\begin{proposition}
\label{tensor product = centralizer}
There is a canonical equivalence
\[
\M \boxtimes_\C \N \cong \Z_\C({\M \boxtimes \N})
\]  
such that
$I_{\M,\N}:  \M \boxtimes \N \to \Z_\C({\M \boxtimes \N})$  is identified with
the extension of the universal bifunctor $\B_{\M,\N} : \M \times \N \to \M \boxtimes_\C \N$.
\end{proposition}
\begin{proof}
From the above discussion we have a commutative diagram
\begin{equation}
\label{adjoint universality diagram1}
\xymatrix{
\M \boxtimes \N & &\\
& &  \\
\Z_\C({\M\boxtimes_\C \N}) \ar[uu]^{U} 
& & \A. \ar@{-->}[ll]^(.4){G'} \ar[uull]_{G}
}
\end{equation}
Taking the adjoint diagram gives the result.
\end{proof}

%
%

\begin{remark}
\label{yet one more}
There is yet one more description of $\M \bt_\C \N$. Namely, let 
$A \in \C \bt \C^\rev$ be the object representing the functor
$ \otimes : \C \bt \C^\rev \to \C$. Then $A=\oplus_{X}\,X^*\bt X$ (summation taken over simple objects of $\C$)
is an algebra in  $\C \bt \C^\rev$, and $\M \bt_\C \N$ is equivalent to the category
of  left $A$-modules in $\M \bt \N$. The canonical functor
\[
\M \bt \N \to \M\bt_\C \N
\]
is identified with the middle multiplication by $A$.
\end{remark}


\subsection{Tensor product  of module categories over a braided category}
\label{mod prod}

Let $\B$ be a braided fusion category.  Since every left (or right) $\B$-module category is
automatically a $\B$-bimodule category (using the braiding in
$\B$), one can tensor any two
such categories, to get a third one.
Let ${\bf Modc}(\B)$ denote the monoidal $2$-category of
(left) $\B$-module categories. Clearly, ${\bf Modc}(\B)$
is a full subcategory  of the monoidal 2-category ${\bf Bimodc}(\B)$ of
all $\B$-bimodule  categories.

\begin{remark}
The monoidal 2-category ${\bf Modc}(\B)$ is a categorification
of the monoidal category ${\bf Mod}(A)$ of modules over a
commutative ring $A$. Note that unlike ${\bf Mod}(A)$, 
the monoidal 2-category ${\bf Modc}(\B)$ is, in general, not
symmetric or braided; in fact, in this category, $X\otimes Y$
may be non-isomorphic to $Y\ot X$. 
\end{remark}

Let $\C$ be a fusion category.
Recall \cite{EO} that there is a $2$-equivalence  
\begin{equation}
\label{EO-equiv}
\Z_\C: {\bf Bimodc}(\C) 
\xrightarrow{\sim} {\bf Modc}(\Z(\C)) : \M \mapsto \Z_\C(\M),
\end{equation}
where the center $\Z_\C(\M)$ is defined in Section~\ref{bimodule center}.

The next Proposition is proved in \cite{Gr}. We include its proof for reader's convenience.

\begin{proposition}
\label{bimod =modZC}
The 2-equivalence $\Z_\C$ is monoidal. That is, 
for any pair $\M,\,\N$ of $\C$-bimodule categories we have 
a natural equivalence
\begin{equation}
\label{Z is eq}
\Z_\C(\M \bt_\C \N) \cong \Z_\C(\M)  \bt_{\Z(\C)} \Z_\C(\N),
\end{equation}
which satisfies appropriate compatibility conditions. 
\end{proposition}
\begin{proof}
By Proposition~\ref{tensor product = centralizer} the left hand side
of \eqref{Z is eq} is identified as a $\Z(\C)$-module category 
with  $\Z_{\C\bt \C^\rev}(\M \bt \N)$
where the left and right actions of the object $X\bt Y \in \C\bt \C^\rev$ 
on $M\bt N \in \M \bt \N$ are given by
\begin{eqnarray*}
(X\bt Y) \ot (M\bt N) &:=&  ( X \ot M ) \bt (Y\ot N)  \\
(M\bt N)  \ot (X\bt Y)  &:=&  (M \ot Y) \bt  (N \ot X).
\end{eqnarray*}
On the other hand, combining Proposition~\ref{tp = functors} and the
$2$-equivalence \eqref{EO-equiv} we obtain a sequence of
$\Z(\C)$-module category  equivalences
\begin{eqnarray*}
\Z_\C(\M)  \bt_{\Z(\C)} \Z_\C(\N) 
&\cong& \mbox{Fun}_{\Z(\C)}(\Z_\C(\M)^\op,\, \Z_\C(\N)) \\ 
&\cong& \mbox{Fun}_{\C\bt \C^\rev}(\M^\op,\, \N) \\
&\cong& \Z_{\C\bt \C^\rev}(\M \bt \N),
\end{eqnarray*}
where the bimodule action of $\C\bt \C^\rev$ on $\M \bt \N$ is the same 
as the one described above.
\end{proof}

\begin{remark}
It is possible to show that \eqref{EO-equiv} is  an equivalence of monoidal $2$-categories.
\end{remark}

\begin{corollary}
Any Morita equivalence between fusion categories $\C_1$ and $\C_2$ canonically gives rise to an equivalence
of monoidal $2$-categories ${\bf Bimodc}(\C_1)$  and ${\bf Bimodc}(\C_2)$. 
\end{corollary}
\begin{proof}
By \cite{Mu}, the Morita equivalence between $\C_1$ and $\C_2$ gives rise to an equivalence 
$\Z(\C_1)\cong \Z(\C_2)$ as braided fusion categories, and so the result follows from
Proposition~\ref{bimod =modZC}.
\end{proof}

The next statement was formulated in \cite[Section 4.3]{DGNO}.

\begin{corollary}
Let $\C_1$ and $\C_2$ be Morita equivalent fusion categories.  Let $\K_i,\, i=1,2$
be the $2$-category of fusion categories $\C$ equipped with a tensor functor $\C_i\to \C$.
Then $\K_1$ and $\K_2$ are $2$-equivalent.
\end{corollary}
\begin{proof}
This follows from the observation that $\K_i$ can be interpreted as the $2$-category of
algebras in $\C_i,\, i=1,2$ (cf.\ \cite[Remarks 4.38(i)]{DGNO}).
\end{proof}

\subsection{Tensor product of module categories 
over $\Vec_A$, where $A$ is a finite abelian group}

Let $A$ be an abelian group and $\C=\Vec_A$. 
Since $\C$ is a symmetric category, any left
$\C$-module category can be viewed as a right $\C$-module
category, and thus the dual $\M^{\rm op}$ of a left $\C$-module
category $\M$ is again a left $\C$-module category. 
Also, the tensor product over $\C$ of two left $\C$-module
categories is again a left $\C$-module category.

For any subgroup $H\subset A$ and a skew-symmetric bicharacter $\psi$
on $H$ let $\M({H,\psi})$ be the $\C$-module category constructed in Section~\ref{VecGmod}.
The following Lemma is easy, and its proof is omitted. 

\begin{lemma}\label{dua}
One has $\M({H,\psi})^{\rm op}=\M({H,\psi^{-1}})$.
\end{lemma}

Now let us give an explicit description of the 
tensor product of $\C$-module categories.

We repeat the construction preceding Proposition \ref{noninj}.
Let $H_1,H_2\subset A$ be subgroups of a finite abelian group
$A$, and $\psi_1,\psi_2$ be skew-symmetric bicharacters on
them. Consider the group $H_1\cap H_2$ embedded antidiagonally (i.e., by $h\mapsto
(-h,h)$) into $H_1\oplus H_2$. Let $(H_1\cap H_2)^\perp$ be the
orthogonal complement of this group under the bicharacter
$\psi_1\times \psi_2$ on $H_1\oplus H_2$. Let $H$ be the image of
$(H_1\cap H_2)^\perp$ in $H_1+H_2\subset A$, under the map
$(h_1,h_2)\mapsto h_1+h_2$. We have an exact sequence 
$$ 
0\to {\rm Rad}((\psi_1\times \psi_2)|_{H_1\cap H_2})\to (H_1\cap
H_2)^\perp\to H\to 0. 
$$ 
Therefore, the bicharacter
$(\psi_1\times \psi_2)|_{(H_1\cap H_2)^\perp}$ descends to a
bicharacter on $H$, which we will denote by $\psi$.

\begin{proposition}\label{tenpro} One has 
$$ 
\M({H_1,\psi_1})\boxtimes_\C
\M({H_2,\psi_2})= 
m\cdot \M({H,\psi}), 
$$ 
where 
$$ 
m=\frac{|(H_1\cap H_2)^\perp|\cdot |H_1\cap H_2|}{|H_1|\cdot
|H_2|}=|H_1\cap H_2\cap {\rm Rad}(\psi_1\times \psi_2)|. 
$$
\end{proposition}

\begin{proof}
Using Lemma \ref{dua}, we get
$$
\M({H_1,\psi_1})\boxtimes_\C
\M({H_2,\psi_2})= 
{\rm Fun}_\C(\M({H_1,\psi_1^{-1}}),\M({H_2,\psi_2})). 
$$
According to \cite{O1}, this category can be described 
as the category of $A$-graded vector spaces 
which are left equivariant under the action of $H_1$
with 2-cocycle $\psi_1$ and right equivariant 
under the action of $H_2$ with 2-cocycle $\psi_2$. 
Since $A$ is abelian, this is the same as considering 
$A$-graded vector spaces which are right-equivariant 
under the action $H_1\oplus H_2$ with cocycle $\psi_1\times
\psi_2$. So the result follows immediately from Proposition \ref{noninj}. 
\end{proof}

\begin{corollary}\label{inve} 
(i) The $\C$-module category 
$\M({H,\psi})$ is invertible
if and only if 
$\psi$ is nondegenerate. 

(ii) The group of equivalence classes of 
invertible $\C$-module categories
is naturally isomorphic to the group $H^2(A^*,{\bold k}^\times)$ of 
skew-symmetric bicharacters of $A^*$ via
$\M({H,\psi})\mapsto \psi^\vee|_{A^*}$, where 
$\psi^\vee$ is the bicharacter 
on $H^*$ dual to $\psi$ (i.e., $\widehat{\psi^\vee}=\widehat{\psi}^{-1}$).
\end{corollary}

\begin{proof}
This follows from Proposition \ref{tenpro} via a direct
calculation. 
\end{proof} 

\begin{example}
Assume that $H_1=H_2=A$, and $\psi_1,\psi_2$ are such that 
$\psi_1\psi_2$ is a nondegenerate bicharacter. 
In this case, Proposition \ref{tenpro} implies that $H=A$ and 
$$
\widehat{\psi}=\widehat{\psi_1}\circ
(\widehat{\psi_1\psi_2})^{-1} \circ \widehat{\psi_2}=
\widehat{\psi_2}\circ
(\widehat{\psi_1\psi_2})^{-1} \circ \widehat{\psi_1}.
$$
Note that if $\psi_1,\psi_2$ are themselves nondegenerate,
this is a special case of Corollary \ref{inve}. 
\end{example}

\subsection{Tensor product of bimodule categories} 

Let us now compute the tensor product of bimodule categories over the categories of vector spaces graded by 
finite abelian groups. 

Let $A_1,A_2,A_3$ be finite abelian groups. 
Let $H\subset A_1\oplus A_2$, $H'\subset A_2\oplus A_3$ be subgroups, and let $\psi,\psi'$ be 
skew-symmetric bicharacters of $H$, $H'$ respectively. Let us repeat, with some modifications, the construction preceding 
Proposition \ref{noninj}. Namely, let $H\circ H'$ be the subgroup of elements $(a_1,-a_2,a_2,a_3)$ 
in $H\oplus H'$, and $H\cap H'\subset A_2$ be the intersection of 
$H$ and $H'$ with $A_2$. We regard $H\cap H'$ as a subgroup of $H\circ H'$ 
via the antidiagonal embedding $h\to (-h,h)$, and let $(H\cap H')^\perp$ denote the orthogonal complement 
of $H\cap H'$ in $H\circ H'$ with respect to the bicharacter $(\psi\times \psi')|_{H\circ H'}$. Finally, let $H''$ be the image of 
$(H\cap H')^\perp$ in $A_1\oplus A_3$. Obviously, the bicharacter $(\psi\times \psi')|_{H\circ H'}$ 
descends to a skew-symmetric bicharacter of $H''$, which we denote by $\psi''$.  

Then we have the following proposition, whose proof is parallel to the proof of Proposition \ref{tenpro}.

For a subgroup $B\subset A$ of a finite abelian group $A$, write $B_\perp$ for the annihilator of $B$ in $A^*$
(to avoid confusion with the orthogonal complement with respect to a bicharacter, we use a subscript rather than a superscript). 

\begin{proposition}\label{tenspro1}
$$
\M(H,\psi)\boxtimes_{\Vec_{A_2}} \M(H',\psi')=m\cdot \M(H'',\psi''),
$$
where 
$$
m=
\frac{|H\cap H'|\cdot |(H\cap H')^\perp|\cdot |A_2|}{|H|\cdot |H'|}=
$$
$$
\frac{|H\cap H'|\cdot |(H\cap H')^\perp|\cdot |H_\perp\cap H_\perp'|}{|H\circ H'|}
$$
\end{proposition}

\begin{proof} Let $B:=H\oplus H'\oplus A_2$, and 
$\phi: B\to A_1\oplus A_2\oplus A_2\oplus A_3$ 
be the homomorphism given by the formula $\phi(h,h',a)=(h,h')+(0,a,-a,0)$. 
Let $\xi$ denote the bicharacter $\psi\times \psi'\times 1$ of $B$.

Using Lemma \ref{dua}, we get
$$
\N:=\M({H,\psi})\boxtimes_{\Vec_{A_2}}
\M({H',\psi'})= 
{\rm Fun}_{\Vec_{A_2}}(\M({H,\psi^{-1}}),M({H',\psi'})). 
$$
Thus, according to \cite{O1}, $\N$ can be described 
as the category of $A_1\oplus A_2\oplus A_2\oplus A_3$-graded vector spaces 
which are right equivariant under the action $B$ 
with cocycle $\xi$. By Proposition \ref{noninj}, 
this means that as a left $\Vec_{A_1\oplus A_2\oplus A_2\oplus A_3}$-module 
category, $\N$ is equivalent to $r\cdot \M(E,\theta)$, with
$E=\phi(K_B^\perp)$, $\theta=\xi|_E$ (the pushforward of $\xi$ to $E$, which is obviously well-defined), and 
$$
r=\frac{|K|\cdot |K_B^\perp|}{|H|\cdot |H'|\cdot |A_2|},
$$ 
where $K={\rm Ker}(\phi)=H\cap H'$ embedded into $H\oplus H'\oplus A_2$ 
via $a\mapsto (a,-a,a)$. (Here $K_B^\perp$ stands for the orthogonal complement of $K$ in $B$.)
Thus, by Proposition \ref{subg}, 
as a left $\Vec_{A_1\oplus A_3}$-module 
category, $\N$ is indeed a multiple of $\M(H'',\psi'')$. 

It remains to prove the formulas for the coefficient $m$. 
>From the above we get 
$$
m=\frac{|H\cap H'|\cdot |K_B^\perp|}{|H|\cdot |H'|\cdot |A_2|}|{\rm Coker}(K^\perp_B\to A_2\oplus A_2)|=
$$ 
$$
\frac{|H\cap H'|\cdot |A_2|}{|H|\cdot |H'|}|{\rm Ker}(K^\perp_B\to A_2\oplus A_2)|=
$$ 
$$
\frac{|H\cap H'|\cdot |A_2|}{|H|\cdot |H'|}|(H\cap H')^\perp|,
$$ 
which is the first formula for $m$. To get the second formula, 
note that we have an exact sequence 
$$
0\to H\circ H'\to H\oplus H'\to A_2\to (H_\perp\cap H_\perp')^*\to 0,
$$
so 
$$
\frac{|A_2|}{|H|\cdot |H'|}=\frac{|H_\perp\cap H_\perp'|}{|H\circ H'|}.
$$ 
Substituting this into the first formula, we get 
$$
m=
\frac{|H\cap H'|\cdot |H_\perp\cap H_\perp'|}{|H\circ H'|}|(H\cap H')^\perp|,
$$
which is the second formula for $m$.
\end{proof} 

\section{Higher groupoids attached to fusion categories}
\label{Sect4}

\subsection{Invertible bimodule categories over fusion categories and the Brauer-Picard
3-groupoid}
\label{inv+BP}

Let $\C,\D$ be fusion categories. Recall from Section~\ref{op modcat}
that given a $(\C,\D)$-bimodule category $\M$  its opposite  $\M^\op$
is  a $(\D,\C)$-bimodule category.

\begin{definition}
\label{definition inv}
We will say that a $(\C,\D)$-bimodule category $\M$ is {\em invertible} 
if there exist bimodule equivalences
\begin{equation}
\label{Mop times M}
\M^\op \bt_\C \M \cong \D \quad \mbox{and} \quad \M \bt_\D \M^\op \cong \C.
\end{equation}
\end{definition}

Let $\M$ be a $(\C,\D)$-bimodule category.
We will denote $\Fun_\C(\M,\, \M)$ (respectively,  $\Fun(\M,\, \M)_\D$) 
the category of left (respectively, right) module endofunctors of $\M$.
Note that these categories $\Fun_\C(\M,\, \M)$
and $\Fun(\M,\, \M)_\D$ are at the same time multifusion categories and bimodule 
categories (over $\D$ and $\C$, respectively).

Note that for any object $X$ in $\D$ (respectively, $\C$) the right  (respectively, left)  multiplication
by $X$ gives rise to a left (respectively, right)  $\C$-module 
(respectively, $\D$-module) endofunctor of $\M$ denoted $R(X)$ (respectively, $L(X)$). 
Thus, we have tensor functors
\begin{eqnarray}
\label{R}
R: X  \mapsto R(X) &:& \D^\rev \to \Fun_\C(\M,\, \M) \quad \mbox{and} \\
\label{L}
L: X \mapsto  L(X) &:& \C \to \Fun(\M,\, \M)_\D.
\end{eqnarray}

\begin{proposition}
\label{inv criterion}
Let $\M$ be a $(\C,\D)$-bimodule category. The following conditions are equivalent
\begin{enumerate}
\item[(i)] $\M$ is invertible,
\item[(ii)]  There exists a $\D$-bimodule equivalence $ \M^\op \bt_\C \M \cong \D$,
\item[(iii)]  There exists a $\C$-bimodule equivalence $\M \bt_\D \M^\op \cong \C$,
\item[(iv)] The functor \eqref{R} is an equivalence,
\item[(v)] The functor \eqref{L} is an equivalence.
\end{enumerate}
\end{proposition}
%
\begin{proof}
By definition, (i) is equivalent to (ii) and (iii) combined. 

Recall that by Proposition~\ref{tp = functors} $\M^\op \bt_\C \M \cong \Fun_\C(\M,\, \M)$
as $\D$-bimodule categories. We also have $\M \bt_\D \M^\op \cong  \Fun_\D(\M^\op,\, \M^\op)
= \Fun(\M,\, \M)_\D$ as $\C$-bimodule categories. So  (iv) implies (ii) and (v) implies (iii).

Next, suppose that $\M$ is invertible and let $\phi: \D \cong \Fun_\C(\M,\, \M)$
be a $\D$-bimodule  equivalence.  Let $F = \phi(\be)$. Then any functor in 
$\Fun_\C(\M,\, \M)$ is isomorphic to $F\circ R(X)$ for some $X\in \D$. This
is only possible when $R$ is an equivalence. Thus, (ii) implies (iv).
The proof that (iii) implies (v)  is completely similar.

It remains to show that (iv) is equivalent to (v).  If \eqref{R} is an equivalence then
$\D^\rev$ is the dual $\C^*_\M$ of $\C$ with respect to $\M$ and
the functor \eqref{L} identifies with the canonical tensor functor $\C \to (\C^*_\M)^*_\M$.
By \cite[Theorem 3.27]{EO} this functor is an equivalence. So (iv) implies
(v). The opposite implication is completely similar.
\end{proof}

\begin{remark}
\label{they are tensor}
%
In view of Proposition~\ref{inv criterion} invertible $(\C,\D)$-bimodule categories 
can be thought of as {\em Morita equivalences} $\C\to \D$.
\end{remark}


\begin{corollary}
\label{inv --> indec}
An invertible $(\C,\D)$-bimodule category is indecomposable
as both a left $\C$-module category and a right $\D$-module category.
\end{corollary}

\begin{definition} The {\it Brauer-Picard groupoid of fusion
categories} $\uuPic$ is a 3-groupoid, whose objects are
fusion categories, 1-morphisms from $\C$ to $\D$ 
are invertible $(\C,\D)$-bimodule categories, 2-morphisms are equivalences
of such bimodule categories, and 3-morphisms are isomorphisms of such
equivalences. 
\end{definition}

In other words, $\uuPic$ is the subcategory 
of ${\bf Bimodc}$ obtained by
extracting the invertible morphisms
at all levels. 

The 3-groupoid $\uuPic$ can be truncated  
(by forgetting 3-morphisms and identifying isomorphic 2-morphisms)
to a 2-groupoid $\uPic$, and further truncated 
(by forgetting 2-morphisms and identifying isomorphic 1-morphisms)
to a 1-groupoid (i.e., an ordinary groupoid) $\Pic$.

In particular, for every fusion category $\C$ we have 
the following hierarchy of objects: 

$\bullet$ the categorical $2$-group $\uuPic(\C)$ of automorphisms of $\C$ in
$\uuPic$ (which is obtained by extracting invertible objects and
morphisms at all levels from ${\bf Bimodc}(\C)$);

$\bullet$ the categorical group $\uPic(\C)$ 
of automorphisms of $\C$ in
$\uPic$;

$\bullet$ the group 
$\Pic(\C)$ 
of automorphisms of $\C$ in
$\Pic$, which we will call the 
{\it Brauer-Picard group} 
of $\C$.

\begin{remark}
\label{mortorsor}
For any pair of isomorphic objects in $\uPic(\C)$ the set of morphisms between them
is a torsor over the group ${\rm Inv}(\Z(\C))$ of invertible objects of $\Z(\C)$.
\end{remark}



\begin{remark} The 3-groupoid $\uuPic$ is a categorification of 
the 2-groupoid $\underline{\rm Pic}$, whose objects are rings, 
1-morphisms from $A$ to $B$ are invertible $(A,B)$-bimodules, and 
2-morphisms are isomorphisms of such bimodules. 
In particular, the Brauer-Picard group $\Pic(\C)$ 
is a categorical analog of the classical Picard group 
${\rm Pic}(A)$ of a ring $A$. 
\end{remark}

\begin{remark}
The terminology ``Brauer-Picard group'' is justified by the
following observation. 

For a moment, let ${\bold k}$ be any field (not necessarily
algebraically closed). 

\begin{proposition} $\Pic(\Vec_{\bold k})$ is isomorphic to the
classical Brauer group ${\rm Br}({\bold k})$. 
\end{proposition}

\begin{proof} 
First of all, note that bimodule categories over $\C=\Vec_{\bold k}$ 
is the same thing as module  categories. 

Let $R$ be a finite dimensional simple algebra over ${\bold k}$. Then  
$\M(R):=R$-mod is an indecomposable module category over $\C$. 
It is easy to see that ${\rm Mat}_N(R)$-mod is naturally
equivalent to $R$-mod as a module category, and 
that 
$$
\M(R)\boxtimes_\C \M(S)=\M(R\otimes_{\bold k} S).
$$
Also, any indecomposable semisimple module category 
over $\C$ is of the form $\M(R)$ for a finite dimensional simple 
${\bold k}$-algebra $R$, determined uniquely up to a Morita equivalence. 
This implies that $M(R)$ is invertible if and only if $R$ is
central simple, which implies the claim. 
\end{proof}
\end{remark}

\begin{proposition}\label{relp}
Let $\C_1$, $\C_2$ be two fusion categories of relatively prime Frobenius-Perron dimensions. 
Then 
$$
\Pic(\C_1\boxtimes \C_2)=\Pic(\C_1)\times \Pic(\C_2).
$$ 
\end{proposition}

\begin{proof}
This follows from \cite{ENO1}, Proposition 8.55.  
\end{proof}

\subsection{Integral bimodule categories}
\label{int bim}
 
Let $\C$ be an integral fusion cate\-gory, i.e., a cate\-gory such that the Frobenius-Perron
dimension of any simple object of $\C$ is an integer. Recall that the Frobenius-Perron
dimensions in module categories were defined in Section~\ref{FPmod}. It is clear
that the Frobenius-Perron dimension of any object in a $\C$-module cate\-gory
is a square root of an integer.

\begin{definition}
We will say that a $\C$-module category $\M$ is {\em integral} if 
$\FPdim(M)\in \mathbb{Z}$  for every object $M\in \M$. 
\end{definition}

Equivalently, $\M$ is integral  if 
for every simple object $M\in \M$ the number
$\FPdim(\underline\Hom(M,\, M))$ is the square of an integer.

Now let $\M$ be an invertible $\C$-bimodule category. To avoid possible confusion 
let us agree that we compute Frobenius-Perron dimensions in $\M$  by
regarding it as a one-sided (left or right) $\C$-module category. 
In particular, 
\[
\sum_{M\in \O(\\M)}\FPdim(M)^2 =\FPdim(\C).
\]

\begin{definition}
We will say that an invertible $\C$-bimodule category $\M$ is integral if it is integral
as one-sided (left or right) module category.
\end{definition}

It is clear that here the choice of left or right $\C$-module structure is not important
since the Frobenius-Perron dimensions in $\M$ defined using these structures coincide.

\begin{proposition}
Let $\C$ be an integral fusion category.
If $\M,\, \N$ are invertible integral $\C$-bimodule categories then $\M\bt_\C \N$ is integral.
\end{proposition}
\begin{proof}
It is easy to see that according to our conventions the canonical $\C$-bimodule functor
$F: \M \bt \N \to \M \bt_\C \N$ satisfies 
\[
\FPdim(F(M\bt N)) =\FPdim(M)\FPdim(N) \qquad \mbox{for all } M,N\in \M.
\]
Since this functor $F$ is surjective, we conclude that  $\M\bt_\C \N$ is integral. 
Indeed, no integer can be equal to a sum of non-integer square roots.
\end{proof}

It is easy to show that equivalence classes of invertible integral $\C$ bimodule categories form a normal subgroup 
of $\Pic(\C)$, denoted by $\Pic_+(\C)$, such that $\Pic(\C)/\Pic_+(\C)$ is an elementary abelian $2$-group.
It gives rise to a full categorical 2-subgroup $\uuPic_+(\C)$ of the categorical 2-group
$\uuPic(\C)$.

\subsection{The categorical 2-group of outer autoequivalences a fusion category}\label{outer}\label{quasitr}
	
Let $\C$ be a fusion category. 
Let us say that an invertible $(\C,\C)$-bimodule category $\M$ 
is {\it quasi-trivial} if it is equivalent to $\C$ as a left module category. 
It is easy to see that if $\M$ is quasi-trivial, then there exists a tensor autoequivalence 
$\phi: \C\to \C$, such that $\M=\C$ with the left action of $\C$ by left multiplication, 
and the right action of $\C$ by right multiplication twisted by $\phi$. Moreover, $\phi$ is 
uniquely determined up to composing with conjugation by an invertible object of $\C$. 
In other words, it is uniquely determined as an {\it outer autoequivalence}. 

Now define the categorical 2-group $\underline{\underline{\rm Out}}(\C)$ to be the 2-subgroup 
of $\uuPic(\C)$ which includes only the quasi-trivial invertible bimodule categories
(and all the corresponding equivalences and isomorphisms). This 2-group can be truncated to a 1-group 
$\underline{\rm Out}(\C)$ and further to the usual group 
${\rm Out}(\C)$, of isomorphism classes of outer tensor autoequivalences of $\C$
(i.e. autoequivalences modulo conjugations by invertible objects). 

\subsection{The Picard 2-groupoid of a braided fusion category}\label{braided Picard}

Let $\B$ be a braided fusion category. 
The monoidal 2-category ${\bf Modc}(\B)$ contains a categorical 
2-group $\underline{\underline{\rm Pic}}(\B)$, obtained by
extracting invertible objects and morphisms at all levels, which we will call 
{\it the Picard 2-group} of $\B$. 
This categorical 2-group is a categorical analog of 
the categorical 1-group of invertible modules over a commutative ring
$A$ (or, more generally, of the Picard 1-group, or groupoid, of a scheme). 
By truncating it one obtains a categorical 1-group 
$\underline{\rm Pic}(\B)$, and an ordinary group ${\rm Pic}(\B)$,
called {\it the Picard group} of the braided category $\B$. 

\begin{remark}
\label{uPic0}
If $\B$ is a braided fusion category 
then $\uuPic(\B)$ contains 
$\underline{\underline{\rm Pic}}(\B)$ as a full categorical
2-subgroup (of bimodule categories in which the left and right
action are related via the braiding). 
\end{remark}

\subsection{The 2-groupoid of equivalences}\label{2grpEq}

Following \cite{Ga}, we define the 2-groupoid $\underline{\rm Eq}$, whose objects are 
fusion categories, 1-morphisms are tensor equivalences,
and 2-morphisms are isomorphisms of such equivalences. 
It can be truncated to an ordinary groupoid 
${\rm Eq}$. So for every fusion category 
$\C$, we obtain the groupoid $\underline{\rm Eq}(\C)$ 
of tensor autoequivalences of $\C$, and the corresponding group 
${\rm Eq}(\C)$ of isomorphism classes of tensor autoequivalences
of $\C$. 

\subsection{The 2-groupoid of braided equivalences} 
\label{2grp EqBr}

Here is the braided version of the construction of the previous subsection.
We define the 2-groupoid $\uEqBr$, whose objects are 
braided fusion categories, 1-morphisms are braided equivalences,
and 2-morphisms are isomorphisms of such equivalences. 
It can be truncated to an ordinary groupoid 
$\EqBr$. So for every braided fusion category 
$\B$, we obtain the groupoid $\uEqBr(\B)$ 
of braided autoequivalences of $\B$, and the corresponding group 
$\EqBr(\B)$ of isomorphism classes of braided autoequivalences
of $\B$. 

\subsection{The finiteness theorem}

\begin{theorem}\label{finit}
The groups $\Pic(\C)$, ${\rm Out}(\C)$, ${\rm Eq}(\C)$, $\EqBr(\B)$, 
${\rm Pic}(\B)$ are finite. 
\end{theorem}

\begin{proof}
This follows from the finiteness results from \cite{ENO1}
(Theorem 2.31, Corollary 2.35).
\end{proof}


\section{Proof of Theorem~\ref{breq}}
\label{Sect 5}

It is sufficient to prove for every fusion category 
$\C$, the functor $\Z: \uPic(\C)\to \uEqBr(\Z(\C))$
is an equivalence.   

\subsection{A monoidal functor $\mathbf{\Phi : \uPic(\C) \to \uEqBr(\Z(\C))}$}
\label{Phi-section}

Let $\M$ be an indecomposable right $\C$-module category. Let $\C_\M^*$
denote the dual of $\C$ with respect to $\M$, i.e., the category of right
$\C$-module endofucnctors of $\M$. By \cite{O1}  $\C^*_\M$ is a fusion
category. We can regard $\M$ as a $\C^*_\M \bt \C^\rev$-module category.
Its $\C^*_\M \bt \C^\rev$-module endofunctors can be identified, on the one hand,
with functors of left multiplication by objects of  $\Z(\C_\M^*)$, and on the other hand,
with functors of right multiplication by objects of $\Z(\C)$. Combined, these identifications 
yield a canonical equivalence of braided categories 
\begin{equation}
\label{schauenburg}
\Z(\C)\xrightarrow{\sim} \Z(\C_\M^*).
\end{equation} 
This result is due to Schauenburg, see  \cite{S}.

Now suppose that $\M$ is an invertible $\C$-bimodule category. 
Let us view it as a right $\C$-module category.
By Proposition~\ref{inv criterion} and Remark~\ref{they are tensor}
we have an equivalence of tensor categories 
\begin{equation}
\label{invMor}
\C^*_\M \cong \C
\end{equation}
obtained by identifying right $\C$-module endofunctors of $\M$ 
with the functors of left multiplication by objects of $\C$.

Thus, we have a braided tensor equivalence
\begin{equation}
\label{Phi}
\Phi(\M) : \Z(\C) \xrightarrow{\sim} \Z(\C_\M^*)  \xrightarrow{\sim}  \Z(\C),
\end{equation}
where the first equivalence is \eqref{schauenburg} and the second one
is induced from \eqref{invMor}.

Clearly, a $\C$-bimodule equivalence between $\M,\,\N\in \uPic(\C)$ gives
rise to an isomorphism of  tensor  functors $\Phi(\M)$ and $\Phi(\N)$.

To see that the functor \eqref{Phi} is monoidal, observe that the  $\C$-bimodule 
functor of  right multiplication by an object $Z\in \Z(\C)$ on  $\M \bt_\C \N$
is isomorphic  to the well-defined functor of ``middle" multiplication by $\left(\Phi(\N)\right)(Z)$,
which, in turn,  is isomorphic to the functor of left multiplication by $\left(\Phi(\M)\circ \Phi(\N)\right)(Z)$.
This gives a natural isomorphism of tensor functors 
$\Phi(\M) \circ\Phi(\N) \cong \Phi(\M\bt_\C \N)$, i.e., a monoidal structure on~$\Phi$.

\subsection{A functor $\mathbf{\Psi :  \uEqBr(\Z(\C))  \to \uPic(\C) }$}
\label{Psi-section}

Let $\alpha$ be a braided tensor autoequivalence of $\Z(\C)$.  Below we recall
a construction of an invertible $\C$-bimodule category from $\alpha$ given
in \cite{ENO3}.

Let $F: \Z(\C)\to \C$ and $I: \C \to \Z(\C)$
denote the forgetful functor and its adjoint.
Given an algebra $A$ in $\C$ let $A-\mbox{mod}_\C$ and $A-\mbox{bimod}_\C$ 
denote, respectively, the categories of left $A$-modules and $A$-bimodules in $\C$.

The object $I(\be)$ is a commutative algebra in $\Z(\C)$ and so is  
\begin{equation}
\label{algebra L}
L:=\alpha^{-1}(I(\be)).
\end{equation} 
Furthermore, there is a tensor  equivalence
\begin{equation}
\label{D=L-mod}
\C  \xrightarrow{\sim}  L-\mbox{mod}_{\Z(\C)} : X \mapsto I(X).
\end{equation}

Note that $L$  is indecomposable in $\Z(\C)$ but 
might be decomposable as an algebra {\em in $\C$}, i.e., 
\[
L =\bigoplus_{i\in J} \, L_i,
\]
where $L_i,\, i\in J,$ are indecomposable algebras in $\C$ such that the 
multiplication of $L$ is zero on $L_i \ot L_j,\, i\neq j$ .
Here and below 
we abuse notation and write $L$ for an object of $\Z(\C)$ and its
forgetful image in $\C$.

For any $i\in J$ let 
\begin{equation}
\label{Psi-i}
\Psi_i(\alpha):= L_i-\mbox{mod}_\C.
\end{equation}
Clearly, it is a right $\C$-module category.
We would
like to show that $\Psi_i(\alpha)$ is, in fact, an invertible $\C$-bimodule category.


Consider the following commutative diagram of tensor functors:

\begin{equation}
\label{our dear diagram}
\xymatrix{
\Z(\C)  \ar[d]_{Z \mapsto L\ot Z} \ar[rr]^{Z\mapsto L_i\ot Z} & &
\Z(L_i-\text{bimod}_\C) \ar[d]^{F_i} \\
L-\text{mod}_{\Z(\C)}  \ar[r]^(.35){F} & \bigoplus\, L_i-\text{bimod}_\C \subset L-\text{bimod}_\C
\ar[r]^(.63){\pi_i} & L_i-\text{bimod}_\C.
}
\end{equation}
Here $F_i: \Z(L_i-\text{bimod}_\C)\to L_i-\text{bimod}_\C$ is the  forgetful functor 
 and $\pi_i$ is a projection from 
$L-\mbox{bimod}_\C =\oplus_{ij}\, (L_i-L_j)-\mbox{bimod}_\C$ to its
$(i,\,i)$ component. We have $\pi_i(L\ot X) = L_i \ot X$
for all $X\in \C$.  The top arrow is an equivalence and
the forgetful functor
$\Z(L_i-\mbox{bimod}_\C)\to L_i-\mbox{bimod}_\C$ (the right down arrow)
is surjective. Hence,  the composition  $G_i := \pi_i F$ of the functors
in the bottom row is surjective. But $G_i$ is a tensor functor between
fusion categories of equal Frobenius-Perron dimension and hence it is an 
equivalence by \cite[Proposition 2.20]{EO}.

In view of \eqref{D=L-mod} this gives a tensor equivalence between 
$\C$ and $\C_{\Psi_i(\alpha)}^*$. Hence, $\Psi_i(\alpha)$ is a
$\C$-bimodule category. It is easy to see that the above functor $G_i$
identifies  with \eqref{L} when $\M=\Psi_i(\alpha)$, 
therefore  $\Psi_i(\alpha)$ is invertible by Proposition~\ref{inv criterion}.

We claim that definition \eqref{Psi-i} does not depend on a choice of $i\in J$.

\begin{lemma}
\label{no choice}
For all $i,\,j\in J$ there is an equivalence of  
$\C$-bimodule categories $\Psi_i(\alpha)$ and $\Psi_j(\alpha)$.
\end{lemma} 
\begin{proof}
Let us consider the category $\D:=L-\text{mod}_\C$.  It is a {\em multifusion}
category in the  sense of \cite[Section 2.4]{ENO1}, i.e., it has a decomposition
\[
\D = \bigoplus_{ij\in J}\, \D_{ij}, 
\]
such that $\D_{ii}$ is a fusion category and $\D_{ij}$ is a $(\D_{ii},\D_{jj})$-bimodule
category for all $i,j\in J$. Furthermore for $X\in \D_{ij}$ and $Y\in \D_{kl}$
we have $X\ot Y \in \D_{il}$ if $j=k$ and $X\ot Y =0$ if $j\neq k$.

It follows from the result of Schauenburg \cite[Corollary 4.5]{S} that $\Z(\D)\cong \Vec$
as a tensor category. Therefore, $\D_{ij} \cong \Vec$ for all $i,j\in J$, i.e.,
simple objects of $\D$ can be labeled $E_{ij}$ in such a way
that the tensor product $\ot_L$ satisfies the usual matrix multiplication rules:
\[
E_{ij} \ot_L E_{kl} =\delta_{jk} E_{il},\quad i,j,k,l\in J.
\]
It follows that $L_i = E_{ii}$ and $L_i-\text{mod}_\C$ is spanned by $E_{ik},\,k\in J$.
Thus, the functor
\[
X \mapsto E_{ji}\ot_L X :  L_i-\text{mod}_\C \to L_j-\text{mod}_\C,\qquad i,j\in J
\]
is an equivalence of $\C$-bimodule categories. 
\end{proof}

Let us choose a $\C$-bimodule category $\Psi(\alpha)\in \uPic(\C)$ in the equivalence class
of $\C$-bimodule  categories $\Psi_i(\alpha),\, i\in J$.

Let $f:\alpha\xrightarrow{\sim} \alpha'$ be an isomorphism in $\uEqBr(\Z(\C))$. It gives
rise to an equivalence of the corresponding algebras $L,\, L'$ in $\Z(\C)$ and, 
consequently, to a $\C$-bimodule  equivalence 
$\Psi_i(f):  \Psi_i(\alpha) \xrightarrow{\sim} \Psi_i(\alpha')$.
By Lemma~\ref{no choice} we obtain a $\C$-bimodule  equivalence 
$\Psi(f):  \Psi(\alpha) \xrightarrow{\sim} \Psi(\alpha')$.

Thus, we have a functor
\begin{equation}
\label{Psi}
\Psi: \uEqBr(\Z(\C)) \to \uPic(\C).
\end{equation}
It remains to check that $\Psi$ is an inverse
of the monoidal functor $\Phi$ introduced in Section~\ref{Phi-section}.

\subsection{Equivalences $\mathbf{\Phi\circ\Psi \cong {\rm Id}_{\uEqBr(\Z(\C))}}$
and $\mathbf{\Psi\circ \Phi \cong {\rm Id}_{\uPic(\C)}}$}

First we prove an equivalence $\Phi\circ\Psi \cong {\rm Id}_{\uEqBr(\Z(\C))}$.
Given  $\alpha\in \uEqBr(\Z(\C))$ let $\M =\Psi(\alpha) \cong L_i -\text{mod}_\C$, where
the algebra $L_i$ is defined as in Section~\ref{Psi-section}. From
\eqref{Phi} we see that $\Phi(\M)$ is defined by
\[
\Phi(\M) : \Z(\C) \xrightarrow{Z\mapsto L_i\ot Z} \Z(L_i -\text{bimod}_\C) \xrightarrow{\iota} \Z(\C),
\]
where the second equivalence $\iota$ is induced from the inverse of the equivalence
in the bottom row of \eqref{our dear diagram}.  Since $\C \cong I(\be)-\text{mod}_{\Z(\C)}$
we have 
\[
\iota^{-1}(Z) = \pi_i F\alpha^{-1} I(Z) = \pi_i F\alpha^{-1} (I(\be)\ot Z)  =L_i \ot \alpha^{-1}(Z)
\]
for all $Z\in \Z(\C)$. Therefore,  $\Phi\circ \Psi(\alpha) \cong \alpha$. 

Next, we prove that $\Psi\circ \Phi \cong {\rm Id}_{\uPic(\C)}$.
Take $\M \in \uPic(\C)$.  Let $A\in \C$ be an algebra such that $\M\cong A-\text{mod}_\C$
as a right $\C$-module category.  Since $\M$ is invertible, we have an equivalence
$\C \cong A-\text{bimod}_\C$ by Proposition~\ref{inv criterion}.

Construct a braided autoequivalence $\alpha:=\Phi(\M)\in \uEqBr(\Z(\C))$ as in  \eqref{Phi}. 
Upon the identification $\Z(\C) \cong \Z(A-\text{bimod}_\C)$ we have
\[
\alpha(Z)= A \ot Z,\quad Z\in \Z(\C),
\]
where $A\ot Z$ has an obvious  structure of a central object in the category of $A$-bimodules.
So the algebra $L$ in $\Z(\C)$ defined by \eqref{algebra L} is identified with 
the algebra $A\ot I(\be)$ in $\Z(A-\text{bimod}_\C)$.   Hence, the category of $L-\text{mod}_\C$
is identified with the category $\widetilde{\M}$ of $A-I(\be)$-bimodules in $\C$ (recall that
$I(\be)$ is a commutative algebra in $\Z(\C)$).  Indecomposable components
of $\widetilde{\M}$ are equivalent to $A-\text{mod}_\C$ and so they are identified with $\M$
as $\C$-bimodule categories, i.e.,  $\Psi\circ\Phi(\M) \cong \M$, as required. 


It is easy to check that $\Phi$ and $\Psi$ are bijective on morphisms (cf.\ Remark~\ref{mortorsor}).

This completes the proof of Theorem~\ref{breq}.


\subsection{Generalization} Let $\B$ be a non-degenerate braided fusion cate\-gory 
(see \cite[Definition 2.28]{DGNO}). By \cite[Proposition 3.7]{DGNO} this means that the braiding 
on $\B$ induces an
equivalence $\B\bt \B^{\rev}\simeq \Z (\B)$. 
Now let $\M$ be an invertible module category over $\B$ (see Section~\ref{braided Picard}) and let 
$\B^*_\M=\Fun_\B(\M,\M)$. Combining the equivalence above with \eqref{schauenburg} we get
an equivalence $\B \bt \B^{\rev}\simeq \Z (\B_\M^*)$. The compositions
\begin{equation}
\alpha_+: \B =\B \bt \be \subset \B\bt \B^{\rev}\simeq \Z(\B_\M^*)\to \B_\M^*
\end{equation}
and 
 \begin{equation}
\alpha_-: \B =\be \bt \B^{\rev} \subset \B\bt \B^{\rev}\simeq \Z(\B_\M^*)\to \B_\M^*
\end{equation}
are called {\em alpha-induction functors}, see e.g. \cite{O1}. Proposition \ref{inv criterion} says that invertibility
of $\M$ is equivalent to $\alpha_+$ and $\alpha_-$ being tensor equivalences. 
Thus $$\alpha_+=\alpha_-\circ \theta_\M$$ where $\theta_\M:\B \to \B$ is an autoequivalence. One verifies
directly that $\theta_\M$ is actually a {\em braided} autoequivalence of $\B$. 
Furthermore, the same argument as the one in the end of Section~\ref{Phi-section} 
shows that  $\theta_\M$ naturally extends to a functor ${\underline {\rm Pic}}(\B)\to \uEqBr(\B)$.

Conversely, let $\gamma \in \EqBr(\B)$. Then $\id \bt \gamma \in \EqBr(\B \bt \B^\rev)=\EqBr(\Z(\B))$.
Thus Theorem \ref{breq} assigns to $\gamma$ an invertible $\B-$bimodule category $\M_\gamma$.
It follows
immediately from definitions that right and left actions of $\B$ on $\M_\gamma$ are related
by the braiding, so $\M_\gamma$ is an invertible module category over $\B$. It is clear that
this assignment $\gamma \mapsto \M_\gamma$ extends naturally to a functor
$\uEqBr(\B)\to {\underline {\rm Pic}}(\B)$. A careful examination of the constructions involved 
shows the following result:

\begin{theorem} \label{genbreq}
For a non-degenerate braided fusion category $\B$ the functors above are
mutually inverse equivalences of ${\underline {\rm Pic}}(\B)$ and {\em $\uEqBr(\B)$}.
\end{theorem} 
Details of the proof  of Theorem~\ref{genbreq} will be given in a subsequent article.

\begin{remark} 
\begin{enumerate}
\item[(i)] We notice that the construction of $\theta_\M$ above makes sense for arbitrary
braided fusion category $\B$, see \cite{O1}. Thus, we have a monoidal functor
\begin{equation}
\label{we formally agree with the government policy concerning hallucinogenic mushrooms}
\Theta: \underline{\rm Pic}(\B)\to \uEqBr(\B) : \M \mapsto \theta_\M.
\end{equation}
However it is clear that \eqref{we formally agree with the government policy concerning hallucinogenic mushrooms} 
does not produce an equivalence as in 
Theorem \ref{genbreq}. For example it is clear that for a symmetric braided fusion 
category $\alpha_+=\alpha_-$
for any $\M$, so $\theta_\M=\id_\B$ for any $\M$ in this case.
\item[(ii)] For a fusion category $\C$ the braided category $\Z(\C)$ is non-degene\-rate, see 
\cite[Corollary 3.9]{DGNO}.
Thus combining Theorem \ref{breq} and Proposition \ref{bimod =modZC} we get an equivalence 
${\underline {\rm Pic}}(\Z(\C))\simeq \uEqBr(\Z(\C))$ in this case. One verifies that
this equivalence and equivalence from Theorem \ref{genbreq} are canonically identified.
\end{enumerate}
\end{remark}

\begin{remark}
\label{conceptual}
Given a braided category $\B$ 
we have  a monoidal functor $\Theta: \underline{\rm Pic}(\B)\to \uEqBr(\B)$ given by 
\eqref{we formally agree with the government policy concerning hallucinogenic mushrooms}. 
Recall that in Section~\ref{Phi-section} we constructed a monoidal equivalence 
$\Phi : \uPic(\C) \to \uEqBr(\Z(\C))$  for any fusion category $\C$. The
following conceptual explanation of these functors were suggested to us by
V.~Drinfeld.  

Namely, let $\A$ be a monoidal $2$-category (see \cite{KV}).
Then the monoidal category $\underline{\End}(\be_\A)$ of endofunctors of the unit object of $\A$
has a canonical structure of a braided category (this is a higher categorical version of the well known fact
that endomorphisms of the unit object in a monoidal category form a commutative monoid).
The categorical group $\A^\times$ of invertible objects of $\A$ acts on $\underline{\End}(\be_\A)$
by tensor conjugation. Hence, we have a monoidal functor
\begin{equation}
\label{concept functor}
\A^\times \to \EqBr(\underline{\End}(\be_\A)).
\end{equation}
\end{remark}
For $\A =\mathbf{Bimodc}(\C)$, the monoidal $2$-category of $\C$-bimodule categories
over a fusion category $\C$, one has $\underline{\End}(\be_\A) = \Z(\C)$ and the above
functor \eqref{concept functor} is  precisely  the functor $\Phi : \uPic(\C) \to \uEqBr(\Z(\C))$ 
from Section~\ref{Phi-section}. For $\A= \mathbf{Modc}(\B)$, the monoidal $2$-category of module 
categories   over a braided fusion category $\B$, it gives the functor 
\eqref{we formally agree with the government policy concerning hallucinogenic mushrooms}.

\subsection{The truncation of the categorical 2-group of outer autoequivalences
of a fusion category}\label{ximap}

For any fusion category $\C$, we have a natural homomorphism of categorical groups 
$\xi: \underline{\rm Eq}(\C)\to \underline{\underline{\rm Out}}(\C)$, attaching to every tensor autoequivalence 
its class of outer autoequivalences. 

\begin{proposition}\label{xiiso}
If $\C$ has no nontrivial invertible objects, then 
$\xi$ is an isomorphism of $\underline{\rm Eq}(\C)$ onto the truncation 
$\underline{\rm Out}(\C)$.
\end{proposition}

\begin{proof} 
Let $\M$ be a quasi-trivial invertible bimodule category over $\C$. 
Then there exists a unique, up to an isomorphism, equivalence of left module categories 
$\C\to \M$, so we may assume that $\M=\C$ as a left module category. Then 
the right action of $\C$ is given by some uniquely determined autoequivalence $\phi$.
Then we can define $\xi^{-1}(\M)=\phi$. 
\end{proof} 

\section{Invertibility of components of graded fusion categories}
\label{Sect6}

Let $G$ be a finite group and let
\[
\C = \bigoplus_{g\in G}\, \C_g
\]
be a graded fusion category, cf.\ Section~\ref{graded TC}.
The trivial component $\C_e$ is a tensor subcategory of $\C$,
and each $\C_g$ is a $\C_e$-bimodule category. It follows that for all  $g,h\in G$
the tensor product of $\C$ restricts to a $\C_e$-balanced bifunctor  
\begin{equation}
\label{def M0}
\otimes: \C_g \times \C_h \to \C_{gh},
\end{equation}
which gives rise to a functor
\begin{equation}
\label{def M}
M_{g,h} : \C_g \boxtimes_{\C_e} \C_h \to \C_{gh}.
\end{equation}

\begin{theorem}
\label{component invertibility}
Let $\C= \oplus_{g\in G}\, \C_g$ be a $G$-extension. Then:
\begin{enumerate}
\item[(i)] each  $\C_g,\, g\in G,$ is an invertible $\C_e$-bimodule category;
\item[(ii)]  the functor $M_{g,h} : \C_g \boxtimes_{\C_e} \C_h \to \C_{gh},\, g,h\in G,$
is an equivalence of  $\C_e$-bimodule categories. 
\end{enumerate}
\end{theorem}
\begin{proof}
For each $g\in G$ let us pick a non-zero object $Y_g$ in $\C_g$. 
Then $A_g = Y_g \ot Y_g^*$ is an algebra in $\C_e$ (and, therefore, in $\C$).  
By \cite{EO, O1} the regular left $\C$-module category $\C$ is equivalent
to the category of right $A_g$-modules in $\C$, and the left $\C_e$-module 
category $\C_g$ is equivalent to the category of right $A_g$-modules in $\C_e$.
Furthermore, there are tensor 
equivalences
\[
F_g : \C   \xrightarrow{\sim} A_g\mbox{-bimodules in } \C : X \mapsto Y_g \ot X \ot Y_g^*,
\qquad g\in G.
\]
Let $R_g,\, g\in G$ denote the restriction of $F_g$ to $\C_e$. It establishes
a tensor equivalence 
\[
R_g : \C_e  \xrightarrow{\sim}  A_g\mbox{-bimodules in } \C_e \cong 
\Fun_{\C_e}(\C_g,\, \C_g).
\]
It is straightforward to see that $R_g$  coincides with functor defined in \eqref{R}.
Passing from right to left $A_g$-modules, one similarly obtains an equivalence
$L_g: \C_e^\rev  \xrightarrow{\sim} \Fun(\C_g,\, \C_g)_{\C_e}$.
By Proposition~\ref{inv criterion}, $\C_g$ is an invertible $\C_e$-bimodule category. 
This proves (i).

To prove (ii), note that tensor
equivalences $F_g,\, g\in G,$ make the category of $(A_g,A_h)$-bimodules in $\C$
into a $\C$-bimodule category, with the left (respectively, right) action of an object 
$X$ in $\C$ by multiplication by $F_g(X)$  (respectively, by $F_h(X)$).
Thus we have $\C$-bimodule equivalences
\[
F_{g,h}: \C \cong  (A_g-A_h)\mbox{-bimodules in } \C : X \mapsto Y_g \ot X \ot Y_h^*,
\qquad g,h\in G.
\]
Therefore, the restriction of $F_{g^{-1},h}$ to $\C_{gh}$ establishes
a  $\C_e$-bimodule equivalence between $\C_{g h}$ and the category of
$(A_{g^{-1}},A_h)$-bimodules in  $\C_e$. The latter category is equivalent to 
$\Fun_{\C_e}(\C_{g^{-1}},\, \C_h) \cong \C_g  \boxtimes_{\C_e} \C_h$.

Thus, we have constructed a $\C_e$-bimodule equivalence
\[
 \C_g \boxtimes_{\C_e} \C_h \to \C_{gh},\qquad g,h\in G.
\]
It is easy to see that it coincides with the functor \eqref{def M} induced by
the $\C_e$-balanced bifunctor $\otimes: \C_g \times \C_h \to \C_{gh}$.
Indeed, both functors are identified with
\[
\C_{gh} \to \Fun_{\C_e}(\C_{g^{-1}},\, \C_h) : X \mapsto\,  ? \ot X,
\]
so the proof is complete.
\end{proof}

\begin{corollary}
\label{dual is center}
The dual category of $\C_e\bt \C_e^\rev$ with respect to each  
$\C_e$-bimodule category $\C_g,\, g\in G$ is equivalent
to the center $\Z(\C_e)$ of $\C_e$: 
\begin{equation}
(\C_e\bt \C_e^\rev)^*_{\C_g} \cong \Z(\C_e).
\end{equation}
\end{corollary}
\begin{proof}
This follows by \cite[Theorem 3.34]{EO} since $(\C_e)^*_{\C_g} \cong \C_e^\rev$
by Theorem~\ref{component invertibility}.
\end{proof}


Thus, a $G$-extension $\C$ defines a group homomorphism 
\[
c: G \to \Pic(\C_e).
\]
The tensor product and associator of $\C$ give rise to 
an additional data which we will investigate next.

\section{Classification of extensions (topological version)}
\label{classification topological}

\subsection{The classifying space of a categorical $n$-group} 

It is well known that any categorical $n$-group ${\mathcal G}$ 
gives rise to a (connected) classifying space $B{\mathcal G}$
(well defined up to homotopy), 
which determines the equivalence class of ${\mathcal G}$ uniquely 
(so that $B{\mathcal G}$ carries the same information 
as ${\mathcal G}$). Moreover, the homotopy groups 
of $B{\mathcal G}$ are as follows: 
$\pi_i(B{\mathcal G})={\rm Mor}_{i+1}(X_i,X_i)$ 
for any $i$-morphism $X_i$ for $i=1,...,n+1$, and zero 
if $i\ge n+2$. 

A convenient model for the 
space $B{\mathcal G}$ is the simplicial complex 
given by the well known ``nerve''
construction. For the convenience of the readers, we 
recall this construction in the case of $n=2$ (which is 
the highest value of $n$ we will need). 
For brevity we omit associativity isomorphisms. 

{\bf Step 0.} We start with one  
0-simplex. 

{\bf Step 1.} For every  
isomorphism class $x$ of objects of ${\mathcal G}$, 
we pick an object representing
$x$ (which we also call $x$, abusing the notation) 
and add a 1-simplex $s_x$. 

{\bf Step 2.} For every 
isomorphism classes of objects $x_1,x_2$ and an 
isomorphism class of 1-morphisms $f: x_1\otimes x_2\to x_1x_2$,
where $x_1x_2$ is the representative of $x_1\otimes x_2$
chosen in the previous step,
we pick a 1-morphism representing $f$ 
(which we also call $f$, abusing the notation),  
and add a 2-simplex $s_f$, such
that $\partial s_f=s_{x_1}+s_{x_2}-s_{x_1x_2}$.

{\bf Step 3.} For each isomorphism classes of objects 
$x_1,x_2,x_3$, isomorphism classes of 1-morphisms
$f_{1,2}: x_1\otimes x_2\to x_1x_2$, 
$f_{2,3}: x_2\otimes x_3\to x_2x_3$, 
$f_{12,3}: x_1x_2\otimes x_3\to x_1x_2x_3$, 
$f_{12,3}: x_1\otimes x_2x_3\to x_1x_2x_3$, 
where $x_1x_2, x_1x_2x_3,$ etc., are representatives
of tensor products chosen in step 1, 
and a 2-morphism 
$$
g: f_{12,3}\circ (f_{1,2}\otimes \id_3)\to f_{1,23}\circ
(\id_1\otimes f_{2,3})
$$
we add a 3-simplex $s_g$ such that 
$\partial
s_g=s_{f_{1,2}}-s_{f_{1,23}}+s_{f_{12,3}}-s_{f_{2,3}}$. 

{\bf Step 4.} Given isomorphism classes of objects 
$x_1,x_2,x_3,x_4$, isomorphism classes of 1-morphisms
$f_{1,2}: x_1\otimes x_2\to x_1x_2$, 
$f_{2,3}: x_2\otimes x_3\to x_2x_3$, 
$f_{3,4}: x_3\otimes x_4\to x_3x_4$, 
$f_{12,3}: x_1x_2\otimes x_3\to x_1x_2x_3$, 
$f_{1,23}: x_1\otimes x_2x_3\to x_1x_2x_3$, 
$f_{23,4}: x_2x_3\otimes x_4\to x_2x_3x_4$, 
$f_{2,34}: x_2\otimes x_3x_4\to x_2x_3x_4$, 
$f_{1,234}: x_1\otimes x_2x_3x_4\to x_1x_2x_3x_4$, 
$f_{123,4}: x_1x_2x_3\otimes x_4\to x_1x_2x_3x_4$, 
$f_{12,34}:x_1x_2\otimes x_3x_4\to x_1x_2x_3x_4$, 
and 2-morphisms 
$$
g_{1,2,3}: f_{12,3}\circ (f_{1,2}\otimes \id_3)\to f_{1,23}\circ
(\id_1\otimes f_{2,3}),
$$
$$
g_{2,3,4}: f_{23,4}\circ (f_{2,3}\otimes \id_4)\to f_{2,34}\circ
(\id_2\otimes f_{3,4}),
$$
$$
g_{1,23,4}: f_{123,4}\circ (f_{1,23}\otimes \id_4)\to f_{1,234}\circ
(\id_1\otimes f_{23,4}),
$$
$$
g_{12,3,4}: f_{123,4}\circ (f_{12,3}\otimes \id_4)\to f_{12,34}\circ
(\id_{12}\otimes f_{3,4}),
$$
$$
g_{1,2,34}: f_{12,34}\circ (f_{1,2}\otimes \id_{34})\to f_{1,234}\circ
(\id_{1}\otimes f_{2,34}),
$$
such that 
$$
(1_1\otimes g_{2,3,4})\circ g_{1,23,4}
\circ (g_{1,2,3}\otimes 1_4)=g_{1,2,34}\circ g_{12,3,4}, 
$$
we add a single 4-simplex $s$ whose boundary is 
$$
\partial
s=s_{g_{1,2,3}}-s_{g_{1,2,34}}+s_{g_{1,23,4}}-s_{g_{12,3,4}}+
s_{g_{2,3,4}}.
$$

{\bf Step ${k}$, $k\ge 5$.} 
Any boundary of a $k$-simplex, $k\ge 5$, is filled in with a 
$k$-simplex. 

Note that the obtained model is a Kan complex. 

\subsection{Homotopy groups of classifying spaces of higher
groupoids attached to fusion categories}

\begin{proposition}\label{homot}
Let $\C$ be a fusion category and let $\uuPic(\C)$ be its 
Brauer-Picard $2$-group introduced in Section~\ref{inv+BP}. We have:
\begin{enumerate} 
\item[(i)] $\pi_1(B\uuPic(\C))=\Pic(\C)$;
\item[(ii)] $\pi_2(B\uuPic(\C))={\rm Inv}(\Z(\C))$, the group of
isomorphism classes of invertible objects in the Drinfeld center
of $\C$;
\item[(iii)] $\pi_3(B\uuPic(\C))={\bold k}^\times$;
\item[(iv)] $\pi_i(B\uuPic(\C))=0$ for all $i\ge 4$. 
\end{enumerate}
\end{proposition}
\begin{proof}
(i) is clear. To prove (ii), we need to calculate the
group of equivalence classes of automorphisms of any object. Take
the object $\C$ regarded as a $\C$-bimodule. Its endomorphisms 
as a $\C$-bimodule is the dual category to $\C$ with respect to
$\C\boxtimes \C^{\rm rev}$, so it is $\Z(\C)$ \cite[Corollary 3.37.]{EO}. 
Thus the automorphisms are the invertible objects in $\Z(\C)$. To prove (iii), 
we need to compute the group of automorphisms of
any 1-morphism. Take this 1-morphism to be the neutral object in
$\Z(\C)$. Then the group of automorphisms is ${\bold k}^\times$. 
(iv) is clear, since by construction we have killed all
the homotopy groups of degree $\ge 4$. 
\end{proof} 

\begin{proposition}\label{homot11}
Let $\C$ be a fusion category, and let $\underline{\underline{\rm Out}}(\C)$
its categorical $2$-group of outer autoequivalences introduced in Section~\ref{outer}. We have
\begin{enumerate}
\item[(i)] $\pi_1(B\underline{\underline{\rm Out}}(\C))={\rm Out}(\C)$,
\item[(ii)] $\pi_2(B\underline{\underline{\rm Out}}(\C))={\rm Inv}(\Z(\C))$, 
\item[(iii)] $\pi_3(B\underline{\underline{\rm Out}}(\C))={\bold k}^\times$,
\item[(iv)] $\pi_i(B\underline{\underline{\rm Pic}}(\C))=0$ for $i\ge 4$.
\end{enumerate}
\end{proposition}

\begin{proposition}\label{homot1}
Let $\B$ be a braided fusion category, and let $\underline{\underline{\rm Pic}}(\B)$
its Picard $2$-group  introduced in Section~\ref{braided Picard}. We have
\begin{enumerate}
\item[(i)] $\pi_1(B\underline{\underline{\rm Pic}}(\B))={\rm Pic}(\B)$,
\item[(ii)] $\pi_2(B\underline{\underline{\rm Pic}}(\B))={\rm Inv}(\B)$, the group of
isomorphism classes of invertible objects of $\B$,
\item[(iii)] $\pi_3(B\underline{\underline{\rm Pic}}(\B))={\bold k}^\times$,
\item[(iv)] $\pi_i(B\underline{\underline{\rm Pic}}(\B))=0$ for $i\ge 4$.
\end{enumerate}
\end{proposition}

\begin{proposition}\label{homot22}
Let $\C$ be a fusion category, and let  $\underline{\rm Eq}(\C)$ be the 
categorical group of autoequivalences of $\C$ introduced in Section~\ref{2grpEq}. We have
\begin{enumerate}
\item[(i)] $\pi_1(B\underline{\rm Eq}(\C))={\rm Eq}(\C)$,
\item[(ii)] $\pi_2(B\underline{\rm Eq}(\C))={\rm Aut}_{\otimes}({\rm Id}_\C)$, the group of
tensor isomorphisms of the identity functor of $\C$,  
\item[(iii)] $\pi_i(B\underline{\rm Eq}(\C))=0$ for $i\ge 3$. 
\end{enumerate}
\end{proposition}

\begin{proposition}\label{homot2}
Let $\B$ be a braided fusion category, and let  $\underline{\rm EqBr}(\B)$ be the 
categorical group of braided autoequivalences of $\B$ introduced in Section~\ref{2grp EqBr}. We have
\begin{enumerate}
\item[(i)] $\pi_1(B\underline{\rm EqBr}(\B))={\rm EqBr}(\B)$,
\item[(ii)] $\pi_2(B\underline{\rm EqBr}(\B))={\rm Aut}_{\otimes}({\rm Id}_\B)$,  
\item[(iii)] $\pi_i(B\underline{\rm EqBr}(\B))=0$ for $i\ge 3$. 
\end{enumerate}
\end{proposition}

The proofs of Propositions~\ref{homot11}, \ref{homot1}, \ref{homot22} and \ref{homot2} are analogous to that of Proposition \ref{homot}.

\subsection{The Whitehead half-square and the braiding}

Recall that for any $i,j>1$ we have the Whitehead bracket 
$[,]: \pi_i\times \pi_j\to \pi_{i+j-1}$ on homotopy groups 
of any topological space. Also, since there is a 
map $S^3\to S^2$ of Hopf invariant 1, we have 
the Whitehead half-square map $W: \pi_2\to \pi_3$ such that 
$W(x+y)-W(x)-W(y)=[x,y]$. 

The following proposition was pointed out to us by V. Drinfeld. 

\begin{proposition}\label{whit} Let $\B$ be a braided fusion category. 
For the space $B\underline{\underline{Pic}}(\B)$, the
Whitehead half-square map $W: \pi_2\to \pi_3$ 
is given by the braiding $c_{ZZ}$ on invertible objects $Z\in \B$. 
Therefore, the Whitehead bracket 
$[\, ,\,]: \pi_2\times \pi_2\to \pi_3$ 
coincides with the squared braiding $c_{ZY}c_{YZ}$ on invertible objects
$Y,Z\in \B$. 
\end{proposition}
\begin{proof}
For any pointed space $X$, the fundamental groupoid
of the double loop space $\Omega^2(X)$ is a braided monoidal category (see \cite[Subsection 2.3]{Q}). 
Note that $\pi_2(X)=\pi_0(\Omega^2(X))$ and $\pi_3(X)=\pi_1(\Omega^2(X))$.
So, the map $Z \mapsto c_{ZZ}$, where $c$ denotes the braiding of the above category, defines a map
$\pi_2(X)\to \pi_3(X)$. We claim that this map is the Whitehead half-square map. 
To prove this, it suffices to treat the universal
example $X = S^2$. That is, one needs to show that 
for $X=S^2$, the map in question is the  map 
$\Bbb Z=\pi_2(S^2)\to \Bbb Z=\pi_3(S^2)$ given by $n\mapsto n^2$. This is done by a straightforward 
verification. 
\end{proof}

In particular, taking $\B=\Z(\C)$, we find that the Whitehead half-square map and the Whitehead bracket 
for $B\uuPic(\C)$ are given by the braiding on invertible objects of $\Z(\C)$. 

{\bf Remark.} 
Proposition \ref{whit} can also be derived from 
\cite{B}, Chapter IV. 

\subsection{Classification of extensions}

Now we would like to classify $G$-extensions $\C$ 
of a given fusion category $\D$. 
As we have seen in Theorem~\ref{component invertibility}, such a category 
necessarily defines  a group homomorphism $c: G \to \Pic(\D)$. We would  like to study 
additional data and conditions on them that define a category
$\C$ given a homomorphism $c$.  

\begin{theorem}\label{eqcl}
Equivalence classes of $G$-extensions $\C$ of $\D$ are in bijection
with morphisms of categorical 2-groups 
$G\to \uuPic(\D)$, or, equivalently, with 
homotopy classes of maps between their classifying spaces: 
$BG\to B\uuPic(\D)$. 
\end{theorem}
\begin{proof}
Let us consider what it takes to define a continuous map 
$\xi: BG\to B\uuPic(\D)$, using the simplicial model of
$\uuPic(\D)$ described above. Note that since our model of 
$B\uuPic(\D)$ is a Kan complex, any map $\xi$ is homotopic to a
simplicial map, so it suffices to restrict our attention to
simplicial maps (which we will do from now on).  

{\bf Step 1.} Defining the map $\xi$ at the level of 1-skeletons 
(up to homotopy) obviously amounts to a choice of a set-theoretical map 
of fundamental groups $c: G\to \Pic(\D)$. On the categorical side, this is just a choice of 
an assignment $g\mapsto c(g)=\C_g$, $g\in G$, where $\C_g$ is an invertible 
bimodule category over $\D$. They can be combined into a single $\D$-bimodule category 
$\C=\oplus_g \C_g$. 

{\bf Step 2.} Extendability of this $\xi$ to the level of 2-skeletons
amounts to the condition that $c$ is a group homomorphism.
On the categorical side, this means that one has equivalences
$\C_g\boxtimes_\D \C_h\cong \C_{gh}$, and in particular
$\C_e\cong \D$. 

Next, any choice of an extension of $\xi$ to the level of 2-skeletons
amounts to picking the equivalences $M_{g,h}: \C_g\boxtimes_\D \C_h\to
\C_{gh}$, which defines a functor $\otimes$ of tensor multiplication on $\C$. 

{\bf Step 3.} Further, extendability of such a $\xi$ to the level of 3-skeletons amounts, on the categorical side, 
to the condition that there exists a functorial isomorphism 
$$
\alpha: (\bullet \otimes \bullet)\otimes \bullet \to \bullet \otimes (\bullet \otimes \bullet)
$$
(respecting the $\D$-bimodule structure, but 
not necessarily satisfying the pentagon relation), and once a good 
$M$ (for which $\alpha$ exists) has been fixed, the freedom of choosing an extension of $\xi$ to the level of 3-skeletons 
is a choice of $\alpha$. 

{\bf Step 4.} Finally, once $\xi$ has been extended to 3-skeletons, its extendability to the level of 4-skeletons 
amounts to the condition that $\alpha$ satisfies the pentagon relation. 
Once such an $\alpha$ has been fixed, there is a unique extension of $\xi$ to the level of 4-skeletons.

{\bf Step 5.} Once $\xi$ has been extended to a map of 4-skeletons, it canonically extends to a 
map of skeletons of all dimensions. 

The theorem is proved. 
\end{proof} 

\subsection{Proof of Theorem \ref{extcl}}

Theorem \ref{extcl} follows from Theorem \ref{eqcl} and classical obstruction theory 
in algebraic topology. Let us describe this derivation in more detail. 
For brevity we denote the homotopy 
groups $B\uuPic(\D)$ just by $\pi_i$, without specifying the space. 

Let us go back to the proof of Theorem \ref{eqcl}. 
At Step 3, it may be necessary to modify $M=(M_{g,h})$ to secure the existence of $\alpha$. 
But even if we allow modifications of $M$, there is an obstruction $O_3(c)$ to the 
existence of $\alpha$. Let us discuss the nature of this obstruction. 

If we have a map $\xi$ of 2-skeletons, then the condition for this map to be extendable to 3-skeletons
is that for every 3-simplex $\sigma\subset BG$, $\xi(\partial \sigma)$ represents the trivial element 
in the group $\pi_2$. Thus we get an obstruction which is a 3-cochain of $G$ with values in 
$\pi_2$. It is easy to see that this 3-cochain is actually a cocycle (where $G$ acts on $\pi_2$ via the homomorphism $c$), i.e. 
we get an obstruction $\psi\in Z^3(G,\pi_2)$. Now, $M$ can be modified by adding a 2-cochain 
$\chi$ on $G$ with coefficients in $\pi_2$, and this modification replaces $\psi$ with $\psi+d\chi$. 
This implies that the actual obstruction to extending $\xi$ to 3-skeletons (allowing the modifications of $M$) is 
the cohomology class $[\psi]=O_3(c)\in H^3(G,\pi_2)$. 

If the obstruction $O_3(c)$ vanishes, then, as we see from the above, the freedom of choosing $M$ so that $\xi$ is extendable 
to 3-skeletons is in $H^2(G,\pi_2)$. That is, we can modify $M$ by adding a cocycle $\chi\in Z^2(G,\pi_2)$, but 
if $\chi$ is a coboundary, then the homotopy class of the extension does not change.

Further, at Step 4, there is an obstruction $O_4(c,M)$ to choosing $\alpha$. 
Let us discuss its nature. 

If we have a map $\xi$ of 3-skeletons, then the condition for this map to be extendable to 4-skeletons
is that for every 4-simplex $\sigma\subset BG$, $\xi(\partial \sigma)$ represents the trivial element 
in the group $\pi_3$. Thus we get an obstruction which is a 4-cochain of $G$ with values in 
$\pi_3$. It is easy to see that this 4-cochain is actually a cocycle, i.e. 
we get an obstruction $\eta\in Z^4(G,\pi_3)$. Now, $\alpha$ can be modified by adding a 3-cochain 
$\theta$ on $G$ with coefficients in $\pi_3$, and this modification replaces $\eta$ with $\eta+d\theta$. 
This implies that the actual obstruction to extending $\xi$ to 4-skeletons (allowing the modifications of $\alpha$) is 
the cohomology class $[\eta]=O_4(c,M)\in H^4(G,\pi_3)$. 

If the obstruction $O_4(c,M)$ vanishes, then, as we see from the above, the freedom of choosing $\alpha$ so that $\xi$ is extendable 
to 4-skeletons is in $H^3(G,\pi_3)$. That is, we can modify $\alpha$ by adding a cocycle $\theta\in Z^3(G,\pi_3)$, but 
if $\theta$ is a coboundary, then the homotopy class of the extension does not change.

This proves Theorem \ref{extcl}.

\subsection{Classification of group actions on fusion categories}

\begin{proposition}\label{groupbrai} 
(i) Actions of a group $G$ by autoequivalences 
of a fusion category $\C$, up to an isomorphism, 
are in natural bijection with homotopy classes of mappings 
$BG\to B\underline{\rm Eq}(\C)$.  

(ii) Actions of a group $G$ by braided autoequivalences 
of a braided fusion category $\B$, up to an isomorphism, 
are in natural bijection with homotopy classes of mappings 
$BG\to B\uEqBr(\B)$.  
\end{proposition} 

\begin{proof} (i) We argue as in the previous subsection. 
Namely, a map between 2-skeletons of the spaces in question
is the same thing as an assignment $g\to F_g$, which attaches to every $g\in G$ 
a tensor equivalence $F_g: \C\to \C$, and a collection of
functorial isomorphisms $\eta_{g,h}: F_g\circ F_h\to F_{gh}$, 
$g,h\in G$. This map is extendable to 3-skeletons 
if and only if $\eta_{g,h}$ satisfies the 2-cocycle condition,
i.e. if the data $(F_g,\eta_{g,h})$ is an action of $G$ on $\C$
by tensor autoequivalences. Note that the extension to
3-skeletons is unique if exists, and extends uniquely to 
skeletons of higher dimensions. So (i) is proved.

(ii) is proved similarly.  
\end{proof} 

\begin{corollary}\label{braiparam} 
(i) (\cite[Theorem 5.5]{Ga}) Actions of a group $G$ by tensor autoequivalences 
of a fusion category $\C$, up to an isomorphism, 
are parametrized by pairs $(c,\eta)$, where $c: G\to {\rm Eq}(\B)$ is a homomorphism, 
such that the corresponding first obstruction $O_3(c)\in
H^3(G,{\rm Aut}_{\otimes}({\rm Id}_\C))$ vanishes, and 
$\eta=(\eta_{g,h})$ is the equivalence class of the identification $F_g\circ F_h\to
F_{gh}$, belonging to a torsor over $H^2(G,{\rm Aut}_{\otimes}({\rm Id}_\C))$. 

(ii) Actions of a group $G$ by braided autoequivalences 
of a braided fusion category $\B$, up to an isomorphism, 
are parametrized by pairs $(c,\eta)$, where $c: G\to \EqBr(\B)$ is a homomorphism, 
such that the corresponding first obstruction $O_3(c)\in
H^3(G,{\rm Aut}_{\otimes}({\rm Id}_\B))$ vanishes, and 
$\eta=(\eta_{g,h})$ is the equivalence class of the identification $F_g\circ F_h\to
F_{gh}$, belonging to a torsor over $H^2(G,{\rm Aut}_{\otimes}({\rm Id}_\B))$. 
\end{corollary} 

\subsection{Classification of quasi-trivial extensions}

Let $G$ be a finite group, and $\D$ a fusion category. 
We call a $G$-extension $\C$ of $\D$ {\it quasi-trivial} 
if every category $\C_g$ is a quasi-trivial bimodule category over 
$\D$. This condition is equivalent to the condition 
that $\C$ is strongly $G$-graded in the sense of \cite{Ga}, i.e.,
every category $\C_g$ contains an invertible object. 

The following proposition is a corollary of   
Theorem \ref{eqcl}. 

\begin{proposition}
Quasi-trivial $G$-extensions of $\D$, up to graded equivalence, 
are in natural bijection with homotopy classes of mappings 
$BG\to B\underline{\underline{\rm Out}}(\D)$.  
\end{proposition}

\begin{remark}
Note that a mapping 
$\tau: BG\to B\underline{\underline{\rm Out}}(\D)$
is representable as $\tau=B\xi\circ \zeta$, where $\xi: \underline{\rm Eq}(\D)\to \underline{\underline{\rm Out}}(\D)$ is defined in Subsection 
\ref{ximap} and $\zeta: BG\to \underline{\rm Eq}(\D)$ if and only if the corresponding quasi-trivial extension 
is trivial, i.e. $\C$ is the semidirect product category $\Vec_G\ltimes \D$ 
for the  $G$-action on $\D$ corresponding to $\zeta$.     
\end{remark} 

\subsection{Classification of faithfully graded braided $G$-crossed fusion categories}

Let $G$  be a finite group.
The notion of  a braided $G$-crossed fusion category is due to Turaev, see \cite{Tu1, Tu2}.
By definition, it is a $G$-graded category 
\begin{equation}
\label{Gcrossgrad}
\C =\bigoplus_{g\in G}\,\C_g,
\end{equation}
equipped with an action of $G$ such that $g(\C_h)  = \C_{ghg^{-1}}$ and a natural
family of isomorphisms
\begin{equation}
\label{G-braiding}
c_{X,Y} : X \ot Y \xrightarrow{\sim} g(Y) \ot X,\qquad g\in G,\, X\in \C_g,\, Y\in \C,
\end{equation}
called the {\em $G$-braiding}. The above action and $G$-braiding are required to satisfy certain 
natural compatibility conditions. In particular, the trivial component $\B:=\C_e$ is a braided fusion
category. We refer the reader to  \cite{Tu1, Tu2} for the precise definition
and to \cite[\S 4.4.3]{DGNO} for a detailed discussion of braided $G$-crossed categories.

Below we only consider braided $G$-crossed fusion categories with a {\em faithful} grading
\eqref{Gcrossgrad}.  The general case will be treated elsewhere.

\begin{theorem}
\label{ClassGcross}
Let $\B$ be a braided fusion category. Equivalence classes of braided $G$-crossed categories
$\C$ having a faithful $G$-grading with the trivial component  $\B$ 
are in bijection with morphisms of categorical
$2$-groups $G \to \underline{\underline{\rm Pic}}(\B)$, or,  equivalently, with homotopy classes of maps between
their classifying spaces $BG \to B\underline{\underline{\rm Pic}}(\B)$. 
\end{theorem}
\begin{proof}
Since $\underline{\underline{\rm Pic}}(\B) \subseteq \uuPic(\B)$, it follows from Theorem~\ref{eqcl} that a morphism
$G \to \underline{\underline{\rm Pic}}(\B)$ determines a $G$-extension $\C$ of $\B$. It remains to check that the additional
condition that each $\C_g,\, g\in G,$ is an invertible $\B$-module category  is equivalent
to the existence of an action of $G$ and a $G$-braiding on $\C$. 

Indeed, if the image of $G$ is inside $\underline{\underline{\rm Pic}}(\B)$ then for all
$g,h\in G$ the category $\Fun_\B(\C_g,\, \C_{gh})$ of $\B$-module functors from $\C_g$
to $\C_{gh}$ is identified, on one hand, with functors of right tensor multiplication by objects
of $\C_h$ and, on the other hand, with functors of left tensor multiplication by objects
of $\C_{ghg^{-1}}$.  So there is an equivalence $g: \C_h \to \C_{ghg^{-1}}$  
defined by the isomorphism of $\B$-module functors 
\begin{equation}
\label{defGact}
?\, \ot Y \cong g(Y) \ot\, ? : \C_g \to \C_{gh},\qquad Y\in \C_h.
\end{equation}
Extending it to $\C$ by linearity we obtain an action  of $G$ by tensor autoequivalences of $\C$.
Furthermore, evaluating \eqref{defGact} on $X\in \C_g$ we obtain a natural family
of isomorphisms
\[
X \ot Y \xrightarrow{\sim} g(Y) \ot X,\qquad g\in G,\, X\in \C_g,\, Y\in \C,
\]
which gives a $G$-braiding on $\C$.

To prove the converse, one can follow the proof of Theorem~\ref{component invertibility} 
to verify  that  components of a braided $G$-crossed
category are invertible module categories over its trivial component.
\end{proof}

\begin{remark} The somewhat similar problem of classifying
$G$-extensions of braided 2-groups is discussed in by E. Jenkins
in \cite{Je} (the notion of a $G$-extension of a braided 2-group is defined in
\cite{DGNO}, Appendix E, Definition E.8).
\end{remark}

\section{Classification of extensions (algebraic version)}
\label{classification proper}

Now we would like to retell the contents of the previous section in a purely algebraic language,
without using homotopy theory, and thus give an algebraic proof of Theorem \ref{extcl}.

\subsection{Decategorification} We start by recalling a well known decategorified version
of Theorem \ref{extcl}. 
Let $G$ be a group and let $R=\oplus_{g\in G}R_g$ be a $G-$graded
ring.
Recall that $R$ is called 
{\em strongly graded} if the multiplication map $R_g\ot_\BZ R_h\to R_{gh}$ is surjective;
in this situation we say that $R$ is strongly $G-$graded extension of $R_e$.
By \cite{Da} for a strongly $G-$graded ring $R$ the induced maps $R_g\ot_{R_e}R_h\to R_{gh}$ 
are isomorphisms;
in particular $R_g$ is an invertible $R_e-$bimodule for any $g\in G$. Thus any strongly $G-$graded 
ring $R$ defines a homomorphism $g\mapsto R_g$ of 2-groupoids $G\to \upic(R_e)$,  where 
$\upic(R_e)$ is the 2-groupoid of invertible $R_e-$bimodules. Conversely, it is clear that any
homomorphism of 2-groupoids $G\to \upic(S)$ determines a strongly $G-$graded extension
of $S$.

By definition, the group of 1-morphisms in $\upic(S)$ is the group of
isomorphism classes of invertible $S-$bimodules $\pic(S)$ and the group
of 2-endomorphisms of the unit 1-morphism (which is $S$ considered as an $S-$bimodule)
is the group $\Z(S)^\times$ of invertible elements of the center $\Z(S)$ of $S$. Thus the group
$\pic(S)$ acts on the abelian group $\Z(S)^\times$ and the equivalence class of 2-groupoid 
$\upic(S)$ is completely determined by a class $\omega \in H^3(\pic(S), \Z(S)^\times)$.
Thus we obtain the following result:

\begin{theorem} {\em (see \cite{CG})} There exists a class $\omega \in H^3(\pic(S), \Z(S)^\times)$
such that strongly $G-$graded extensions of a ring $S$ corresponding to a homomorphism
$\phi: G\to \pic(S)$ form an $H^2(G, \Z(S)^\times)-$torsor which is nonempty if and only if
$\phi^*(\omega)=0\in H^3(G, \Z(S)^\times)$. \qed
\end{theorem} 

\subsection{An action determined by a $G$-extension}
\label{action pi}

Let 
\[
\C =\bigoplus_{g \in G}\, \C_g
\]
 be a $G$-graded fusion category and let $\C_e:=\D$
(i.e., $\C$ is a $G$-extension of $\D$).
By Theorem~\ref{component invertibility}
for each pair $g,h \in G$ there is an equivalence of $\D$-bimodule categories
\begin{equation}
\label{Mgh}
M_{g,h} : \C_g \boxtimes_\D \C_h \cong \C_{gh}
\end{equation}
which comes from the restriction of the tensor product of $\otimes: \C\bt \C \to \C$ to 
\begin{equation}
\label{otgh}
\ot_{g,h} : \C_g \bt \C_h \to\C_{gh}.
\end{equation}

By Corollary~\ref{dual is center}
the group $Z_g:=\Aut_\D(\C_g)$ of $\D$-bimodule autoequivalences of $\C_g$ is abelian and 
is isomorphic to the group $Z:=Z_e$ of isomorphism classes of 
invertible objects in $\Z(\D)$.

Observe that for all $g, f \in G$ there are group isomorphisms
$i_{f,g} : Z_g \cong Z_{gf}$ and $j_{f,g} : Z_g \cong Z_{fg}$ defined by
\begin{eqnarray}
\label{i-iso}
i_{f,g}(b) &:=& M_{g,f} \circ (b \boxtimes_\D {\rm Id}_{\C_f}) \circ M_{g,f}^{-1} \\
\label{j-iso}
j_{f,g}(b) &:=& M_{f,g} \circ ( {\rm Id}_{\C_f} \boxtimes_\D b) \circ M_{f,g}^{-1}
\end{eqnarray}
for all $b\in Z_g$.
One can easily check that for all $f,h,g \in G$ there are equalities
\begin{equation}
\label{truth about actions}
i_{fh, g} =  i_{h, gf}i_{f,g}, \qquad \mbox{ and } \qquad
j_{fh, g} =  j_{f, hg}j_{h, g}.
\end{equation}
It follows from \eqref{Tfgh} that
the isomorphisms $i$ and $j$ commute
with each other, i.e.,  for all $f,g,h\in G$
there is an equality
\begin{equation}
\label{i,j commute}
i_{f, hg}j_{h,g} = j_{h, gf}i_{f,g}.
\end{equation}
Define an action of $G$ on $Z$ by $\rho : G \to Aut(Z) : g \mapsto \rho_g$ where
\begin{equation}
\label{pig}
\rho_g := i_{g^{-1}, g}\, j_{g, 1},\qquad g\in G.
\end{equation}

\begin{proposition}
\label{action independent}
The action $\rho : G \to Aut(Z)$ depends only on the homomorphism $c: G \to \Pic(\C_e)$
and does not depend on the choice of equivalences $M_{g,h} : \C_g \boxtimes_\D \C_h \cong \C_{gh}$.
\end{proposition}
\begin{proof}
It suffices to show that isomorphisms $i_{f,g},\, j_{f,g}$ 
defined by equations \eqref{i-iso} and \eqref{j-iso}
do not depend on the choice of equivalences $M_{g,h}$. 

Indeed, if $M'_{g,h}:
\C_g \boxtimes_\D \C_h \cong \C_{gh}$ is another $\D$-bimodule equivalence
then $M'_{g,h} = L_{g,h}\circ M_{g,h}$ where $L_{g,h} \in \Aut_\D(\C_{gh})$.
It follows that isomorphisms $\Aut_\D(\C_g) \cong   \Aut_\D(\C_{gh})$
determined by $M_{g,h}$ and $M'_{g,h}$ differ by a conjugation by $L_{g,h}$.
But since $\Aut_\D(\C_{gh})$ is abelian, this conjugation is trivial.
\end{proof}

\subsection{Cohomological data determined by a $G$-extension}
\label{algebraic data}

Let $\C =\oplus_{g \in G}\, \C_g$ be a $G$-extension of $\C_e=:\D$.
We continue to use the notation introduced in Section~\ref{action pi}, see
\eqref{Mgh} and \eqref{otgh}.

For all $g,h\in G$ let 
\begin{equation}
\label{Bgh}
B_{g,h} := B_{\C_g,\C_h}: \C_g \boxtimes \C_h \to \C_g \boxtimes_\D \C_h
\end{equation} 
be the canonical functor coming from Definition~\ref{tensor product in BC}. 
For all $f,g,h\in G$ consider the following diagram of $\D$-bimodule categories
and functors:
\begin{equation}
\label{two associativities}
\xymatrix{
\C_f \boxtimes_\D \C_g \boxtimes \C_h \ar[d]_{M_{f,g}} \ar@{-->}[rd]^>>>>>{B_{g,h}}  &
\C_f \boxtimes \C_g \boxtimes \C_h \ar[l]_{B_{f,g}} \ar[r]^{B_{g,h}}
\ar[dl]^>>>>>>>{\otimes_{f,g}}  \ar[dr]^>>>>>{\otimes_{g,h}} &
\C_f \boxtimes \C_g \boxtimes_\D \C_h  \ar[d]^{M_{g,h}} \ar@{-->}[dl]^>>>>>>>{B_{f,g}} \\
\C_{fg} \boxtimes \C_h \ar[d]_{B_{fg,h}}  \ar[rd]^>>>>>>>{\otimes_{fg,h}} &
\C_f \boxtimes_\D \C_g \boxtimes_\D \C_h  \ar@{-->}[dl]_>>>>>>{M_{f,g}} 
\ar@{-->}[dr]^>>>>>{M_{g,h}}&
\C_f \boxtimes \C_{gh} \ar[d]^>>>>>>>>{B_{f,gh}} \ar[dl]_>>>>>>>{\otimes_{f,gh}} \\
\C_{fg} \boxtimes_\D \C_h \ar[r]_{M_{fg,h}} & 
\C_{fgh} &
\C_f \boxtimes_\D \C_{gh} \ar[l]^{M_{f,gh}}.
}
\end{equation}
In this diagram we
will refer to polygons formed by solid lines as those in the ``front''
and to polygons formed by dotted lines as  the ``rear''.
The four triangles in the front commute 
by the universal property of tensor product of module categories. The square in the front commutes
up to an associativity constraint,
which is an isomorphism of $\D$-bimodule functors.
Hence, the perimeter of the diagram commutes 
up to a natural isomorphism of $\D$-bimodule functors. The upper rear quadrangle 
commutes by Remark~\ref{properties of btD}(ii)
and the left and right rear quadrangles commute since functors $B_{g,h}, g,h\in G,$ 
come from  $\D$-balanced functors.
Therefore, the lower quadrangle in the rear commutes up to 
a natural isomorphism  of $\D$-bimodule functors :
\begin{equation}
\label{etafgh}
\alpha_{f,g,h}: 
M_{f, gh} ( {\rm Id}_{\C_f} \boxtimes_\D M_{g,h}) \cong M_{fg,h} (M_{f,g} \boxtimes_\D {\rm Id}_{\C_h}).
\end{equation}
Equivalently, the following $\D$-bimodule functor 
\begin{equation}
\label{Tfgh}
T_{f,g,h}: = M_{fg,h} (M_{f,g} \boxtimes_\D {\rm Id}_{\C_h}) 
( {\rm Id}_{\C_f} \boxtimes_\D M_{g,h}^{-1}) M_{f, gh}^{-1} :\C_{fgh} \to \C_{fgh}
\end{equation}
is isomorphic (as a $\D$-bimodule functor) to the identity.

The pentagon axiom for the tensor product in $\C$ implies the following equality
of natural transformations
\begin{multline}
\label{pentagon for eta}
M_{f,ghk} (\id_f \boxtimes_\D
\alpha_{g,h,k})\circ \alpha_{f,gh,k} ({\rm Id}_{\C_f} \boxtimes_\D M_{g,h} 
\boxtimes_\D {\rm Id}_{\C_k}) \circ M_{fgh,k}(\alpha_{f,g,h} \boxtimes_\D \id_k)  \\
= \alpha_{f,g,hk} ({\rm Id}_{\C_f} \boxtimes_\D {\rm Id}_{\C_g}
\boxtimes_\D M_{h,k})
\circ \alpha_{fg,h,k} (M_{f,g} \boxtimes_\D {\rm Id}_{\C_h} \boxtimes_\D {\rm Id}_{\C_k}),
\end{multline}
for all $f,g,h,k\in G$ (note that we use the notation ${\rm Id}$ for the identity functor, and $\id$ for
the identity morphism). 

To summarize, a $G$-extension $\C$ determines the following data:
\begin{enumerate}
\item[(1)] a fusion category $\D$, a collection of invertible $\D$-bimodule categories
$\C_g,\, g\in G$ such that $\C_e \cong \D$, and an action $\rho$ of $G$ 
by automorphisms of the group $Z$ of invertible objects of $\Z(\D)$,
\item[(2)] a collection of $\D$-bimodule isomorphisms 
$M_{gh}: \C_g \boxtimes_\D \C_h \cong \C_{gh}$ 
such that each $T_{f,g,h}$ defined by \eqref{Tfgh} is isomorphic to the identity
as a $\D$-bimodule functor,
\item[(3)] natural isomorphisms $\alpha_{f,g,h}$ \eqref{etafgh} satisfying identity
\eqref{pentagon for eta}.
\end{enumerate}

In the next few subsections we will show that, conversely,
a set of data with the above properties gives rise to a $G$-extension.



\subsection{Obstruction to the existence of tensor product}\label{obsten}

Let us consider a situation opposite to the one studied in Section~\ref{algebraic data}. 
Let $G$ be a finite group. 
Suppose that we are given a fusion category $\D$,
a group homomorphism  $c: G \to \Pic(\D),\ g \mapsto \C_g$, 
and  there are $\D$-bimodule equivalences 
\begin{equation}
M_{g,h}: \C_g \boxtimes_\D \C_h \cong \C_{gh}, 
\end{equation}
for all $g,h\in G$.  Let $\rho : G \to {\rm Aut}(Z), Z:=\Aut_\D(\D)$  be the action of $G$
defined in Section~\ref{action pi}. By Proposition~\ref{action independent},
$\rho$ depends only on $c$ and not on the choice $M_{g,h},\, g,h\in G$.
We would like to parameterize fusion category structures on
$\C =\oplus_{g\in G}\,\C_g$ which give rise to this data.

%
%


First, let us investigate the existence of a 
$G$-graded quasi-tensor category structure
on $\C$. By a {\em quasi-tensor category}  we mean
a category $\C$ with  a bifunctor $\otimes :\C \times \C \to \C$
such that $\otimes\circ (\otimes \times {\rm Id}_\C) \cong \otimes \circ ({\rm Id}_\C \times \otimes) $
(so we do not yet require existence of an associativity constraint for $\otimes$).

As before, let $Z_g = \Aut_\D(\C_{g})$.  Then $Z := Z_e$ is
the group of invertible objects in $\Z(\D)$.
We have isomorphisms $i_{g,1} : Z \cong Z_g$ for all $g\in G$. 

For all $f,g,h \in G$ let
\begin{equation}
T_{f,g,h} =  M_{fg,h} (M_{f,g} \boxtimes_\D {\rm Id}_{\C_h}) 
            ( {\rm Id}_{\C_f} \boxtimes_\D M_{g,h}^{-1}) M_{f, gh}^{-1}
\end{equation}
Then $\bar{T}_{f,g,h} = i_{fgh, 1}^{-1}(T_{f,g,h})$ 
defines a function on $G^3$ with values in the abelian group $Z$.
One can directly check that $\bar{T}_{f,g,h}$ is an element of
$Z^3(G, Z)$, i.e., a $3$-cocycle on $G$ with values in $Z$ (the latter
is a $G$-module via $\rho$).
Let us find how this function depends on the choice of equivalences $M_{g,h}$. 

Suppose each $M_{g,h}$ is replaced by $M'_{g,h} = L_{g,h}\circ M_{g,h}$ where 
$L_{g,h} \in Z_{gh}$ as above. Then the corresponding function on $G^3$ with values in
$Z_{fgh}$ is 
\begin{eqnarray*}
T'_{f,g,h} 
&=&  M'_{fg,h} (M'_{f,g} \boxtimes_\D {\rm Id}_{\C_h})
( {\rm Id}_{\C_f} \boxtimes_\D M'^{-1}_{g,h})  M'^{-1}_{f, gh} \\
&=& L_{fg,h} M_{fg,h} (L_{f,g} M_{f,g} \boxtimes_\D {\rm Id}_{\C_h})
 ( {\rm Id}_{\C_f} \boxtimes_\D M_{g,h}^{-1} L_{g,h}^{-1}) M_{f, gh}^{-1} L_{f, gh}^{-1}\\
&=& L_{fg,h} \,i_{h, fg}(L_{f,g})\, j_{f, gh}(L_{g,h}^{-1})\, L_{f, gh}^{-1}\, T_{f,g,h}.
\end{eqnarray*}
Let $\bar{L}_{g,h} := i_{gh,1}^{-1}(L_{g,h}) \in Z$. We compute,
using Equations~\eqref{truth about actions}, \eqref{i,j commute} 
and definition \eqref{pig} of the action $\rho$:
\begin{eqnarray*}
\bar{T}'_{f,g,h}
&=&  i_{fgh, 1}^{-1}(T'_{f,g,h})  \\
&=& \bar{L}_{fg,h} \,i_{fgh,1}^{-1} i_{h, fg} i_{fg, 1}(\bar{L}_{f,g})\, 
i_{fgh,1}^{-1} j_{f, gh} i_{gh,1}(\bar{L}_{g,h}^{-1})\, \bar{L}_{f, gh}^{-1}\, 
\bar{T}_{f,g,h}\\
&=& \bar{L}_{fg,h} \,\bar{L}_{f,g}\, 
\rho_f(\bar{L}_{g,h}^{-1})\, \bar{L}_{f, gh}^{-1}\, 
\bar{T}_{f,g,h}.
\end{eqnarray*}
Thus, the function $\bar{T}'$ differs from $\bar{T}$ by a coboundary. 
This yields a cohomology class in $H^3(G, Z)$ independent on the choice
of the equivalences $M_{g,h}$.

\begin{definition}
We will call the cohomology class of $\bar{T}$
in $H^3(G, Z)$ the {\em tensor product obstruction class} and denote it $O_3(c)$.
\end{definition}

\begin{theorem}
\label{tensor product obstruction}
Let $G$ be a finite group and let $\D$ be a fusion category. Let  
\[
c: G \to \Pic(\D) : g \mapsto \C_g
\] 
be a group homomorphism.
Then there exist $\D$-bimodule category equivalences $\C_e \cong \D$ and 
$\C_g \boxtimes_\D \C_h \cong \C_{gh},\, g,h\in G$ 
defining a $\D$-bimodule tensor product $\otimes$
on $\C =\oplus_{g\in G}\,\C_g$ such that $\otimes\circ (\otimes \boxtimes {\rm Id}_\C) \cong
\otimes\circ ({\rm Id}_\C \boxtimes\otimes)$
if and only if the obstruction class $O_3(c)$ is the trivial element of $H^3(G, Z)$.
\end{theorem}
\begin{proof}
Consider the diagram of functors~\eqref{two associativities}. 
It was explained above that its natural $\D$-bimodule commutativity is equivalent 
to the natural $\D$-bimodule isomorphism $T_{f,g,h}\cong {\rm Id}_{\C_{fgh}},\ f,g,h\in G$.
The latter is equivalent to $\bar{T}$ being cohomologous to $1$ in $Z^3(G, Z)$. 
\end{proof}

\subsection{Construction of a quasi-tensor product}

\begin{theorem}
\label{products correspond to H2}
Suppose that the obstruction class $O_3(c)$ vanishes. Then isomorphism classes
of $\D$-bimodule tensor products on $\C =\oplus_{g\in G}\,\C_g$ form a torsor over  
the second cohomology group $H^2(G, Z)$.
\end{theorem}
\begin{proof}
Choose $\D$-bimodule equivalences 
$M_{g,h}: \C_g \boxtimes_\D \C_h \cong \C_{gh}$ for all $g,h \in G$ 
and natural isomorphisms of $\D$-bimodule functors 
\begin{equation}
\label{eta}
\alpha_{f,g,h}: 
M_{fg,h} (M_{f,g} \boxtimes_\D {\rm Id}_{\C_h}) \xrightarrow{\sim}
M_{f, gh} ( {\rm Id}_{\C_f} \boxtimes_\D M_{g,h}) 
\end{equation}
which give rise to a natural isomorphism   
\begin{equation}
\label{big eta}
\alpha: \otimes\circ (\otimes \times {\rm Id}_{\C}) \xrightarrow{\sim}  \otimes\circ ({\rm Id}_{\C} \times \otimes).
\end{equation}
The computations done in the previous  subsection show that 
replacing each $M_{g,h}$ by $L_{g,h}\circ M_{g,h}$ where 
$L_{g,h} \in Z_{gh}$ makes the two sides of \eqref{eta}  differ by 
\begin{equation}
\bar{L}_{fg,h} \,\bar{L}_{f,g}\, 
\rho_f(\bar{L}_{g,h}^{-1})\, \bar{L}_{f, gh}^{-1} \in Z.
\end{equation} 

Thus, substituting  $L_{g,h}\circ M_{g,h}$ for $M_{g,h}$
does not affect the existence
of an isomorphism $\alpha$ as in \eqref{eta} if and only if 
$\bar{L}\in Z^2(G, Z)$ is a $2$-cocycle.
Clearly, two $2$-cocycles define isomorphic tensor products
if and only if they are cohomologous.
\end{proof}

Thus, in the case when $O_3(c)$ is cohomologically trivial, one defines a tensor product
on $\C =\oplus_{g\in G}\, \C_g$ as follows. Choose $\D$-bimodule equivalences
$$
M_{g,h}: \C_g \boxtimes_\D \C_h \cong \C_{gh}
$$ 
and natural isomorphisms $\alpha_{f,g,h}$ as in \eqref{eta}. Then each $M_{g,h}$ gives rise
to a product $\otimes_{g,h} : \C_g \boxtimes \C_h \to \C_{gh}$ and $\alpha_{f,g,h}$
gives rise to an isomorphism 
$$
\otimes_{f,gh} \circ ({\rm Id}_{\C_f} \boxtimes \otimes_{g,h})
\cong \otimes_{fg,h}\circ (\otimes_{f,g} \boxtimes {\rm Id}_{\C_h}).
$$ 

\subsection{Obstruction to the existence of an associativity constraint}

We continue to assume that the tensor product obstruction $O_3(c)$ vanishes,
i.e., that $c$ gives rise to a $\D$-bimodule quasi-tensor product
on $\C =\oplus_{g\in G}\, \C_g$. 

Let us determine when this quasi-tensor product is in fact a tensor product, 
i.e., when it admits an  associativity constraint satisfying the pentagon equation. 
Choose a collection of $\D$-bimodule category equivalences
\[
M =\{ M_{g,h}: \C_g \boxtimes_\D \C_h \cong \C_{gh} \}
\]
and  natural isomorphisms  $\alpha_{f,g,h}$ as in  \eqref{eta}.

\begin{definition}
\label{system of products M}
We will call $M =\{ M_{g,h}: \C_g \boxtimes_\D \C_h \cong \C_{gh} \}_{g,h\in G}$ 
a {\em system of products}.
\end{definition}

By Theorem~\ref{products correspond to H2}, $M$ is an element of an $H^2(G, Z)$-torsor.

Note that each $\alpha_{f,g,h}$ is determined up to
an automorphism of a simple object in $\Z(\C_e)$, i.e. up to a nonzero scalar.

For all $f,g,h,k\in G$ let us consider the following cube
whose vertices are $\D$-bimodule categories, 
edges are $\D$-bimodule equivalences $M_{a,b}$, and faces are natural 
isomorphisms $\alpha_{a,b,c},\, a,b,c\in G$, see \eqref{eta} 
(to keep the diagram readable, only the faces are labeled):
\begin{equation}
\label{polytope}
\xymatrix{
\C_f \bt_\D \C_g \bt_\D \C_h \bt_\D \C_k \ar[r] \ar[dd]^{}="a" \ar[drr]^{}="b"
& \C_{fg} \bt_\D \C_{h} \bt_\D \C_k \ar[drr]  \ar@{-->}[dd]_{}="c"  & & \\
& & \C_f \bt_\D \C_g \bt_\D \C_{hk} \ar[dd]^{}="h" \ar[r]^{}="g" & 
\C_{fg} \bt_\D \C_{hk} \ar[dd]^{}="f" \\
\C_f \bt_\D \C_{gh} \bt_\D \C_k \ar@{-->}[r]^{}="d"  \ar[drr]_{}="k" & 
\C_{fgh} \bt_\D \C_k \ar@{-->}[drr]_{}="l"^{}="e" & & \\
& & \C_f \bt_\D \C_{ghk} \ar[r] &  \C_{fghk}.
\ar@2{->}^{\alpha_{g,h,k}}"a";"b"
\ar@2{-->}_{\alpha_{f,g,h}}"c";"d"
\ar@2{-->}_{\alpha_{fg,h,k}}"e";"f"
\ar@2{-->}^{\alpha_{f,gh, k}}"k";"l"
\ar@2{->}^{\alpha_{f,g,hk}}"g";"h"
}
\end{equation}
%
The composition of the natural transformations corresponding to faces
of this cube is a $\D$-bimodule automorphism of the functor 
$$
M_{fgh,k}\circ (M_{fg,h} \boxtimes_\D {\rm Id}_{\C_k})\circ 
(M_{f,g} \boxtimes_\D {\rm Id}_{\C_h} \boxtimes_\D {\rm Id}_{\C_k}) 
$$
i.e., the scalar
\begin{multline}
\label{nu=1}
\nu_{f,g,h,k}:=
(M_{f,g} \boxtimes_\D {\rm Id}_{\C_h} \boxtimes_\D {\rm Id}_{\C_k}) \alpha_{fg,h,k}^{-1}
\circ ({\rm Id}_{\C_f} \boxtimes_\D {\rm Id}_{\C_g}  \boxtimes_\D M_{h,k}) \alpha_{f,g,hk}^{-1} \\
\circ M_{f,ghk} (\id_f \boxtimes_\D \alpha_{g,h,k}) 
\circ \alpha_{f,gh,k} ({\rm Id}_{\C_f} \boxtimes_\D M_{g,h} 
\boxtimes_\D {\rm Id}_{\C_k}) \circ M_{fgh,k}(\alpha_{f,g,h} \boxtimes_\D \id_k)
\end{multline}
The commutativity of the cube \eqref{polytope} 
is equivalent to the existence of an associativity constraint for $\C$
satisfying the pentagon axiom.  The cube
commutes if and only if $\{\alpha_{f,g,h}\}_{ f,g,h\in G}$ can be chosen 
in such a way that $\nu_{f,g,h,k}= 1$.
See \cite{KV} regarding the notions of composition of faces and commutativity of a polytope.

It is easy to check that the above $\nu$ is a $4$-cocycle on $G$ with values in ${\bold k}^\times$.
Let us determine how $\nu$ changes when we change the choice of $\alpha$. Let 
\begin{equation}
\label{eta prime}
\alpha'_{f,g,h} =  \alpha_{f,g,h} \lambda_{f,g,h}, \qquad \mbox{where}\qquad \lambda_{f,g,h} \in {\bold k}^\times.
\end{equation}
Then the corresponding scalar is
\begin{equation}
\label{nu prime}
\nu'_{f,g,h,k}= \nu_{f,g,h,k} \lambda_{f,g,hk}^{-1} \lambda_{fg,h,k}^{-1} \lambda_{g,h,k}
\lambda_{f,gh,k} \lambda_{f,g,h}
\end{equation}
Therefore, there is a canonical element in $H^4(G, {\bold k}^\times)$ (the class of $\nu$) which depends only on $c$
and the choice of $M$. 

\begin{definition}
We will call this element the {\em associativity constraint obstruction class} and denote
it $O_4(c, M)$.
\end{definition}

\begin{theorem}
\label{associativity obstruction}
Suppose that a homomorphism $c: G \to \Pic(\D)$ is such that $\C =\oplus_{g\in G}\, \C_g$ 
(with $\C_e = \D$) admits a $\D$-bimodule quasi-tensor product via a choice of 
a system of products $M$. Then this product
admits an associativity constraint satisfying the pentagon equation if and only if
$O_4(c, M)$ is the trivial element of $H^4(G, {\bold k}^\times)$.
\end{theorem}
\begin{proof}
Clear from the discussion above.
\end{proof}

Let $\C =\oplus_{g\in G}\, \C_g$ be a $G$-extension and let
$\alpha_{X,Y,Z}$ be the associativity constraint for the tensor product of $\C$, 
where $X,Y,Z$ are objects in $\C$.
Given a $3$-cocycle $\omega\in Z^3(G,\, {\bold k}^\times)$ one can define
a new associator 
\[
\alpha^\omega_{X ,Y,Z} := \omega(f,g,h) \alpha_{X,Y,Z}
\]
for all $X\in \C_f,\, Y\in\C_g,\, Z\in \C_h$. 

Let $\alpha'_{X,Y,Z}$ be another associativity constraint for the tensor product of $\C$.
We will say that $\alpha'$ is {\em equivalent} to $\alpha$ if $\alpha' =\alpha^\omega$
for some coboundary $\omega$. Clearly, equivalent associators determine
equivalent tensor categories.

\begin{theorem}
\label{associativities correspond to H3}
Suppose that the obstruction classes $O_3(c)$ and  $O_4(c, M)$ vanish. 
Then the equivalence classes of associativity constraints
for the tensor product of $\C =\oplus_{g\in G}\, \C_g$ 
coming from the system of products $M$ form  
a torsor $T_c^2$ over $H^3(G,{\bold k}^\times)$.
\end{theorem}
\begin{proof}
The proof is similar to the proof of Theorem~\ref{products correspond to H2}.
We need to establish a bijection between the choices of 
$\alpha =\{ \alpha_{f,g,h} \}_{f,g,h\in G}$ leading to associativity constraints on $\C$
and elements of $H^3(G,{\bold k}^\times)$. If one such $\alpha$ is chosen, any other choice
has a form \eqref{eta prime}. From equation \eqref{nu prime} we see that the corresponding 
coboundary $\nu'_{f,g,h,k}$ is equal to 1 precisely when
\begin{equation}
\label{n is a 3 cocycle}
\lambda_{f,g,h} \lambda_{f,gh,k} \lambda_{g,h,k} = \lambda_{fg,h,k} \lambda_{f,g,hk},\qquad
f,g,h,k\in G,
\end{equation}
i.e., when $\lambda$ is a $3$-cocycle. Moreover, two $3$-cocycles are cohomologous if and 
only if the corresponding associators on $\C$ are  equivalent.
\end{proof}

\subsection{The associativity constraint and the Pontryagin-White\-head quadratic function}

Let $\E$ be a pointed braided fusion category, and let 
$A$ be the group of isomorphism classes of invertible objects 
of $\E$. Let $G$ be a finite group acting on $\E$ by braided
autoequivalences. In this situation we can define a quadratic
map $PW: H^2(G,A)\to H^4(G,{\bold k}^\times)$, which we call
{\it the Pontryagin-Whitehead quadratic function}, as follows. 

Let 
\[
L : G \times G \to A : (f,g)\mapsto L_{f,g}
\]
be a 2-cocycle of $G$ with coefficients in $A$, i.e. 
a collection of simple objects of $\E$, such that there exist isomorphisms 
\begin{equation}
\label{Pasha zeta}
\zeta_{f,g,h}: L_{fg,h}\otimes L_{f,g}\cong L_{f,gh}\otimes f(L_{g,h}), \qquad f,g,h\in G.
\end{equation}
Then we can consider the automorphism of $L_{fgh,k}\otimes L_{fg,h}\otimes L_{f,g}$
(identified with a  scalar) $\nu_{f,g,h,k}$, given by the composition
$$
L_{fgh,k}\otimes L_{fg,h}\otimes L_{f,g}\to
L_{fgh,k}\otimes L_{f,gh}\otimes f(L_{g,h})\to
L_{f,ghk}\otimes f(L_{gh,k})\otimes f(L_{g,h})
$$
$$
\to L_{f,ghk}\otimes f(L_{g,hk})\otimes fg(L_{h,k})\to
L_{fg,hk}\otimes L_{f,g}\otimes fg(L_{h,k})
$$
$$
\to L_{fg,hk}\otimes fg(L_{h,k})\otimes L_{f,g}\to
L_{fgh,k}\otimes L_{fg,h}\otimes L_{f,g},
$$
where we suppress the associativity isomorphisms, 
and all the maps except the fifth map are given by the isomorphisms $\zeta_{x,y,z}$ 
from \eqref{Pasha zeta}
for appropriate $x,y,z$, while the fifth map is given by the braiding 
acting on $L_{f,g}\otimes fg(L_{h,k})$. 

\begin{proposition}\label{threeprope}
\begin{enumerate}
\item[(i)] $\nu$ is a 4-cocycle of $G$ with coefficients in ${\bold k^\times}$,
\item[(ii)] if $\zeta$ is changed by a cochain $\xi_{x,y,z}$, then $\nu$ is multiplied by $d\xi$
(so the cohomology class of $\nu$ does not change),
\item[(iii)] if $L_{f,g}$ is changed by a coboundary, i.e. replaced by 
\[
L_{f,g}'=X_{fg}\otimes L_{f,g}\otimes f(X_g^{-1})\otimes X_f^{-1},
\] 
where $(X_f)$ is a collection of simple objects, then $\nu$ is changed by a coboundary. 
\end{enumerate}
\end{proposition}

\begin{proof}
Straightforward verification. 
\end{proof} 

\begin{definition} The map $PW: H^2(G,A)\to H^4(G,{\bold k}^\times)$ is defined by 
\begin{equation}
\label{PW}
PW(L)=\nu.
\end{equation}
\end{definition} 

Note that if the action of $G$ on $\E$ is altered by an element 
$\theta\in H^2(G,A^*)$ (by changing the isomorphisms $\eta_{g,h}: F_g \circ F_h \xrightarrow{\sim} F_{gh}$), 
then the map $PW$ is modified  according to the rule
$$
PW'(L)=PW(L)(L,\theta),
$$
where $(\, , \,):  H^2(G,\, A)\times H^2(G,A^*) \to H^4(G,\, {\bold k}^\times)$ is the evaluation map
combined with the cup product in the cohomology of $G$.

Let $q$ be the quadratic form on $A$ defined by the braiding on $\E$
($q(Z)=c_{ZZ}$), and $b_q$ be the corresponding symmetric bilinear form
($b_q(Y,Z)=c_{YZ}c_{ZY}$). The following proposition shows that $PW$ is indeed a quadratic function. 

\begin{proposition}\label{quadra}
$$
PW(L_1L_2)=PW(L_1)PW(L_2)b_q(L_1,L_2).
$$
\end{proposition} 

\begin{proof} 
This is verified by a direct computation from the definition, by using the hexagon relations for the braiding. 
\end{proof} 

Now let us assume that $|A|$ is odd. Then $\E=\Vec_A$, and the braiding on $\E$ is canonically defined by 
the quadratic form $q$ on $A$, which is, in turn, determined by the corresponding symmetric bilinear form $b_q$.
Thus, every homomorphism $\phi: G\to O(A,q)$ canonically defines an action of $G$ on $\E$. 
In this case, we can pick the associativity morphisms and the maps 
$\zeta_{x,y,z}$ to be the identities, and one gets 
\begin{equation}\label{nuform}
\nu_{f,g,h,k}=c_{L_{f,g},fg(L_{h,k})}.
\end{equation}

\begin{proposition} For the canonical action of $G$ on $\E$, one has 
$$
PW(L)=b_q(L^{1/2},L)
$$
(i.e., $PW(L)$ is $b_q$ applied to the cup product of $L^{1/2}$ 
with $L$). Thus, for the canonical action shifted by $\theta\in H^2(G,A^*)$, one has 
$$
PW(L)=b_q(L^{1/2},L)(L,\theta).
$$
\end{proposition}

\begin{proof} This follows from 
formula (\ref{nuform}).
\end{proof}

\begin{remark} The map $PW$ can be alternatively characterized as follows. 
Since $G$ acts on $\E$ by braided autoequivalences, it acts canonically 
on the Drinfeld center $\Z(\E)$. Note that $\Z(\E)$ is a pointed category, 
and its group of simple objects is $A\oplus A^*$. Thus, an element $L\in H^2(G,A)$ is 
nothing but a way to alter the canonical action of $G$ on $\Z(\E)$ (keeping its action on isomorphism 
classes of objects fixed), so that its action on $\E\subset \Z(\E)$ remains the same. 
By Theorem \ref{breq} and Theorem \ref{extcl}, having fixed $L$, we fix a collection 
of $\E$-bimodule categories $\E_g$, $g\in G$, with a tensor product functor on them. 
Then $PW(L)$ is nothing but the obstruction $O_4(c,L)$ to the existence of the associativity constraint for 
this tensor product functor.
\end{remark}

Now let $\D$ be any fusion category.
Let $c: G\to \Pic(\D)$ be a group homomorphism, 
$c(g)=\C_g$, and let $M=(M_{g,h})$ be a choice of
isomorphisms $\C_g\boxtimes_\D \C_h\to \C_{gh}$ defining 
a tensor product functor on $\C=\oplus_g \C_g$. 
Then $G$ acts on the braided category $\Z(\D)$, in particular, on the subcategory 
of its invertible objects. Thus, we have the Pontryagin-Whitehead quadratic function 
$$
PW_M: H^2(G,\pi_2)\to H^4(G,\pi_3),
$$  
where $\pi_2={\rm Inv}(\Z(\D))$, $\pi_3={\bold k}^\times$. 

\begin{proposition}\label{whi}
For any $L\in H^2(G,\pi_2)$, 
one has 
$$
O_4(c,LM)/O_4(c,M)=PW_M(L).
$$
Thus, for 
$L_1,L_2\in H^2(G,\pi_2)$, 
one has 
$$
\frac{O_4(c,L_1L_2M)O_4(c,M)}
{O_4(c,L_1M)O_4(c,L_2M)}=
[L_1,L_2],
$$
where $[,]$ is the Whitehead bracket combined with the cup product in the cohomology of $G$. 
\end{proposition}

\begin{proof}
The first statement follows by replacing $M$ by $LM$ in diagram (\ref{polytope})
and applying the definition of $PW_M$. The second statement follows from the first one and Proposition \ref{quadra}.  
\end{proof}

{\bf Remark.} 
A version of the map $PW$ 
is discussed in \cite{B}, Chapter V, 
and under additional assumptions, 
the above results can be derived from 
the statements in \cite{B}. 

\subsection{A divisibility theorem}

The following theorem is somewhat analogous 
to the Anderson-Moore-Vafa theorem for tensor categories (see \cite{Et}).

\begin{theorem}\label{divi}
Let $D$ be the Frobenius-Perron dimension of $\D$. 
Then the order of $O_4(c,M)$ in $H^4(G,{\bold k}^\times)$ divides $D^4$. 
\end{theorem}

\begin{proof} 
For $a\in G$, let $R_a=\oplus_{X\in {\rm Irr}\C_a}{\rm FPdim}(X)X$ be the regular (virtual) object of $\C_a$ 
(where the Frobenius-Perron dimensions in $\C_a$ are normalized in such a way that ${\rm FPdim}(R_a)=D$).  
Let us apply equation (\ref{nu=1}) to the product $R_f\otimes R_g\otimes R_h\otimes R_k$,
and compute the determinants of both sides (where the determinant is understood 
in the sense of \cite{Et}, Section 2). Since $\nu$ is a scalar, on the left hand side we get 
$\nu_{f,g,h,k}^{D^4}$. To compute the right hand side, we use that $R_a\otimes R_b=DR_{ab}$.
Then the right hand side takes the form 
$$
\det(\alpha_{fg,h,k})^{-D}\det(\alpha_{f,g,hk})^{-D}\det(\alpha_{g,h,k})^D\det(\alpha_{f,gh,k})^D\det(\alpha_{f,g,h})^D. 
$$
Thus we see that 
$$
\nu^{D^4}=d(\det(\alpha)^D),
$$
where $d$ is the differential in the standard complex of $G$ with coefficients in ${\bold k}^\times$. 
This implies the statement. 
\end{proof}

\section{Examples of extensions}
\label{ex section}

Throughout this section we freely use the notation and terminology from the previous sections.

\subsection{Extensions of finite groups}
Let $G$ be a finite group.
The problem of finding all $G$-extensions of a fusion category $\D$ 
(i.e., $G$-graded fusion categories $\C =\oplus_{g\in G}\, \C_g$
with a prescribed identity component $\C_e=\D$) includes, as a special case,
the classical theory of group extensions \cite{EM}.

Indeed, let $H$ be a finite group and let $\D = \Vec_H$ be the fusion category of
$H$-graded vector spaces with the trivial associator. For any automorphism $\alpha \in \Aut(H)$
let $\M_\alpha$ be the $\D$-bimodule category, which is $\D$ as an abelian category,
with the actions given by
\begin{equation*}
{\bold k}_{h} \otimes {\bold k}_x = {\bold k}_{\alpha(h)x},\qquad \mbox{and} \qquad {\bold k}_{x} \otimes {\bold k}_h = {\bold k}_{xh},
\qquad h,x \in H,
\end{equation*}
where ${\bold k}_h, h\in H$ are the simple objects of $\Vec_H$, and with 
the usual vector space associator.
Note that $\M_\alpha$ is a typical example of an indecomposable  
$\D$-bimodule category which is equivalent
to $\D$ as a right $\D$-module category.

It is easy to check that  $\M_\alpha$ is isomorphic to the regular $\D$-bimodule
category if and only if $\alpha$ is an inner automorphism and that
$\M_\alpha \boxtimes_\D \M_\beta \cong \M_{\alpha\beta}$ for all $\alpha,\beta\in \Aut(H)$.
In particular, each $\M_\alpha$ is an invertible $\D$-bimodule category.

Thus, in this case a homomorphism $c: G \to \Pic(\D)$ with the property that each $\C_g$ 
is equivalent to $\D$ as a right $\D$-module category is the same thing as
a homomorphism $c: G \to \Out(H)$ to the quotient of $\Aut(H)$ by the subgroup
of inner automorphisms. For such a homomorphism choose a representative
$\gamma_g \in \Aut(H)$ from each coset $c(g)$ and let $\C_g= \M_{\gamma_g}$.

If there is a fusion category structure on $\C =\oplus_{g\in G}\, \C_g$ then
this category is pointed and hence is equivalent to a category of $K$-graded
vector spaces for some group $K$, possibly with a 3-cocycle $\omega$. 
It is clear that this $K$ is an extension of
$G$ by $H$, i.e., there is a short exact sequence of finite groups:
\begin{equation}
\label{exact sequence of groups}
1 \longrightarrow H \longrightarrow K \longrightarrow G \longrightarrow 1.
\end{equation}
In this case the group $Z$ is isomorphic to $Z(H)\oplus \Hom(H,\,{\bold k}^\times)$, where
$Z(H)$ is the center of $H$ and $\Hom(H,\,{\bold k}^\times)$ is the group of homomorphisms
from $H$ to ${\bold k}^\times$. Indeed, as we observed earlier, $Z$ is isomorphic to the group 
of invertible objects of $\Z(\D)$ ($=\Z(\Vec_H)$).

One can easily check that the obstruction class $O_3(c)$ belongs to 
$H^3(G, Z(H))\subset H^3(G,Z(H))\oplus \Hom(H,{\bold k}^\times)$,
and coincides with the Eilenberg-MacLane obstruction to the existence of extension
\eqref{exact sequence of groups}, with a given action $G\to \Out(H)$
see \cite{EM}. When this obstruction vanishes, we have a choice of $M=(M_1,M_2)$ 
where $M_1$ belongs to a torsor $T_1$ over $H^2(G,Z(H))$, and $M_2$ belongs 
to a torsor $T_2$ over $H^2(G, \Hom(H,{\bold k}^\times))$. One can check that the torsor $T_1$ is 
exactly the one classifying group extensions, see \cite{EM}. Furthermore, the torsor 
$T_2$ is canonically trivial, since every group extension canonically determines
a categorical extension. 
Finally, it is easy to check that the obstruction $O_4(c,M_1,M_2)$ is linear in $M_2$, and 
it follows from Proposition \ref{whi} that 
for any $L\in H^2(G,Z(H))$, 
$$
\frac{O_4(c,LM_1,M_2)}{O_4(c,M_1,M_2)}=(L,M_2)\in H^4(G,{\bold k}^\times).
$$
Thus, our theory of categorical extensions reproduces the classical theory of group extensions.

\begin{remark}
If $H$ is an abelian group, then it is clear that $O_3(c)$ vanishes, and 
the torsor $T_1$ is canonically trivial ($T_1=H^2(G,H)$). In this case, we have 
$$
O_4(c,M_1,M_2)=(M_1,M_2)\in H^4(G,{\bold k}^\times).
$$
\end{remark}



\subsection{Invertible fiber functors and Tambara-Yamagami categories}
\label{invertible ff}

Recall that a fiber functor on a tensor category $\C$ 
is the same thing as a $\C$-module category  structure on $\Vec$.

Let $G$ be a finite group and let $\C=\Vec_G$ be the tensor category 
of $G$-graded vector spaces.
We will  describe all invertible $\C$-bimodule category structures on $\Vec$.
Let $\phi$ be a 2-cocycle on $G\times G^\op$ and let $\M_\phi$
denote the $\Vec_G$-bimodule category based on $\Vec$ with the action
$(a,\,b)\ot {\bold k} ={\bold k}$, where ${\bold k}$ is a one-dimensional vector space,
and an associativity constraint 
\[
\phi((a_1, a_2),\, (b_1, b_2))\id_{\bold k} : 
((a_1, a_2)\ot(b_1, b_2))\ot {\bold k} \xrightarrow{\sim}
(a_1, a_2)\ot ((b_1, b_2)\ot {\bold k}).
\]
Every $\C$-bimodule category structure on $\Vec$ is equivalent to $\M_\phi$.

Recall the Schur isomorphism, see \cite[2.2.10]{Kar}:
\begin{equation}
\label{Schur}
s: H^2(G\times G,\, {\bold k}^\times) \xrightarrow{\sim} H^2(G,\, {\bold k}^\times) \times H^2(G,\, {\bold k}^\times)
\times (G_{ab}\otimes_\mathbb{Z} G_{ab})^*,
\end{equation}
where $G_{ab} =G/G'$ is the abelianization of $G$. Note that $(G_{ab}\otimes_\mathbb{Z} G_{ab})^*$
is isomorphic to the group of bicharacters on $G$ (or, equivalently, on $G_{ab}$).

Below we will abuse notation and identify cocycles with their cohomology classes.
Let us write 
\begin{equation}
\label{s(phi)}
s(\phi) = (\phi_1,\, \phi_2,\, \phi_{12}).
\end{equation}
Here $\phi_1,\, \phi_2\in H^2(G,\, {\bold k}^\times)$
define the left and right $\C$-module structures on $\M_\phi$ and  $\phi_{12}\in
(G_{ab}\otimes_\mathbb{Z} G_{ab})^*$ defines its $\C$-bimodule structure.


\begin{remark}
\label{phi op}
The category $(\M_\phi)^\op$ is also a $\C$-bimodule category based on $\Vec$.
It is easy to check that
$(\M_\phi)^\op \cong \M_{\tilde{\phi}}$ where 
\[ 
\tilde{\phi}((a,a'),(b,b')) := \phi((a',a), (b',b))^{-1}.
\]
Thus, $\tilde{\phi}_{1}= \phi_{2}^{-1}$, $\tilde{\phi}_{2}= \phi_{1}^{-1}$, and
$\tilde{\phi}_{12}(a,b)= \phi_{12}(b,a),\, a,b\in G_{ab}$.
\end{remark}

Given two $2$-cocycles $\phi, \, \phi'$ on $G\times G$,
the category $\Fun_\C(\M_\phi,\, \M_{\phi'})$ is equivalent, as an abelian category,  to the category
$\Rep_\mu(G)$
of projective representations of $G$ with the Schur multiplier $\mu = \phi_{1}'/\phi_{1}$. This cate\-gory
is acted upon by $\Vec_{G\times G^\op}$ 
via 
\[
((a, b)\ot \pi)(x) = \phi_{12}(x,a) \pi(x) \phi'_{12}(x, b),
\]
where $a,b,x\in G,\, \pi \in  \Rep_\mu(G)$. The associativity constraint isomorphism
between $((a, a')\ot (b, b'))\ot \pi$ and $(a, a')\ot ((b, b')\ot \pi)$ is given by $\phi_{2}(a,b)\phi'_{2}(b', a')$.

\begin{proposition}
\label{description of inv ff}
\begin{enumerate}
\item[(i)] Let $G$ be a finite group and let 
$\omega\in H^3(G,\, {\bold k}^\times)$.
Let $\Vec_{G,\omega}$ be the corresponding pointed fusion category.
If $\Vec$ has a structure of an invertible $\Vec_{G,\omega}$-bimodule category
then $G$ is abelian and $\omega$ is cohomologically trivial.
\item[(ii)] Let $G$ be abelian. Then $\M_\phi$ is an invertible $\Vec_G$-bimodule category
if and only if $\phi_{12}$ is a non-degenerate bicharacter on $G$.
\item[(iii)] The category $\M_\phi$ has order $2$ in $\Pic(\Vec_G)$ if and only if 
$\phi_1 =\phi_2^{-1}$ and $\phi_{12}$ is a symmetric non-degenerate bicharacter.
\end{enumerate}
\end{proposition}
\begin{proof} (i) Since $\Vec_{G,\omega}$ has a fiber functor, $\omega$ must be trivial.
By Proposition~\ref{inv criterion}  the dual of $\Vec_G$ with respect to its module 
category $\Vec$ must be pointed, which forces $G$ to be abelian. 
(ii) The computations done before this Proposition show that $\Fun_\C(\M_\phi,\, \M_{\phi}) \cong \C$
as a $\C$-bimodule category if and only if $\phi_{12}$ is non-degenerate (there are no conditions
on $\phi_1,\, \phi_2$).
(iii) This is equivalent to existence of a $\C$-bimodule equivalence $\M_\phi^\op \cong \M_\phi$,
so we can apply Remark~\ref{phi op}.
\end{proof}

\begin{example}
In \cite{TY} D.~Tambara and  S.~Yamagami classified all 
$\mathbb{Z}/2\mathbb{Z}$-graded fusion categories $\C=\C_+\oplus
\C_-$ in which $\C_+$ is a pointed category, and $\C_-$ has a unique simple object. 
They showed that any such category
is determined, up to a tensor equivalence, by a
finite abelian group $A$, an isomorphism class of a  
non-degenerate symmetric bilinear form
$\chi: A\times A \to {\bold k}^\times$, and a  square root of $|A|$ in ${\bold k}$.
The classification of \cite{TY} uses direct calculations
of associativity constraints as solutions of a system of pentagon equations.

Let us derive this classification from our description of graded categories
in Section~\ref{classification proper} and Proposition~\ref{description of inv ff}.
Let $\C= \C_0 \oplus \C_1$
be a fusion category with $\mathbb{Z}/2\mathbb{Z}$-grading satisfying the above properties.
Its trivial component $\C_0$ is a pointed fusion category. 
By Proposition~\ref{description of inv ff}(i), 
$\C_0\cong \Vec_A$, for some finite abelian group $A$.  
The invertible $\Vec_A$-bimodule category $\C_1$ has order $2$ in $\Pic(\Vec_A)$.
By Proposition~\ref{description of inv ff}, $\C_1\cong \M_\phi$ where $\phi$ is such that
$\phi_1=\phi_2^{-1}$ and $\phi_{12}$ is a non-degenerate symmetric bicharacter
of $A$, cf.\ \eqref{s(phi)}. 

We have $Z: = \mbox{Inv}(\Z(\Vec_A)) = A \oplus A^*$. Let us identify $A$ with $A^*$ using the bicharacter 
$\phi_{12}$; then we have $Z=A\oplus A$, and as a $\mathbb{Z}/2\mathbb{Z}$-module, $Z={\rm Fun}(\mathbb{Z}/2\mathbb{Z},A)$.  
Therefore, by the Shapiro lemma, 
$H^i(\mathbb{Z}/2\mathbb{Z} ,\, Z) = 0$ for $i>0$. Thus, $O_3(c)=0$, and there is no freedom in choosing $M$. 

Furthermore, the associativity constraint obstruction $O_4$ vanishes since  
$H^4(\mathbb{Z}/2\mathbb{Z},{\bold k}^\times)=0$
and hence there are precisely two non-equivalent tensor category structures
on $\C$ corresponding to two elements of the group
$H^3(\mathbb{Z}/2\mathbb{Z},{\bold k}^\times) \cong \mathbb{Z}/2\mathbb{Z}$.

Let $\tau$ be a tensor autoequivalence of $\Vec_A$.
Let $\C_1^\tau$ denote the $\Vec_A$-bimodule category
obtained from $\C_1$ by twisting the action  of $\Vec_A$ by means of $\tau$, i.e.,
by letting the result of action 
of $X\bt Y \in \Vec_A\bt \Vec_A^\rev$ on $M\in \M$ to be  $(\tau(X)\bt \tau(Y))\ot M$.
Clearly, we can replace $\C_1$ by $\C_1^\tau$ without changing the corresponding
extension.

The group of tensor  autoequivalnces of  $\Vec_A$ is isomorphic
to the semi-direct product $H^2(A,\,{\bold k}^\times)\rtimes \Aut(A)$. 
Choosing $\tau$ to be the element corresponding to $(\phi_1^{-1},\, \alpha)$, where $\alpha$ is
any automorphism of $A$, we see that $\phi$ can be chosen in such a way
that $\phi_1=1$ and the choice of $\phi_{12}$ matters only up to 
an automorphism of $A$.

Thus, we obtain the same parameterization as in \cite{TY}.
\end{example}

\subsection{Categories $\C$ graded by a group $G$ of order coprime to ${\rm FPdim}(\C_e)$}

If $|G|$ and $D:={\rm FPdim}(\D)$ are coprime, the classification of extensions of $\D$ by $G$ 
simplifies, as the cohomological obstructions $O_3$ and $O_4$ automatically vanish. Namely, we have the following result. 

\begin{theorem}\label{copr}  Let $\D$ be a fusion category of Frobenius-Perron 
dimension $D$ relatively prime to $|G|$. Then any homomorphism 
$c: G\to \Pic(\D)$ can be upgraded to a $G$-graded fusion category
with trivial component $\C_e=\D$, and such categories 
are parametrized by a torsor $T_{c,M}^3$ over $H^3(G,{\bold k}^\times)$
(up to a grading-preserving equivalence). 
\end{theorem}

\begin{proof}
This follows from Theorem \ref{extcl} and Theorem \ref{divi}. Indeed, 
the order of the group $\pi_2={\rm Inv}(\Z(\D))$ divides $D^2$
(\cite[Proposition 8.15]{ENO1}), so it is relatively prime to $|G|$. 
Thus, $H^i(G,\pi_2)=0$, $i\ge 1$. So $O_3(c)$ vanishes, 
and there is no freedom in choosing $M$. 
Also, by Theorem \ref{divi}, the obstruction $O_4(c,M)$ vanishes.
So the graded category $\C$ exists, and 
the freedom in its construction is just the freedom 
of choosing $\alpha$, which lies in a torsor over 
$H^3(G,{\bold k}^\times)$, as desired. 
\end{proof}

For applications of this Theorem, see \cite{JL}.

\section{Lagrangian subgroups in metric groups and bimodule categories over $\Vec_A$}
\label{Sect10}

Let ${\rm Bimod}_{\rm ab}$ be the category 
whose objects are categories $\Vec_A$ where $A$ is an finite abelian group, 
and morphisms from $\Vec_A$ to $\Vec_B$ are equivalence classes of 
(not necessarily invertible) $(\Vec_B,\Vec_A)$-bimodule categories, with composition of morphisms being 
the tensor product of bimodule categories. The goal of this section is to describe 
this category explicitly. 

First we need to set up some linear algebra, which is well
known, but we work out the details for the reader's convenience.

\subsection{The category of Lagrangian correspondences}

Let us define the category Lag of Lagrangian correspondences. We define the objects of this category
to be metric groups $(E,q)$. Morphisms from $(E_1,q_1)$ to $(E_2,q_2)$ are, by definition, 
formal $\Bbb Z_+$-linear combinations of Lagrangian subgroups in 
$(E_1\oplus E_2, q_1^{-1}\oplus q_2)$. 

The composition of morphisms is defined as follows. Let
\linebreak $L\in {\rm Mor}((E_1,q_1),(E_2,q_2))$,
$M\in {\rm Mor}((E_2,q_2),(E_3,q_3))$ be Lagrangian subgroups. 
Then we define $M\circ L$ to be the set of all pairs 
$(a_1,a_3)\in E_1\oplus E_3$ such that there exists $a_2\in E_2$
for which $(a_1,a_2)\in L$ and 
$(a_2,a_3)\in M$. Also, let $m(M,L)$ be the number of such
$a_2$. Then the composition of morphisms is defined by the
condition that it is biadditive, and 
$$
M\bullet L=m(M,L)M\circ L.
$$  

To validate this definition, we must prove the following Lemma. 

\begin{lemma}
(i) $M\circ L$ is a Lagrangian subgroup of the metric group 
$(E_1\oplus E_3,\,q_1^{-1}\oplus q_3)$.  

(ii) The function $m$ satisfies the 2-cocycle condition, 
$$
m(N,M\circ L)m(M,L)=m(N\circ M,L)m(N,M), 
$$
so that the operation $\bullet$ is associative. 
\end{lemma}   

\begin{proof}
(i) First of all, it is easy to check that $M\circ L$ is an isotropic subgroup. 
Next, $M\circ L$ is the quotient of the intersection of the subgroup $L\oplus M$ with 
the diagonal copy of $E_1\oplus E_2\oplus E_3$ in $E_1\oplus E_2\oplus E_2\oplus E_3$
by the group $N=M\cap L\cap E_2$. It is easy to see that the
image of $L\oplus M$ in $E_2\oplus E_2/E_2^{diag}=E_2$ 
is the orthogonal complement $N^\perp$ of $N$. Thus, the order of the intersection of $L\oplus M$ with 
$E_1\oplus E_2\oplus E_3$ is $|M|\cdot |L|/|N^\perp|$, and hence the order of $M\circ L$ is 
$|M|\cdot |L|/|E_2|=(|E_1|\cdot |E_3|)^{1/2}$, i.e. $M\circ L$ is Lagrangian.  

(ii) This is a straightforward computation.  
\end{proof}

Thus, we have defined the category Lag. Note that the identity
morphism of $(E,q)$ 
in this category is the diagonal subgroup of $E\oplus E$. 

\begin{proposition}\label{groupoi} The groupoid of isomorphisms in Lag is naturally isomorphic
to the groupoid of isometries of metric groups.   
In particular, the group of automorphisms of $(E,q)$ in ${\rm Lag}$ is naturally isomorphic to 
$O(E,q)$. 
\end{proposition}

\begin{proof}
Let $(E,q),(E',q')\in {\rm Lag}$. 
Let $L\subset E\oplus E'$ be Lagrangian under the form $q^{-1}\oplus
q'$. If $L$ defines an isomorphism 
then $L\circ M=\id$ for some Lagrangian $M\subset E'\oplus E$, 
which implies that the intersection $L$ with $E'$ is zero. 
Similarly, the intersection of $L$ with $E$ is zero (because $M\circ
L=\id$). This means that $L$ is the graph of some isomorphism 
of groups $g: E\to E'$, and since $L$ is Lagrangian, 
this isomorphism is an isometry. Conversely, if $g:E\to E'$ is an isometry 
then the graph of $g$ is Lagrangian in $E\oplus E'$. 
It is easy to see that the composition of Lagrangian subgroups 
goes under this identification to the composition of isometries. 
The proposition is proved. 
\end{proof}

\subsection{Subgroups with a skew-symmetric bicharacter in an abe\-lian group}

Let $A$ be a finite abelian group. Denote by $C(A)$ the set of
pairs $(H,\psi)$, where $H\subset A$ is a subgroup, and $\psi$ is a 
skew-symmetric bicharacter of $H$.
Also, for a metric group $(E,q)$, 
let $\mathcal{L}(E,q)$ be the set of Lagrangian subgroups of $E$. 

The following Proposition is a special case of a more general
result proved in \cite{NN}. 

\begin{proposition}\label{bij} There is a natural bijection 
\begin{equation}
\label{Tau}
\tau: C(A)\to \mathcal{L}(A\oplus A^*,q),
\end{equation}
where $q$ is the standard hyperbolic
quadratic form of $A\oplus A^*$ given by $q(a,f)=f(a)$. 
This bijection is given by the formula 
$\tau(H,\psi)=\lbrace{(h,z)|z\in \psi(h)\rbrace}$, 
where $\psi(h)\in H^*=A^*/H_\perp$ is regarded as a coset
of $H_\perp$ in $A^*$.  
\end{proposition}

\begin{proof} It is clear that the subgroup 
$L=\lbrace{(h,z)|z\in \psi(h)\rbrace}\subset A\oplus A^*$
is isotropic. Also, $|L|=|H|\cdot |H_\perp|=|A|$, so
$L$ is Lagrangian. Thus the map $\tau$ is well defined. 
Now we'll prove that $\tau$ is invertible by constructing 
the inverse map. Namely, given a Lagrangian subgroup $L\subset
A\oplus A^*$, 
set $\sigma(L)=(H,\psi)$, where $H$ is the image of $L$ in $A$,
and 
$$
\psi(h_1,h_2):=(h_1',h_2),
$$
where $h_1'$ is any lifting 
of $h_1$ into $L$. 

To prove that $\sigma$ is well defined, 
we need to show that $\psi(h_1,h_2)$ is independent on the choice
of the lifting $h_1'$. In other words, we must show that 
if $v$ is an element of $L\cap H^*$ then for any $h\in H$ we have
$(v,h)=1$. But this holds because $(v,h)=(v,h')$ 
for any lifting $h'$ of $h$ to $L$, and 
$(v,h')=1$ since $v,h'\in L$, and $b_q$ is the standard inner
product on $A\oplus A^*$. 
  
Now we should prove that $(\psi(h),h)=1$, i.e. that $(h',h)=1$,
if $h\in H$ and $h'$ is a lift of $h$ in $L$.
 We have 
$$
(h',h)=q(h)q(h')/q(h'-h),
$$
Now we see that all three factors on the RHS are equal to $1$:
the first one because $h\in A$, the third one because $h'-h\in A^*$, and 
the second one because $h'\in L$ and $L$ is Lagrangian.  

Finally, we should check that $\sigma$ is indeed inverse to
$\tau$. We have $(\sigma\circ \tau)(H,\psi)=(H,\psi')$, where 
$\psi'(h_1,h_2)=(z_1,h_2)$, where $z_1\in \psi(h_1)$. 
Thus $\psi'=\psi$ and we are done (since $\tau$ is a map of
finite sets). 
\end{proof}

\subsection{The structure of the category ${\rm Bimod}_{\rm ab}$}

Now we will define a functor $T$ 
from the category ${\rm Bimod}_{\rm ab}$ to the full subcategory 
${\rm Lag}_{\rm hyp}$ of ${\rm Lag}$, 
whose objects are groups of the form $A\oplus A^*$ 
with the hyperbolic quadratic form $q$. 
Namely, recall that if $G$ is an abelian group, then equivalence classes 
of indecomposable left module categories over $\Vec_G$ are parametrized by the set
$C(G)$ defined in the previous subsection. 
Now, for any indecomposable
$(\Vec_A,\Vec_B)$-bimodule category $\mathcal{M}$, 
regard $\mathcal{M}$ as a $\Vec_{A\oplus B}$-module category 
via $(a,b)\otimes M=a\otimes M\otimes b^{-1}$, and consider its
equivalence class $[\mathcal{M}]\in C(A\oplus B)$.
Set 
$$
T(\mathcal{M}):=\gamma\tau([\mathcal{M}])\in \mathcal{L}(A\oplus A^*\oplus
B\oplus B^*, q_A\oplus q_B^{-1}),
$$
where $\gamma\in {\rm
Aut}(A\oplus A^*\oplus B\oplus B^*)$ is defined by the formula
$\gamma(a,a^*,b,b^*)=(a,a^*,-b,b^*)$ and $\tau$ is defined in \eqref{Tau}. 
Extend $T$ to decomposable module categories by additivity. 

\begin{theorem}\label{bimo1}
The assignment $T$ is a functor, 
i.e. for any $(\Vec_{A_1},\Vec_{A_2})$-bimodule
category $\N$ and $(\Vec_{A_2},\Vec_{A_3})$-bimodule
category $\N'$ one has
$$
T(\mathcal{N}\boxtimes_{\Vec_{A_2}} \mathcal{N'})=T(\mathcal{N})\circ
T(\mathcal{N'}).
$$ 
\end{theorem}

\begin{proof}
Let $A_1,A_2,A_3$ be abelian groups, $(H,\psi)\in C(A_1\oplus A_2)$, 
$(H',\psi')\in C(A_2\oplus A_3)$. We would like to find 
$(H'',\psi'')$ such that 
$$
\gamma\tau(H,\psi)\bullet \gamma\tau(H',\psi')=m\cdot \gamma\tau(H'',\psi''),
$$
and compute the value of $m$. 

By the definition of $\tau$, the subgroup 
$L:=\gamma\tau(H,\psi)\subset A_1\oplus A_1^*\oplus A_2\oplus A_2^*$ 
is the set of all $(a_1,f_1,a_2,f_2)$ such that 
$(a_1,-a_2)\in H$ and $(f_1,f_2)-\widehat{\psi}(a_1,-a_2)\in H_\perp$. 
Similarly, the subgroup 
$L':=\gamma\tau(H',\psi')\subset A_2\oplus A_2^*\oplus A_3\oplus A_3^*$ 
is the set of all $(a_2,f_2,a_3,f_3)$ such that 
$(a_2,-a_3)\in H'$ and $(f_2,f_3)-\widehat{\psi'}(a_2,-a_3)\in H_\perp'$. 

Now, $L\circ L'=m\cdot L''$, where 
$L''$ is the set of all $(a_1,f_1,a_3,f_3)$ 
such that there exist $a_2,f_2$ with 
$(a_1,-a_2)\in H$, $(f_1,f_2)-\widehat{\psi}(a_1,-a_2)\in H_\perp$, 
$(a_2,-a_3)\in H'$, and $(f_2,f_3)-\widehat{\psi'}(a_2,-a_3)\in H_\perp'$. 
Moreover, $m$ is the number of pairs $(a_2,f_2)$ satisfying these conditions. 

Let $L''=\gamma\tau(H'',\psi'')$. 
It can be checked directly from the above conditions that $H'',\psi''$ 
are the same as in Proposition \ref{tenspro1}. Moreover, 
the number $m$ is the number of pairs $(a_2,f_2)$, so we have 
$$
m=|{\rm Ker}((H\cap H')^\perp\to A_1\oplus A_3)|\cdot |H_\perp\cap H_\perp'|
$$
(the first factor represents the number of choices of $a_2$, and the second one 
stands for the number of choices of $f_2$). Thus, 
$$
m=\frac{|(H\cap H')^\perp|}{|H''|}\cdot |H_\perp\cap H_\perp'|.
$$
But $H''=H\circ H'/(H\cap H')$, so we get 
$$
m=
\frac{|(H\cap H')^\perp|\cdot |H\cap H'|}{|H\circ H'|}\cdot |H_\perp\cap H_\perp'|,
$$
which coincides with the second formula for $m$ in Proposition \ref{tenspro1}. 
The theorem is proved. 
\end{proof} 

\begin{corollary}\label{equi}
$T$ is an equivalence of categories
${\rm Bimod}_{\rm ab}\to {\rm Lag}_{\rm hyp}$. 
\end{corollary}

\begin{proof}
This follows from Theorem \ref{bimo1} and Proposition \ref{bij}. 
\end{proof}

\begin{remark}
Note that we have obtained another (direct) proof of 
Corollary \ref{ortho}, which does not use Theorem \ref{breq}.
(Namely, Corollary \ref{ortho} follows from Theorem \ref{bimo1} and 
Proposition \ref{groupoi}.) One can check that the two proofs 
provide the same isomorphism 
$$
\Pic(\Vec_A)\cong O(A\oplus A^*).
$$ 
\end{remark} 

\begin{remark}
The isomorphism of Corollary \ref{ortho} can be understood 
in topological terms as follows. Recall that
$$
\pi_2(B\uuPic(\Vec_A))=A\oplus A^*, 
\pi_3(B\uuPic(\Vec_A))={\bold k}^\times,
$$
and by Proposition \ref{whit}, the Whitehead half-square
$\pi_2\to \pi_3$ is the hyperbolic quadratic form $q$ on
$A\oplus A^*$. Thus, the action of $\pi_1$ on $\pi_2$ must
preserve this form, i.e. we have a homomorphism 
\[
\eta: \Pic(\Vec_A)\to O(A\oplus A^*).
\]
One can show that this $\eta$ coincides with the isomorphism of Corollary \ref{ortho}, i.e. 
with the restriction of $T$ to invertible $\Vec_A$-bimodule
categories. 
\end{remark}
 

\subsection{The number of simple objects in an invertible bimodule category 
over $\Vec_A$}

\begin{proposition}\label{numbersim}
Let $g\in O(A\oplus A^*)$, and $\C_g$ be the corresponding invertible bimodule category. 
Let $P$ be the projection $A\oplus A^*\to A$, and $K$ be the kernel of $P\circ g|_{A^*}$. Then the number 
of isomorphism classes of simple objects of $\C_g$ equals $|K|$. 
\end{proposition}

\begin{proof} The Lagrangian subspace $L$ in $A\oplus A^*\oplus A\oplus A^*$ correpsonding to 
$g$ is the set of $(a,f,g(a,f))$, where $a\in A$, $f\in A^*$. 
The corresponding subgroup $H$ in $A\oplus A$ (such that $\C_g=\M(H,\psi)$ for some $\psi$) is the projection 
of $L$ to $A\oplus A$. Thus, $H$ projects onto $A$ (via the first coordinate), and the kernel is 
the set of possible first coordinates of $g(a,f)$, $f\in A^*$, i.e. the image of 
$P\circ g|_{A^*}$. Thus, $|H|=|A|/|K|$, and we re done. 
\end{proof}

\subsection{Integral $\Vec_A$-bimodule categories}

%

Recall that for an integral fusion category $\C$ we defined in Section~\ref{int bim} the  categorical 2-subgroup 
$\uuPic_+(\C)\subset \uuPic(\C)$ consisting of integral invertible $\C$-bimodule categories.

\begin{proposition}\label{ortho1}
If $A$ is an abelian group  then
$\Pic_+(\Vec_A)=SO(A\oplus A^*)$.   
\end{proposition}
\begin{proof}
This follows easily from Corollary \ref{ortho}. 
Namely, by Proposition \ref{relp}, 
we may assume without loss of generality that 
$A$ is a $p$-group for some prime $p$. 
In this case, the dimensions of simple objects in a bimodule category 
are either integer or half-integer powers of $p$.

Let $\C=\Vec_A$. For $g\in \Pic(\C)=O(A\oplus A^*)$, 
let $P: A\oplus A^*\to A$ be the projection, 
$K$ be the kernel of $P\circ g|_{A^*}$, and 
$I$ be the image of $P\circ g|_{A^*}$.
Then by Proposition \ref{numbersim},  
the dimensions of simple objects 
of $\C_g$ are $(|A|/|K|)^{1/2}=|I|^{1/2}$
(as $I=A^*/K$). This is an integer  
if and only if $|I|=p^n$, where 
$n$ is even, i.e. if and only if $d(A^*,g(A^*))=1$, 
which implies the statement by Proposition \ref{dete1}. 
\end{proof}

\section{Appendix: Group extensions as $G$-graded fusion categories}

\vskip .05in
\centerline{\bf by Ehud Meir}
\vskip .05in

\subsection{Introduction}
In this appendix we will discuss a special class of extensions of a fusion category by a finite group. Let $\Ga$ be a finite group which fits
into a short exact sequence of groups $1\rightarrow N \rightarrow \Ga\rightarrow G\rightarrow 1$. Suppose that we have a 3-cocycle
$\omega\in H^3(\Ga,{\bold k}^\times)$, and the corresponding
fusion category $\C = Vec_{\Ga,\omega}$. This category has a
subcategory $\D = Vec_{N,\omega}$ 
(where by $\omega$ we also mean the restriction of $\omega$ to $N$), and $\C$ is a $G$-extension of $\D$. It is possible to
classify directly extensions of $\D$ by $G$ which are also pointed; one needs to give an extension $\Ga$ of $G$ by $N$, and to give an
extension of the cocycle $\omega$ on $N$ to a cocycle on $\Ga$. We will explain here why this solution and the solution given by the theory
of $G$-extensions developed in the paper are equivalent. We will do so in the following way: we will take a parameterization $(c,M,\alpha)$
of a pointed $G$-graded extension of $\D$, as in Theorem \ref{extcl}, and we will explain why this parameterization is equivalent to giving an
extension $\Ga$ of $G$ by $N$ and an extension of $\omega$ to a cocycle on $\Ga$.
In order to do so we first study the groups $\Aaut$ and $\Oout$ of tensor autoequivalences and outer tensor autoequivalences of $\D$,
respectively, since these two groups will play a decisive role in
understanding the triple $(c,M,\alpha)$. We then describe the
group $T={\rm Inv}
(\mathcal{Z}(\D))$ of invertible objects of the center, in order to understand the obstruction $O_3(c)$ which lies in $H^3(G,T)$. Using
this, we will explain how to ``translate'' a triple $(c,M,\alpha)$ to an extension $\Ga$ of $G$ by $N$ together with a 3-cocycle on $\Ga$
which is an extension of $\omega$.
If $H$ is any finite group and $\omega\in H^3(H,{\bold k}^\times)$, we will denote the simple objects of $Vec_{H,\omega}$ by $\{V_h\}_{h\in H}$.

\subsection{The groups $\Aaut$ and $\Oout$}
Let $\Phi\in \Aaut$. By considering the way in which $\Phi$ acts on simple objects of $\D$ (which correspond to elements of $N$) we get an automorphism $\phi$ of $N$.
The additional data which we need in order to turn $\Phi$ into a
tensor autoequivalence of $\D$ is an isomorphism, for every
$a,b\in N$, 
$$
\Phi(V_{\phi^{-1}(a)})\tensor\Phi(V_{\phi^{-1}(b)})\rightarrow
\Phi(V_{\phi^{-1}(a)}\tensor V_{\phi^{-1}(b)}).
$$
This isomorphism is
given by a scalar which we will denote $\gamma_{\Phi}(a,b)$. It
is easy to see that the equation that $\gamma_{\Phi}$ should
satisfy is 
$$
\del
\gamma_{\Phi}(a,b,c) =
\omega(\phi^{-1}(a),\phi^{-1}(b),\phi^{-1}(c))\omega^{-1}(a,b,c)
= \phi\cdot\omega/\omega.
$$
 In other words, in order
for $\phi$ to furnish a tensor autoequivalence, 
it is necessary and sufficient that 
$\phi\cdot\omega = \omega$ in $H^3(N,{\bold k}^\times)$. We will
denote the subgroup of all such automorphisms by ${\rm Aut}(N,\omega)$. We thus have an onto map $\pi:\Aaut\mapsonto {\rm Aut}(N,\omega)$. A direct
calculation shows that the kernel of this map is $H^2(N,{\bold
k}^\times)$. We thus have a short exact sequence \begin{equation}1\rightarrow H^2(N,{\bold k}^\times)\rightarrow
\Aaut\rightarrow {\rm Aut}(N,\omega)\rightarrow 1\end{equation} Notice
that in the case $\omega\neq 1$ this sequence does not necessarily split.

For every $n\in N$ we have an autoequivalence $C_n$ of
conjugation by $V_n$. This is the autoequivalence which sends the object $V_a$ to
$(V_n\tensor V_a)\tensor V_{n^{-1}}$, and the tensor structure is defined in the obvious way.
Notice that in particular this gives us a canonical 2-cochain $t_n$ such that $\del t_n = \omega(n^{-1}?n)/
\omega(?)$. As expected, this defines a homomorphism of groups $Con: N\rightarrow \Aaut$.  The image of $Con$ is a normal subgroup, and we
will denote the quotient of $\Aaut$ by $im(Con)$ by $\Oout$.

\subsection{The group ${\rm Inv}(\mathcal{Z}(\D))$}
We will now describe the group $T = {\rm Inv}(\mathcal{Z}(\D))$. This is a special case of Theorem 5.2 of \cite{GN}, where the group of invertible objects of a general group-theoretical category was described. An invertible object of $\mathcal{Z}(\D)$ would be an invertible object of $\D$
(that is $V_z$, for some $z\in N$), such that for every 
$a\in N$, we have an isomorphism $V_z\otimes V_a\rightarrow
V_a\tensor V_z$ (and thus, $z\in Z(N)$, the center of $N$). The
element $z$ should satisfy however another condition. The map
$V_z\tensor V_a\rightarrow V_a\tensor V_z$ (if it exists) is just
multiplication by a scalar. Denote this scalar by $r(a)$. Then a
direct calculation shows that the set of scalars $r(a)$ will define on $V_z$ a structure of a
central object if and only if the equation 
$$
r(a)r(b)r(ab)^{-1} = \omega(z,a,b)\omega(a,b,z)\omega^{-1}(a,z,b)
$$ 
holds. We have the following
fact, which can be easily proved directly:
\begin{fact} For every $z\in Z(N)$, the function \begin{equation}c_z(a,b)=\omega(z,a,b)\omega(a,b,z)\omega^{-1}(a,z,b)\end{equation} is a 2-cocycle of $N$ with values
in ${\bold k}^\times$. The conjugation map $Con:N\rightarrow \Aaut$ maps $Z(N)$ to $H^2(N,{\bold k}^\times)$ via $z\mapsto c_z$.\end{fact}
So conjugation by $V_z$ where $z\in Z(N)$ is not necessarily the trivial autoequivalence of $\D$. It is the autoequiovalence given by the 2-cocycle
$c_z$. An object $V_z$, for $z\in Z(N)$ would have a structure of a central object if and only if conjugation by $V_z$ is trivial, that is,
if and only if $c_z$ is the trivial cocycle. We will denote the kernel of $z\mapsto c_z$ by $Z(N,\omega)$ (so this is also the kernel of
$N\rightarrow {\rm Aut}_{\tensor}(\D)$). Thus, we have an onto map $T\mapsonto Z(N,\omega)$. What would be its kernel? To give $V_1$ a structure of
an object of $\mathcal{Z}(\D)$ is the same thing as to give a function $r:N\rightarrow {\bold k}^\times$ which satisfies $r(a)r(b)=r(ab)$, i.e. a 1-
cocycle. Since 1-coboundaries are trivial, we can describe $T$ as an extension of the form \begin{equation}1\rightarrow H^1(N,{\bold k}^\times)\rightarrow T\rightarrow
Z(N,\omega)\rightarrow 1.\end{equation} In case $\omega\neq 1$,
this sequence does not necessarily split.

The group $\Aaut$ acts naturally on $T$. As objects of $T$ are central in $\D$, it is easy to see that inner automorphisms would act trivially
on $T$. We therefore have an induced action of $\Oout$ on $T$.

\subsection{The homomorphism $c$}
If $\C$ is a pointed extension of $\D$, it is easy to see that
for every $g\in G$, the bimodule category $\D_g$ is a
quasi-trivial bimodule (as defined in Section \ref{quasitr}). 
It follows that there are autoequivalences $\Phi(g)\in \Aaut$ for $g\in G$, such  that $\D_g\cong \D^{\Phi(g)}$, that is, $\D_g$ is the same category as $\D$, but the action of $\D\rtensor \D^{op}$ is given by \begin{equation}(V_a\rtensor V_c) \tensor V_b = (V_a\tensor V_b)\tensor\Phi(g)(V_c).\end{equation} It can easily be seen that the bimodule category $\D_g$
defines the autoequivalence $\Phi(g)$ only up to conjugation by an invertible object of $\D$. So $\Oout$ is a subgroup of ${\rm BrPic}(\D)$, and the image of $c:G\rightarrow {\rm BrPic}(\D)$ lies inside $\Oout$. For each $g\in G$, choose an autoequivalence $\Phi(g)$ of $\D$, whose image in $\Oout$ is $c(g)$. We thus have an isomorphism of functors \begin{equation}\label{multaut}p_{g,h}:\Phi(g)\Phi(h)\stackrel{\cong}{\rightarrow}C_{n(g,h)}\Phi(gh),\end{equation} where $n(g,h)\in N$. Notice that we need to make a choice here, as $n(g,h)$ is defined only up to a
coset of $Z(N,\omega)$ in $N$. We can think of the morphism $p_{g,h}$ as a 1-cochain which satisfies a certain boundary condition. We also make a choice in choosing the $p_{g,h}$'s. As explained above, we will think of $\Phi(g)$ as an automorphism $\phi(g)$ of $N$, together with a 2-cochain $\gamma_g$ on $N$ which satisfies \begin{equation}\label{defgamma}\del\gamma_g = \phi(g)\cdot\omega/\omega.\end{equation} Equation (\ref{multaut}) simply means
that we have equality of automorphisms of $N$, $\phi(g)\phi(h)=c_{n(g,h)}\phi(gh)$, where $c_n$ means the automorphism of conjugation by $n$,
and also that the 2-cocycle \begin{equation}U_{g,h}= \gamma_g (\phi(g)\cdot\gamma_h) t_{n(g,h)}^{-1}(c_n\cdot\gamma_{gh}^{-1})\end{equation} where $t_n$ was described above, is trivial, and equals to $\del p_{g,h}$ (this is the boundary
condition that $p_{g,h}$ should satisfy in order to be an isomorphism between the functors described above).

\subsection{The first obstruction}
We will explain now what the first obstruction $O_3(c)$ looks like in our context. Recall that $O_3(c)$ is an element of $H^3(G,T)$. Assume
that $g,h,k$ are elements of $G$. Let us describe $O_3(c)(g,h,k)\in T$. In order to do so we need to choose equivalences of
$\D$-bimodule categories $\D_g\rtensor_{\D}\D_h\cong \D_{gh}$ for every $g,h\in G$. By the universal property of tensor product of bimodule
categories, this is the same as to give a balanced $\D$-bimodule functor $F_{g,h}:\D_g\rtensor\D_h\rightarrow \D_{gh}$ for every $g,h\in G$,
such that the universal property from Definition 3.3 holds. We choose \begin{equation}F_{g,h}(V_a\rtensor V_b) = (V_a\otimes\Phi(g)(V_b))\otimes V_{n
(g,h)}.\end{equation} The isomorphism $p_{g,h}$ of autoequivalences of $\D$ given by Equation (\ref{multaut}) equips the functor $F_{g,h}$ with a structure of a
balanced $\D$-bimodule functor. The idea now is that we have two functors from $\D_g\rtensor\D_h\rtensor\D_k$ into $\D_{ghk}$, namely $F_
{g,hk}F_{h,k}$ and $F_{gh,k}F_{g,h}$. As both functors can be used in order to identify $\D_{ghk}$ with $\D_g\rtensor_{\D}\D_h\rtensor_{\D}
\D_k$, there is an equivalence of $\D$-bimodule categories $y_{g,h,k}:\D_{ghk}\rightarrow \D_{ghk}$ such that $F_{g,hk}F_{h,k}\cong y_{g,h,k}
F_{gh,k}F_{g,h}$. As $\D$-bimodule equivalences of $\D_{ghk}$ correspond to elements of $T$ as explained in Section \ref{obsten}, this $y_{g,h,k}$
corresponds to $O_3(c)(g,h,k)$.
Using these considerations, a more explicit description of $O_3(c)$ can be given in the following way: for $g,h,k\in G$, the isomorphisms of
functors given in Equation (\ref{multaut}) give us the following isomorphism of functors:
\begin{equation}C_{n(g,h)n(gh,k)n(g,hk)^{-1}\phi(g)(n(h,k)^{-1})}\cong \\ \end{equation}
\begin{equation*}C_{n(g,h)}C_{n(gh,k)}C^{-1}_{n(g,hk)}\Phi(g)C^{-1}_{n(h,k)}\Phi(g)^{-1}\\ \end{equation*}
\begin{equation*}\cong \Phi(g)\Phi(h)\Phi(gh)^{-1}\Phi(gh)\Phi(k)\Phi(ghk)^{-1}\Phi(ghk)\Phi(hk)^{-1}\Phi(g)^{-1}\\\end{equation*}
\begin{equation*}\Phi(g)\Phi(hk)\Phi(k)^{-1}\Phi(h)^{-1}\Phi(g)^{-1}\cong Id.\end{equation*}
But to give an isomorphism of functors $C_n\cong Id$, is the same
thing as to give a structure of a central object on $V_n$. This
(invertible) central object $V_{n(g,h)n(gh,k)n(g,hk)^{-1}\phi(g)(n(h,k)^{-1})}$ would be $O_3(c)(g,h,k)$. Notice that choosing different isomorphisms $p_{g,h}$ or different coset representatives $n(g,h)$ would change $O_3(c)$ only by a coboundary, and thus will give an equivalent cocycle.
We will be interested in the case in which $O_3(c)$ vanishes in $H^3(G,T)$. In case it vanishes, we will call an element $\rho\in C^2(G,T)$
which satisfies $\del \rho = O_3(c)$ a \italic{solution} for $O_3(c)$ (and use similar terminology for other obstructions). The choice of a
solution in this case therefore corresponds to the choice of a
system of products. In our case this is equivalent to choosing
the elements $n(g,h)$ and the morphisms $p_{g,h}$ in such a way that the obstruction we get is the trivial 3-cocycle. This means that the equation $n(g,h)n(gh,k) = \phi(g)(n(h,k))n(g,hk)$ holds in $N$ (and not only
up to a coset of $Z(N,\omega)$), and also that the functions
$p_{g,h}$ satisfy a certain boundary condition which we will consider
later.

\subsection{Vanishing of the first obstruction, the Eilenberg - Mac Lane obstruction, and the choice of a solution}
The description of $T$ as an extension of $Z(N,\omega)$ by $H^1(N,{\bold k}^\times)$ will help us understand the vanishing of $O_3(c)$ in two steps.
Assume first that we know that the image of $O_3(c)$ in $H^3(G,Z(N,\omega))$ (which will be denoted by $\overline{O_3(c)}$) vanishes. This means
that we can change the elements $n(g,h)$ by elements of $Z(N,\omega)$ (that is, to take different coset representatives) in such a way that
the equation \begin{equation}\label{associativity}n(g,h)n(gh,k)=\phi(g)(n(h,k))n(g,hk)\end{equation} holds in $N$. A solution to $\overline{O_3(c)}$ in $H^3(G,Z
(N,\omega))$  would therefore be a choice of coset
representatives $n(g,h)$ which satisfy 
Equation (\ref{associativity}). This would give us a
group extension \begin{equation}1\rightarrow N\rightarrow \Ga\rightarrow G\rightarrow 1\end{equation} We will think of elements of $\Ga$ as products of the form
$n\bar{g}$ where $n\in N$ and $g\in G$. The product of two such
elements would be $n\bar{g}m\bar{h} =
n\phi(g)(m)n(g,h)\overline{gh}$. Equation (\ref{associativity})
is thus equivalent to the associativity of $\Ga$. It can be
checked that the image of $O_3(c)$ in $H^3(G,Z(N))$ coincides
with the Eilenberg-Mac Lane obstruction for the existence of a
group extension of $G$ by $N$ with the given ``outer'' action
$\bar{c}:G\rightarrow \Oout\rightarrow {\rm Out}(N)$. See \cite{Mac} for a description of this obstruction. Suppose that we have
chosen a solution $\rho$ for $\overline{O_3(c)}$ (and therefore we get a group extension $\Ga$ of $G$ by $N$). Lift $\rho$ to a 2-cochain
$\tilde{\rho}$ of $G$ with values in $T$. The cocycle $O_3(c) \del\tilde{\rho}^{-1}$ has all its values in the subgroup $H^1(N,{\bold k}^\times)$. It
is easy to see that the class of this cocycle in $H^3(G,H^1(N,{\bold k}^\times))$ is well defined and does not depend on the choice of the particular lifting, but only
on the choice of the solution $\rho$. We will denote this cocycle by $\widehat{O_3(c)}_{\rho}$. It is easy to see that the vanishing of $O_3(c)$
is equivalent to the fact that $\overline{O_3(c)}$ vanishes, and that we can find for it a solution $\rho$ such that $\widehat{O_3(c)}_{\rho}$
vanishes as well. A solution for $O_3(c)$ will then be of the form $\tilde{\rho} \mu$, where $\mu$ is a solution for $\widehat{O_3(c)}_{\rho}$.
So the situation we will consider from now on is the following: we have a group extension $\Ga$ of $G$ by $N$, and we also have $\widehat
{O_3(c)}_{\rho}$, the ``remainder'' of the first obstruction $O_3(c)$.
We will now describe the second obstruction $O_4(c,M)$, and see how all this data corresponds to 
data from the
spectral sequence of the group extension.

\subsection{The second obstruction}
Let us now describe the second obstruction $O_4(c,M)$ (we assume that we have a solution $\mu$ for $\widehat{O_3(c)}_{\rho}$). We ``almost have''
the extension $\bar{\omega}$ of $\omega$ to $\Ga$ in the following sense: if we knew $\bar{\omega}(\bar{g},\bar{h},\bar
{k})$ for every $g,h,k\in G$, we would have known $\bar{\omega}(a,b,c)$ for any $a,b,c\in \Ga$. This is because the system of products
enables us to express any value of $\bar{\omega}$ solely in terms
of the values $\bar{\omega}(\bar{g},\bar{h},\bar{k})$ for $g,h,k\in G$. So choose such
values arbitrarly, for example, $\bar{\omega}(\bar{g},\bar{h},\bar{k})=1$ for every $g,h,k\in G$. Now consider the hexagon diagram for $\bar
{g},\bar{h},\bar{k},\bar{l}$, where $g,h,k,l\in G$. It will be commutative up to a scalar, which will be $O_4(c,M)(g,h,k,l)$. A choice
of different arbitrary values would give us a cohomologous cocycle. A solution for $O_4(c,M)$ means a collection of values $\bar{\omega}
(\bar{g},\bar{h},\bar{k})$ which will make $\bar{\omega}$ a 3-cocycle on $\Ga$. By choosing a different solution, we will get another
extension of $\omega$ to $\Ga$, which differs by a pullback to
$\Ga$ of a class $\alpha\in H^3(G,{\bold k}^\times)$.
In the context of the data $(c,M,\alpha)$, we assume that we have one fixed solution $\eta$ for $O_4(c,M)$, and that we take the solution
$\eta\alpha$, where $\alpha\in H^3(G,{\bold k}^\times)$.

\subsection{The LHS spectral sequence}
We have already seen that the data $(c,M,\alpha)$ yields an extension $\Ga$ of $G$ by $N$. We will now explain how it determines the
extension $\bar{\omega}$ of $\omega$ from $N$ to $\Ga$. In order to do so we will use the Lyndon-Hochshild-Serre (abbreviated LHS) spectral
sequence \begin{equation}E_2^{p,q}=H^p(G,H^q(N,{\bold
k}^\times))\converges E_{\infty}^{p,q}=H^{p+q}(\Ga,{\bold
k}^\times)\end{equation} A general discussion of this spectral sequence can be found in \cite{Rotman} and in \cite{Mac}. We will use this spectral sequence to understand how the vanishing of the obstructions and the mere existence of $c:G\rightarrow \Oout$ imply that $\omega$ can be
extended to a 3-cocycle on $\Ga$, and how the choices of $c$, $M$ and $\alpha$ give us a specific extension of $\omega$ to $\Ga$.
The idea is the following: we consider $\omega$ as an element of $E_2^{0,3}=H^0(G,H^3(N,{\bold k}^\times))$. Using the theory of spectral sequences, we
know that $\omega$ is extendable to $\Ga$ if and only if $d_2(\omega)=0$ in $E_2^{2,2}$, $d_3(\omega)=0$ in $E_3^{3,1}$
and $d_4(\omega)=0$ in $E_4^{4,0}$. In case this is true, the theory of spectral sequences also gives us all possible extensions of
$\omega$ to $\Ga$. They are parameterized in the following way: if $d_2(\omega)=0$, this means that a certain equation has a ``solution'' (we
will soon see this explicitly). We need to choose a solution $\gamma$, and then $\omega$ and $\gamma$ together will define an element $i^
{\gamma,\omega}\in E_2^{3,1}=H^3(G,H^1(N,{\bold k}^\times))$. This element $i^{\gamma,\omega}$ will be in the kernel of $d_2$, and the image of
$i^{\gamma,\omega}$ in $E_3^{3,1}$ would be $d_3(\omega)$ (recall that $E_3$ is the cohomology of $(E_2,d_2)$). The fact that $d_3(\omega)=0$
means that for some of the solutions $\gamma$, the cocycle $i^{\gamma,\omega}$ would be trivial. We need to choose only such $\gamma$'s.
Again, the fact that $i^{\gamma,\omega}=0$ means that some equation has a solution, and we need once again to choose such a solution,
which we will denote by $p$. Exactly like at the previous step, $\omega$, $\gamma$ and $p$ will define an element $j^{\omega,\gamma,p}\in
E_2^{4,0}=H^4(G,H^0(N,{\bold k}^\times))=H^4(G,{\bold k}^\times)$. The cohomology class $j^{\omega,\gamma,p}$ will obviously be in the kernel of $d_2$ and $d_3$, as
they are trivial on $E_2^{4,0}$ and $E_3^{4,0}$ respectively. The image of $j^{\omega,\gamma,p}$ in $E_4^{4,0}$ would be exactly
$d_4(\omega)$. The fact that $d_4(\omega)=0$ is equivalent to the fact that we can choose $\gamma$ and $p$ such that $j^{\omega,\gamma,p}=0$ in
$H^4(G,{\bold k}^\times)$, and we will choose only such $\gamma$'s and $p$'s. The construction of $j^{\omega,\gamma,p}$ gives it as a cocycle rather than
just as a cohomology class. Therefore we also need a 3-cochain $\beta\in C^3(G,{\bold k}^\times)$ which satisfies $\del\beta = j^{\omega,\gamma,p}$,
that is, we need a solution to this equation. The tuple $(\gamma,p,\beta)$ will give us, by the theory of spectral sequences, the
desired extension of $\omega$ to $\Ga$. We will now explain the connection between the tuple $(\gamma,p,\beta)$ and the data $(c,M,\alpha)$.
We do so by considering the different pages of the spectral sequence.

\subsubsection{The first differential in page $E_2$}
Let us describe $d_2(\omega)$. The cocycle $\omega$ is $G$-invariant, and therefore for every $g\in G$ we can find a 2-cochain $\gamma_g$
such that $\del\gamma_g = \phi(g)\cdot\omega/\omega$. Let $g,h\in G$. Since $\phi(g)\phi(h) = c_{n(g,h)}\phi(gh)$, we have a two
cocycle on $N$ with values in ${\bold k}^\times$, \begin{equation}\label{firstdiff}U_{g,h}^{\gamma}=\frac{\gamma_g g\cdot \gamma_h}{t_{n(g,h)}c_n\cdot\gamma_{gh}}\end{equation} The
function which takes $(g,h)$, for $g,h\in G$, to the cocycle $U^{\gamma}_{g,h}$ is a 2-cocycle of $G$ with values in $H^2(N,{\bold k}^\times)$.
Different choices of $\gamma_g$'s would give us cohomologous cocycles. The cocycle $U_{g,h}^{\gamma}$ is trivial if and only if there is a choice of
$\gamma_g$'s for which $U_{g,h}^{\gamma}$ would be a coboundary
for every $g,h\in G$. A direct calculation shows that $d_2(\omega)= U_{g,h}^{\gamma}$. We claim that the existence of $\Phi$ implies that $d_2(\omega)$ is trivial. This is because we can choose the $\gamma_g$'s we
have in the definition of $\Phi$, in equation (\ref{defgamma}), and for this choice we know that $U^{\gamma}_{g,h} = \del p_{g,h}$. So the ``equation'' we have here is $U_{g,h}^{\gamma}=1$ in $H^2(N,{\bold k}^\times)$, and the solution $\gamma$ is given by $\Phi$, which comes from the homomorphism $c$.

This $\gamma$ is the first part of the data needed in order to
define the extension of $\omega$.

\subsubsection{The second differential in page $E_3$}

We consider now the cocycle $i^{\gamma,\omega}\in H^3(G,H^1(N,{\bold k}^\times))$. In order to construct $i^{\gamma,\omega}$ we need to choose 1-cochains $p_{g,h}$ for every $g,h\in G$ in such a way that the equation $\del p_{g,h} = U^{\gamma}_{g,h}$ holds. We have such 1-cochains given in Equation (\ref{multaut}). If we take the 1-cochains from Equation (\ref{multaut}) and compute $i^{\gamma,\omega}$, we get $i^{\gamma,\omega}= \widehat{O_3(c)}_{\rho}$. So the vanishing of the first obstruction implies also that $d_3(\omega)=0$, and the second part we need in order to define the extension of $\omega$ is the collection of isomorphisms $p_{g,h}$ (which
comes from the system of products $M$). Again, $p=\{p_{g,h}\}$ is a solution to an equation which says that $\widehat{O_3(c)}_{\rho}$ is trivial (recall that $O_3(c)$ was defined using the $p_{g,h}$'s).

\subsubsection{The third differential in page $E_4$}
Finally, consider the cocycle $j^{\omega,\gamma,p}$. A direct calculation shows that this is exactly $O_4(c,M)$. So the vanishing of the
second obstruction implies that $d_4(\omega)=0$. The last choice
we need to make is to choose a 3-cochain $\beta\in C^3(G,{\bold k}^\times)$ such that
$\del\beta = j^{\omega,\gamma,p}$. But the data $(c,M,\alpha)$ determines such a solution. The solution will be $\beta=\eta\alpha$, where
$\eta$ is the fixed solution for $O_4(c,M)$ we assumed we have.

\bibliographystyle{ams-alpha}

\end{document}